\documentclass[12pt]{book}
 \usepackage{amsfonts,graphics,amsmath,amsthm,amsfonts,amscd,amssymb,amsmath,latexsym,multicol, 
 mathrsfs}
\usepackage[matrix, arrow]{xy}
\usepackage{epsfig,url}
\usepackage{flafter}
\usepackage{fancyhdr}

\DeclareMathOperator{\Supp}{Supp}
\DeclareMathOperator{\Sym}{Sym}

\DeclareMathOperator{\im}{im}

\DeclareMathOperator{\Proj}{Proj}
\DeclareMathOperator{\Spec}{Spec}

\DeclareMathOperator{\Hom}{Hom}

\DeclareMathOperator{\Ext}{Ext}
\DeclareMathOperator{\pd}{pd}
\DeclareMathOperator{\depth}{depth}
\DeclareMathOperator{\hgt}{ht}
\DeclareMathOperator{\Tor}{Tor}
\DeclareMathOperator{\Hilb}{Hilb}
\DeclareMathOperator{\Quot}{Quot}
\DeclareMathOperator{\Grass}{Grass}

\newcommand{\xExt}{\mathscr{E}xt}
\newcommand{\xHom}{\mathscr{H}om}
\newcommand{\cHilb}{\mathcal{H}ilb}
\newcommand{\cQuot}{\mathcal{Q}uot}
 \newcommand{\x}{\mathscr}
 \newcommand{\C}{\mathbb C}
 \newcommand{\N}{\mathbb N}
 \newcommand{\PP}{\mathbb P}
 \newcommand{\A}{\mathbb A}
 \newcommand{\Q}{\mathbb Q}
 
 \newcommand{\Z}{\mathbb Z}
 
\makeatletter
\renewcommand\section{\@startsection{section}{1}{\z@}%
                                  {-3.5ex \@plus -1ex \@minus -.2ex}%
                                  {2.3ex \@plus.2ex}%
                                  {\tt\large\bfseries}}
\makeatother

 \numberwithin{equation}{subsection}
 \numberwithin{footnote}{subsection}

 \newtheorem{cor}[subsection]{Corollary}
 \newtheorem{lem}[subsection]{Lemma}
 
 \newtheorem{thm}[subsection]{Theorem}

{
    \newtheoremstyle{upright}%
        {8pt plus2pt minus4pt}%
        {8pt plus2pt minus4pt}%
        {\upshape}%
        {}%
        {\bfseries\scshape}%
        {}%
        {1em}%
        {}%
\theoremstyle{upright}        
 \newtheorem{defn}[subsection]{Definition}
 \newtheorem{defn-rem}[subsection]{Definition-Remark}
 \newtheorem{exa}[subsection]{Example}
 \newtheorem{constr}[subsection]{Construction}
 
 \newtheorem{rem}[subsection]{Remark}
  \newtheorem{rem-thm}[subsection]{Remark-Theorem}

}

 \newcommand{\ke}[1]{$\acute{\mbox{e}}$}
 \newcommand{\ku}[1]{$\acute{\mbox{u}$}}
 \newcommand{\kl}[1]{$\acute{\mbox{l}}$}
 \newcommand{\kh}[1]{$\acute{\mbox{h}}$}
 \newcommand{\kr}[1]{$\acute{\mbox{r}}$}
 \newcommand{\kx}[1]{$\acute{\mbox{x}}$}
 \newcommand{\ki}[1]{${\^\i}$}

\title{\vspace{-5cm}\tt\textbf{Topics in algebraic geometry\\ \vspace{0.5cm}
{\small Lecture notes of an advanced graduate course}}}

\author{\tt Caucher Birkar\\  {\small (c.birkar@dpmms.cam.ac.uk)}}
\date{\today}
\begin{document}
\maketitle

\begin{center}
\bf Dedicated to Grothendieck
\end{center}

\vspace{3cm}

\tableofcontents

\chapter*{{\tt Introduction}}

This text contains the slightly expanded lecture notes of an advanced graduate course on algebraic geometry that I taught in Winter 2010 at the DPMMS, Cambridge. A range of advanced topics are treated in detail which are of a foundational 
nature. Most of the topics 
come from Grothendieck's monumental works, EGA, SGA, and FGA. The way he has treated 
algebraic geometry is truly amazing. His ideas show a spectacular level of 
depth and maturity which may be unparalleled in mathematics.

I believe that anyone who wants to pursue algebraic geometry 
should ideally be taught three foundational courses. The first course would be 
on classical topics in the spirit of Shafarevich [\ref{Shafarevich}]. 
The second would be a course on sheaves and schemes using parts of Hartshorne [\ref{Hartshorne}]. 
The third would be a course on more advanced topics similar to the present 
one.

I briefly describe the chapters. In chapter one, various types of cohomology are 
discussed such as the ext sheaves and groups, higher direct images, cohomology 
with support, and local cohomology, and their basic properties are established. 
The local duality theorem for local cohomology is proved. Projective dimension 
and depth are also discussed. 

In chapter two, the relative duality theorem is proved for a projective morphism of 
Noetherian schemes and coherent sheaves satisfying a certain base change property. 
This approach follows Kleiman. Some applications of the duality theorem 
are treated, and the chapter ends with proving some basic properties of Cohen-Macaulay 
schemes.
 
In chapter three, flatness and base change are treated in detail.  The $T$ functors 
and the complex of Grothendieck are studied and applied to the base change and 
semi-continuity problems. The local invariance of the Euler characteristic and   
the Hilbert polynomial, and the generic flatness for flat sheaves are proved. 
Some basic properties of flat morphisms and flat families are treated. 
The stratification by Hilbert polynomials of a sheaf is proved, and 
some further related problems are also discussed.

In chapter four, parameterising schemes and sheaves are discussed. The 
quotient and Hilbert functors are introduced and several examples are 
given, in particular, the Grassmannian. Finally, the existence of quotient 
schemes is proved in a relatively general situation which in particular 
implies the existence of Hilbert schemes.\\

\textbf{Pre-requisites:} We assume familiarity with basic algebraic geometry. To be more precise, 
we assume that the reader is familiar with chapter I, chapter II, and sections 1-5 of chapter III of Hartshorne [\ref{Hartshorne}] but  
the rest of the materials in chapter III, except the theorem on the formal functions and smooth morphisms, are treated in these notes often in a much more general form.
We also assume familiarity with basic commutative algebra. The more advanced results 
and notions of commutative algebra are in many cases explained. We also assume 
some basic facts from homological algebra having to do with complexes and their cohomology 
objects. Moreover, Grothendieck 
spectral sequence associated with the composition of two functors is used many times.

\textbf{Conventions:} All rings are commutative with identity elements, and homomorphisms 
of rings send the identity element to identity element. The notion of a projective 
morphism of schemes is taken from Hartshorne [\ref{Hartshorne}]. For a projective 
morphism $f\colon X\to Y$, $\x{O}_X(1)$ denotes a very ample invertible sheaf on $X$ over $Y$ 
(this is not unique but we assume that a choice has been made). However, when $Y=\Spec A$ and 
$X$ is given as the Proj of a graded algebra $S$ over $A$ which is generated by elements 
of degree one, then $\x{O}_X(1)$ is uniquely determined by $S$.  
A similar situation arises when $X$ is the Proj of a graded sheaf of algebras.
A morphism of sheaves on a ringed space $(X,\x{O}_X)$ is always meant in the sense of $\x{O}_X$-modules unless otherwise stated.

\textbf{Caution:} When $Z$ is a closed subset of a topological space $X$ with 
inclusion map $i\colon Z\to X$, we sometimes instead of $i_*\x{F}$ 
just write $\x{F}$ where $\x{F}$ is a sheaf on $Z$.\\ 

\clearpage

\textbf{Notation:} Here we list some of the notations used in the text.\\

\begin{tabular}{l l}
$\mathfrak{Ab} \hspace{1cm}$ & the category of abelian groups\\
$\mathfrak{S}et \hspace{1cm}$ & the category of sets\\
$\mathfrak{S}h(X) \hspace{1cm}$ & the category of sheaves on a topological space $X$\\
$\mathfrak{M}(X) \hspace{1cm}$ & the category of $\x{O}_X$-modules on a ringed space $(X,\x{O}_X)$\\
$\mathfrak{Q}(X) \hspace{1cm}$ & the category of quasi-coherent $\x{O}_X$-modules on a scheme $X$\\
$\mathfrak{C}(X) \hspace{1cm}$ & the category of coherent $\x{O}_X$-modules on a scheme $X$\\
$\mathfrak{M}(A) \hspace{1cm}$ & the category of $A$-modules for a ring $A$\\
$\x{I}^\bullet \hspace{1cm}$ & a complex $\x{I}^0 \to \x{I}^1 \to \cdots$\\
$\x{L}_\bullet \hspace{1cm}$ & a complex $\cdots \to \x{L}_1 \to \x{L}_0 $\\
$\xExt^p_f(\x{F},\x{G}) \hspace{1cm}$ & the $p$-th relative ext sheaf of the sheaves $\x{F},\x{G}$\\
$\xExt^p_{\x{O}_X}(\x{F},\x{G}) \hspace{1cm}$ & the $p$-th ext group of the sheaves $\x{F},\x{G}$ on $X$\\
$\Ext^p_{A}(M,N) \hspace{1cm}$ & the $p$-th ext group of the modules $M,N$ over $A$\\
$\pd M \hspace{1cm}$  & the projective dimension of the module $M$\\
$\depth_IM \hspace{1cm}$  & the depth of the module $M$ in the ideal $I$\\
$H^p_Z(X,\x{F}) \hspace{1cm}$  & the $p$-th cohomology of the sheaf $\x{F}$ with support in $Z$\\
$H^p_I(M) \hspace{1cm}$  &  the $p$-th local cohomology of the module $M$ in the ideal $I$\\ 
$(f^!,t_f) \hspace{1cm}$  & the dualising pair of a morphism $f$\\ 
$\omega_f \hspace{1cm}$  & the dualising sheaf of a morphism $f$\\
$\Tor^A_p(M,N) \hspace{1cm}$  & the $p$-th tor group of the modules $M,N$ over $A$\\
$T^p_{\x F}(M) \hspace{1cm}$  & the $p$-th $T$ functor of a flat sheaf $\x{F}$\\
$\Phi \hspace{1cm}$ & often denotes the Hilbert polynomial of a sheaf or scheme\\ 
$\rm{N}\mathfrak{S}ch/Y \hspace{1cm}$ & the category of Noetherian schemes over a scheme $Y$\\  
$\cHilb_{X/Y}^{\Phi,\x{L}} \hspace{1cm}$ & the Hilbert functor of $X/Y$, $\Phi$, $\x{L}$\\ 
$\Hilb_{X/Y}^{\Phi,\x{L}} \hspace{1cm}$ &  the Hilbert scheme of $X/Y$, $\Phi$, $\x{L}$\\
$\cQuot_{\x{F}/X/Y}^{\Phi,\x{L}} \hspace{1cm}$ & the quotient functor of $X/Y$, $\Phi$, $\x{L}$, $\x{F}$\\
$\Quot_{\x{F}/X/Y}^{\Phi,\x{L}} \hspace{1cm}$ & the quotient scheme of $X/Y$, $\Phi$, $\x{L}$, $\x{F}$\\
$\xHom_{Y}^\Phi(X,X') \hspace{1cm}$ & the Hom functor $X,X'/Y$, $\Phi$\\
$\Hom_{Y}^\Phi(X,X') \hspace{1cm}$ & the Hom scheme $X,X'/Y$, $\Phi$\\
$\Sym^m X \hspace{1cm}$ & the $m$-th symmetric product of $X$\\
$\Grass(\x{F},d) \hspace{1cm}$ & the Grassmannian of a locally free sheaf $\x{F}$\\
\end{tabular}


\chapter{\tt Cohomology} 
 
In this chapter we discuss several types of cohomology functors such as the 
ext groups, ext sheaves, higher direct images, cohomology with support, and local cohomology. 
In the context of ext we discuss a local-to-global spectral sequence, 
and projective modules, and in relation with local cohomology 
we discuss the notion of depth of modules, and local duality.

More information about these topics can be found in Grothendieck SGA 2 [\ref{SGA2}].
Though here we have discussed the ext sheaves more general than [\ref{SGA2}] in some 
sense.\\

\section{Ext sheaves and groups}

\begin{defn}\label{d-general-ext}
Let $f\colon (X,\x{O}_X)\to (Y,\x{O}_Y)$ be a morphism of ringed 
spaces, and let $\mathfrak{M}(X)$ and $\mathfrak{M}(Y)$ be the category of $\x{O}_X$-modules 
and $\x{O}_Y$-modules respectively. Let $\x{F}\in \mathfrak{M}(X)$. 
 We define $\xExt^p_f(\x{F},-)$ to be the right derived functors of the left exact functor $f_*\xHom_{\x{O}_X}(\x{F},-)\colon \mathfrak{M}(X)\to \mathfrak{M}(Y)$.
\end{defn}

Evidently, the ringed structure $\x{O}_Y$ on $Y$ does not play any role 
but having such a structure would make the sheaves $\xExt^p_f(\x{F},\x{G})$ into 
$\x{O}_Y$-modules. 
If $f\colon X\to Y$ is just a continuous map of topological spaces, then 
the above functors are defined by considering $X$ and $Y$ as ringed spaces 
by taking the ringed structure to be the one defined by the constant sheaf 
associated with $\Z$.

Obviously the above definition is very general and it should not come as a surprise that 
we can derive from it several types of cohomologies by considering special cases.  

\begin{defn-rem}\label{dr-ext}
Assume the setting of Definition \ref{d-general-ext}. 
\begin{enumerate}
\item When $Y$ is just a point we put $\Ext^p_{\x{O}_X}(\x{F},\x{G}):=\xExt^p_f(\x{F},\x{G})$. That is, $\Ext^p_{\x{O}_X}(\x{F},-)$ are the right derived functors of the left exact functor $\Hom_{\x{O}_X}(\x{F},-)$. Note that the usual cohomology functor $H^p(X,-)\simeq\Ext^p_{\x{O}_X}(\x{O}_X,-)$ 
because we have $\Hom_{\x{O}_X}(\x{O}_X,-)\simeq H^0(X,-)$.\\

\item If $f$ is the identity, then instead of $\xExt^p_f(\x{F},\x{G})$ we  
write $\xExt^p_{\x{O}_X}(\x{F},\x{G})$. That is, $\xExt^p_{\x{O}_X}(\x{F},-)$ are the right derived functors of the left exact functor $\xHom_{\x{O}_X}(\x{F},-)$. In particular, 
 $\xExt^0_{\x{O}_X}(\x{O}_X,\x{G})\simeq \x{G}$ 
and $\xExt^p_{\x{O}_X}(\x{O}_X,\x{G})=0$ if $p>0$ because the functor $\xHom_{\x{O}_X}(\x{O}_X,-)\simeq -$ is 
exact and so its right derived functors are trivial.\\

\item Since $f_*\xHom_{\x{O}_X}(\x{O}_X,-)=f_*(-)$, we have $R^pf_*(-)=\xExt^p_f(\x{O}_X,-)$. 
The functors $R^pf_*(-)$ are the right derived functors of the left exact functor $f_*$.\\

\item Later we will see that cohomology with support is also a special 
instance of Ext. 
\end{enumerate}
\end{defn-rem}

\begin{thm}\label{t-ext-sheaf-associated}
The sheaf $\xExt^p_f(\x{F},\x{G})$ is the sheaf associated to the 
presheaf $U\mapsto \Ext^p_{\x{O}_{f^{-1}U}}(\x{F}|_{f^{-1}U},\x{G}|_{f^{-1}U})$ on $Y$. 
In particular, for any open subset $W\subseteq Y$, we have 
$$
\xExt^p_f(\x{F},\x{G})|_{W}\simeq  \xExt^p_f(\x{F}|_{f^{-1}W},\x{G}|_{f^{-1}W})
$$
\end{thm}
\begin{proof}
Let 
$$
0 \to \x{G} \to \x{I}^0 \to  \x{I}^1 \to \cdots
$$
be an injective resolution in $\mathfrak{M}(X)$. Then, the sheaves $\xExt^p_f(\x{F},\x{G})$ 
are the cohomology objects of the complex 
$$
0 \to f_*\xHom(\x{F},\x{I}^0) \xrightarrow{d^0}  f_*\xHom(\x{F},\x{I}^1) \xrightarrow{d^1} \cdots
$$

Let $\x{P}_3=\ker d^p$, $\x{P}_2=\im d^{p-1}$, and let $\x{P}_1$ be the presheaf image of 
$d^{p-1}$. Then, $\xExt^p_f(\x{F},\x{G})=\frac{\x{P}_3}{\x{P}_2}$, and the presheaf quotient 
$\frac{\x{P}_3}{\x{P}_1}$ is the presheaf which assigns to $U$ the $p$-th cohomology object of the complex 
$$
0 \to \Hom_{\x{O}_{f^{-1}U}}(\x{F}|_{f^{-1}U},\x{I}^0|_{f^{-1}U}) \xrightarrow{d^0} \Hom_{\x{O}_{f^{-1}U}}(\x{F}|_{f^{-1}U},\x{I}^1|_{f^{-1}U}) \xrightarrow{d^1} \cdots
$$
which is $\Ext^p_{\x{O}_{f^{-1}U}}(\x{F}|_{f^{-1}U},\x{G}|_{f^{-1}U})$ because 
$$
0 \to \x{G}|_{f^{-1}U} \to \x{I}^0|_{f^{-1}U} \to  \x{I}^1|_{f^{-1}U} \to \cdots
$$
is an injective resolution (see Exercise \ref{exe-injective-restriction-tensor}). By construction, the associated sheaf $\x{P}_1^+=\x{P}_2$. Since 
taking associated sheaves of presheaves is an exact functor, 
$$
(\frac{\x{P}_3}{\x{P}_1})^+=\frac{\x{P}_3^+}{\x{P}_1^+}=\frac{\x{P}_3}{\x{P}_2}=\xExt^p_f(\x{F},\x{G})
$$ 

Regarding the second statement, $\xExt^p_f(\x{F},\x{G})|_{W}$ is the sheaf associated 
to the presheaf $U\mapsto \Ext^p_{\x{O}_{f^{-1}U}}(\x{F}|_{f^{-1}U},\x{G}|_{f^{-1}U})$ 
where $U$ runs through the open subsets of $W$. This in turn is the 
sheaf $\xExt^p_f(\x{F}|_{f^{-1}W},\x{G}|_{f^{-1}W})$.\\
\end{proof}

\begin{thm}\label{t-ext-locally-free}
Let $\x{L}\in \mathfrak{M}(X)$ and $\x{N}\in\mathfrak{M}(Y)$ be locally free sheaves of finite rank. Then, we have 
$$
\xExt^p_f(\x{F}\otimes \x{L},-\otimes f^*\x{N})\simeq \xExt^p_f(\x{F},-\otimes {\x{L}}^{\vee}\otimes f^*\x{N})\simeq \xExt^p_f(\x{F},-\otimes {\x{L}}^{\vee})\otimes \x{N}
$$
where $^{\vee}$ is the dual.
\end{thm}
\begin{proof}
If $\x{I}\in \mathfrak{M}(X)$ is injective then $\x{I}\otimes \x{M}$ is also 
injective for any locally free sheaf $\x{M}$ of finite rank (see Exercise \ref{exe-injective-restriction-tensor}). Now let $\x{G}\in \mathfrak{M}(X)$ and let $0 \to \x{G} \to \x{I}^\bullet$ be an injective resolution. The functorial isomorphisms  
$$
\xHom_{\x{O}_X}(\x{F}\otimes \x{L},\x{I}^p\otimes f^*\x{N})\simeq \xHom_{\x{O}_X}(\x{F},\x{I}^p\otimes {\x{L}}^{\vee}\otimes f^*\x{N})
$$
$$
\simeq \xHom_{\x{O}_X}(\x{F},\x{I}^p\otimes {\x{L}}^{\vee})\otimes f^*\x{N}
$$
give
$$
f_*\xHom_{\x{O}_X}(\x{F}\otimes \x{L},\x{I}^\bullet\otimes f^*\x{N})\simeq f_*\xHom_{\x{O}_X}(\x{F},\x{I}^\bullet\otimes {\x{L}}^{\vee}\otimes f^*\x{N})
$$
$$
\simeq f_*(\xHom_{\x{O}_X}(\x{F},\x{I}^\bullet\otimes {\x{L}}^{\vee}))\otimes \x{N}
$$
which implies the result immediately.\\
\end{proof}

In the theorem, by taking $\x{L}$ to be $\x{O}_X$ we get the projection formula
$$
\xExt^p_f(\x{F},\x{G}\otimes f^*\x{N})\simeq \xExt^p_f(\x{F},\x{G})\otimes \x{N}
$$
and further by taking $\x{F}$ to be $\x{O}_X$ we get the projection 
formula 
$$
R^pf_*(\x{G}\otimes f^*\x{N})\simeq R^pf_*(\x{G})\otimes \x{N}
$$
  
On the other hand, by taking $f$ to be the identity we get 
$$
\xExt^p_{\x{O}_X}(\x{F}\otimes \x{L},\x{G})\simeq \xExt^p_{\x{O}_X}(\x{F},\x{G}\otimes {\x{L}}^\vee)
\simeq \xExt^p_{\x{O}_X}(\x{F},\x{G})\otimes {\x{L}^\vee}
$$\\

\begin{thm}\label{t-ext-long-sequence}
Let $0\to \x{F}'\to \x{F} \to \x{F}'' \to 0$ be an exact sequence of $\x{O}_X$-modules. 
Then, for any $\x{O}_X$-module $\x{G}$ we get a long exact sequence 
{
$$
\xymatrix{
\cdots \to \xExt^p_f(\x{F}'', \x{G}) \to 
 \xExt^p_f(\x{F}, \x{G}) \to
 \xExt^p_f(\x{F}', \x{G}) \to
\xExt^{p+1}_f(\x{F}'', \x{G}) \to\cdots 
}
$$}
\end{thm}
\begin{proof}
Let  $0 \to \x{G} \to \x{I}^\bullet$ be an injective resolution in $\mathfrak{M}(X)$.
For any open subset $U\subseteq X$ and any $p$, the sheaf $\x{I}^p|_U$ is injective hence the sequence 
$$
\xymatrix{
0 \to \Hom_{\x{O}_U}(\x{F}''|_U, \x{I}^p|_U) \to
 \Hom_{\x{O}_U}(\x{F}|_U, \x{I}^p|_U) \to
 \Hom_{\x{O}_U}(\x{F}'|_U, \x{I}^p|_U) \to  0 
}
$$
is exact which implies the exactness of the sequence 
$$
\xymatrix{
0 \to f_*\xHom_{\x{O}_X}(\x{F}'', \x{I}^p) \to
 f_*\xHom_{\x{O}_X}(\x{F}, \x{I}^p) \to
 f_*\xHom_{\x{O}_X}(\x{F}', \x{I}^p) \to  0 
}
$$
so we get an exact sequence of complexes
$$
\xymatrix{
0 \to f_*\xHom_{\x{O}_X}(\x{F}'', \x{I}^\bullet) \to
 f_*\xHom_{\x{O}_X}(\x{F}, \x{I}^\bullet) \to
 f_*\xHom_{\x{O}_X}(\x{F}', \x{I}^\bullet) \to  0 
}
$$
whose long exact sequence is the one we are looking for.\\
\end{proof}

\begin{thm}\label{t-acyclic-resolution}
Fix $\x{G}\in\mathfrak{M}(X)$. Suppose that $\x{L}_\bullet \to  \x{F} \to 0$ is an exact sequence where  
$$
\x{L}_\bullet:\hspace{1cm} \cdots \to \x{L}_1 \to \x{L}_0 
$$
is a sequence in which $\xExt^p_f(\x{L}_i, \x{G})=0$ for every $i\ge 0$ and $p\ge 1$. 
Then, $\xExt^p_f(\x{F}, \x{G})$ are the cohomology objects of the complex 
$0\to f_*\xHom_{\x{O}_X}(\x{L}_\bullet, \x{G})$.
\end{thm}
\begin{proof}
We use induction on $p$.
Let $\x{K}_{-1}:=\x{F}$ and for each $i\ge 0$, let $\x{K}_i$ be the image of $\x{L}_{i+1} \to \x{L}_{i}$. Then,  we have exact sequences 
$$
0\to \x{K}_{i+1} \to \x{L}_{i+1} \to \x{K}_{i}\to 0
$$
for every $i\ge -1$ from which we get a long exact sequence  
$$
\cdots \to \xExt^p_f(\x{K}_i, \x{G}) \to 
 \xExt^p_f(\x{L}_{i+1}, \x{G}) \to
 \xExt^p_f(\x{K}_{i+1}, \x{G}) \to
\xExt^{p+1}_f(\x{K}_i, \x{G}) \to\cdots 
$$
which gives an exact sequence 
$$
0 \to \xExt^0_f(\x{K}_i, \x{G}) \to 
 \xExt^0_f(\x{L}_{i+1}, \x{G}) \to
 \xExt^0_f(\x{K}_{i+1}, \x{G}) \to
\xExt^{1}_f(\x{K}_i, \x{G}) \to 0
$$
and an isomorphism 
$\xExt^p_f(\x{K}_{i+1}, \x{G}) \simeq \xExt^{p+1}_f(\x{K}_i, \x{G})$ for each $p\ge 1$.

For $p=0$, the result is clear from the above sequences. For $p=1$, $\xExt^{1}_f(\x{F}, \x{G})$
is the quotient of  $\xExt^0_f(\x{K}_{0}, \x{G})$ by the image of the morphism 
$ \xExt^0_f(\x{L}_{0}, \x{G}) \to \xExt^0_f(\x{K}_{0}, \x{G})$. The exact 
sequence $\x{L}_2 \to \x{L}_1 \to \x{K}_0\to 0$ gives the exact sequence 
$$
0 \to \xExt^0_f(\x{K}_0, \x{G}) \to 
 \xExt^0_f(\x{L}_{1}, \x{G}) \to
 \xExt^0_f(\x{L}_{2}, \x{G})
$$
meaning that $\xExt^0_f(\x{K}_{0}, \x{G})$ is just the kernel of the morphism 
$\xExt^0_f(\x{L}_{1}, \x{G}) \to
 \xExt^0_f(\x{L}_{2}, \x{G})$.
On the other hand, since  
$
\xExt^0_f(\x{K}_0, \x{G}) \to 
 \xExt^0_f(\x{L}_{1}, \x{G})
$
is injective, the image of the morphism  
$
\xExt^0_f(\x{L}_0, \x{G}) \to 
 \xExt^0_f(\x{L}_{1}, \x{G})
$
and the image of 
$
\xExt^0_f(\x{L}_0, \x{G}) \to 
 \xExt^0_f(\x{K}_{0}, \x{G})
$
are isomorphic. Therefore, $\xExt^{1}_f(\x{F}, \x{G})$ is the cohomology of the 
complex 
$$
\xExt^0_f(\x{L}_0, \x{G}) \to 
 \xExt^0_f(\x{L}_{1}, \x{G}) \to
 \xExt^0_f(\x{L}_{2}, \x{G})
$$
in the middle.

Note that the same arguments apply when we replace $\x{F}$ by $\x{K}_0$, 
and we replace $\cdots \to \x{L}_1 \to \x{L}_0 $ by $\cdots \to \x{L}_2 \to \x{L}_1$. 
This process can be continued by replacing $\x{K}_0$ by $\x{K}_1$ and so on.

Now the isomorphism $\xExt^p_f(\x{K}_{i+1}, \x{G}) \simeq \xExt^{p+1}_f(\x{K}_i, \x{G})$ 
for $p\ge 1$, and the above arguments reduce the problem inductively to the case 
$p=0,1$ which was verified.\\
\end{proof}

\section{The local-to-global spectral sequence}

Let $f\colon (X,\x{O}_X)\to (Y,\x{O}_Y)$ and $g\colon (Y,\x{O}_Y)\to (Z,\x{O}_Z)$ be morphisms 
of ringed spaces and let $h=gf$. The question of what kind of relations one might expect 
between the ext sheaves defined for $f$, $g$, and $h$ arises naturally. The answer lies in 
the theory of spectral sequences. Let $\x{F}$ be an $\x{O}_X$-module. Then, we have a 
commutative diagram of functors 
$$
\xymatrix{
\mathfrak{M}(X) \ar[rd]^{\gamma} \ar[r]^{\alpha} & \mathfrak{M}(Y)\ar[d]^{\beta}\\
& \mathfrak{M}(Z)
}
$$
in which $\alpha$ is the functor  $f_*\xHom_{\x{O}_X}(\x{F},-)$, $\beta$ is 
the functor $g_*$ and $\gamma$ is the functor $h_*\xHom_{\x{O}_X}(\x{F},-)$.\\

\begin{thm}\label{t-ext-spectral-sequence}
For any $\x{G}\in\mathfrak{M}(X)$ 
there is a spectral sequence 
$$
E_2^{q,p}=R^qg_*\xExt^p_f(\x{F}, \x{G}) \implies \xExt^{q+p}_h(\x{F}, \x{G})
$$
\end{thm}
\begin{proof}
Let $\x{I}$ be an injective sheaf in $\mathfrak{M}(X)$. Then, we will show that 
$\xHom_{\x{O}_X}(\x{F},\x{I})$ is a flasque sheaf. Indeed let $U$ be an open subset of $X$ 
and let $\phi\in \xHom_{\x{O}_X}(\x{F},\x{I})(U)$, that is, a morphism 
$\phi \colon \x{F}|_U\to \x{I}|_U$. If $j\colon U\to X$ is the inclusion, 
then we get a morphism $j_!\phi \colon j_!\x{F}|_U\to j_!\x{I}|_U$. 
The fact that $\x{I}$ is injective and that we have a natural injective 
morphism $j_!\x{F}|_U\to \x{F}$ induces a morphism $\psi\colon \x{F}\to \x{I}$ 
which restricts to $\phi$. So,  $\xHom_{\x{O}_X}(\x{F},\x{I})(X)\to \xHom_{\x{O}_X}(\x{F},\x{I})(U)$
is surjective hence the claim follows.

Now $f_*\xHom_{\x{O}_X}(\x{F},\x{I})$ is also a flasque sheaf. Thus, $R^pg_*(f_*\xHom_{\x{O}_X}(\x{F},\x{I}))=0$ if $p>0$, that is, $\alpha$ sends injective objects to $\beta$-acyclic objects. 
Therefore, a theorem of Grothendieck provides the desired spectral 
sequence.\\
\end{proof} 

In the theorem, when $Z$ is just a point and $f$ is the identity map 
the spectral sequence is called the local-to-global spectral sequence for ext.\\

\section{Quasi-coherence and coherence of ext sheaves}

\begin{rem-thm}\label{r-ext-coherent}
Let $f\colon X\to Y$ be a morphism of Noetherian schemes.\\

\begin{enumerate}
\item Suppose that $\xExt^p_f(\x{F},\x{G})$ is quasi-coherent for every $p$   
and assume that $Y$ is affine. Then, $H^q(Y,\xExt^p_f(\x{F},\x{G}))=0$ for 
any $p$ and any $q>0$. Now let $Z$ be a single point and let $X\to Z$ and $Y\to Z$ be the 
constant maps. Then, the above vanishing and the spectral sequence of Theorem \ref{t-ext-spectral-sequence} give an isomorphism
$$
H^0(Y,\xExt^p_f(\x{F},\x{G}))\simeq \Ext^p_{\x{O}_X}(\x{F},\x{G})
$$
hence the sheaf $\xExt^p_f(\x{F},\x{G})$ is just $\Ext^p_{\x{O}_X}(\x{F},\x{G})^\sim$.\\

\item Assume that $\x{G}$ is quasi-coherent and that there is a resolution $\x{L}_\bullet$ 
of $\x{F}$ as in Theorem \ref{t-acyclic-resolution} such that all the 
$\x{L}_i$ are coherent. Then, each $f_*\xHom_{\x{O}_X}(\x{L}_{i}, \x{G})$ would be 
quasi-coherent hence the sheaves $\xExt^p_f(\x{F},\x{G})$ would all be 
quasi-coherent. Moreover, if $\x{G}$ is coherent and if $f$ is projective 
then each\\ $f_*\xHom_{\x{O}_X}(\x{L}_{i}, \x{G})$ is 
coherent hence the sheaves $\xExt^p_f(\x{F},\x{G})$ are all coherent.\\

\item Let $\x{F}$ be coherent and let $\xExt^p_f(\x{F}, -)_{\mathfrak Q}$ be the right derived functors of the left exact functor
$f_*\xHom_{\x{O}_X}(\x{F}, -)\colon \mathfrak{Q}(X)\to \mathfrak{M}(Y)$ 
which form a universal $\delta$-functor. Note that by definition 
$\xExt^p_f(\x{F}, \x{G})_{\mathfrak Q}$ is quasi-coherent for any quasi-coherent $\x{G}$. 

On the other hand, the functors $\xExt^p_f(\x{F}, -)$ restricted to $\mathfrak{Q}(X)$ form 
a $\delta$-functor. The isomorphism $\xExt^0_f(\x{F}, -)_{\mathfrak Q}\simeq \xExt^0_f(\x{F}, -)$ 
induces a unique sequence of maps $\xExt^p_f(\x{F}, -)_{\mathfrak Q}\to\xExt^p_f(\x{F}, -)$.

Now assume that $\x{F}=\x{O}_X$. Then the functors $\xExt^p_f(\x{O}_X, -)$ also form a universal $\delta$-functor 
because in this case the functors are effaceable for any $p>0$ : each $\x{G}\in\mathfrak{Q}(X)$ can be embedded into a flasque quasi-coherent sheaf $\x{H}$ which satisfies $\xExt^p_f(\x{O}_X, \x{H})=0$ by Theorem \ref{t-ext-sheaf-associated} and Definition-Remark \ref{dr-ext}.
Therefore, the maps $\xExt^p_f(\x{O}_X, -)_{\mathfrak Q}\to\xExt^p_f(\x{O}_X, -)$ are isomorphisms and so the sheaves $R^pf_*\x{G}$ are all quasi-coherent when 
$\x{G}$ is quasi-coherent.\\

\item Fix a coherent sheaf $\x{G}$ on $X$. If $f$ is projective then locally over $Y$ 
every coherent sheaf $\x{F}$ has a resolution $\x{L}_\bullet$ as in Theorem \ref{t-acyclic-resolution}. In fact, we could assume that $Y$ is affine, and if $\x{O}_X(1)$ is a 
very ample invertible sheaf for $f$, then the sheaf $\x{F}(l)$ is generated by 
finitely many global sections for any $l\gg 0$. Pick $l$ sufficiently large. 
Then we get a surjective morphism
$\x{O}_X^n\to \x{F}(l)$ for some $n$ hence a surjective morphism $\x{L}_0:=\x{O}_X^n(-l)\to \x{F}$.
Moreover, for any $p>0$
$$
\xExt^p_f(\x{L}_0,\x{G})\simeq \xExt^p_f(\x{O}_X,\x{G}\otimes {\x{L}_0}^{\vee})\simeq 
R^pf_*(\x{G}\otimes {\x{L}_0}^{\vee})
$$
$$
\simeq R^pf_*\x{G}(l)^n=0 
$$
because the latter sheaf is the sheaf associated to $H^p(X,\x{G}(l))^n$ which is zero 
by Hartshorne [\ref{Hartshorne}, III, Theorem 5.2].

By considering the kernel of $\x{L}_0\to \x{F}$ and continuing this process we can construct $\x{L}_\bullet$. Since $\x{F}$ and $\x{G}$ are both assumed to be coherent 
then the sheaves $\xExt^p_f(\x{F},\x{G})$ are all coherent.\\

\item Assume that $X=\Spec A$, $f$ is the identity, and $M$ and $N$ 
are $A$-modules with $M$ finitely generated. Then $\x{F}=\tilde{M}$ 
has a resolution $\x{L}_\bullet$ as in Theorem \ref{t-acyclic-resolution}, 
consisting of free coherent sheaves, which works for every $\x{G}=\tilde{N}$.
Thus, $\xExt^p_{\x{O}_X}(\tilde{M},\tilde{N})$ is quasi-coherent and it is the 
sheaf associated to the module $\Ext^p_{\x{O}_X}(\tilde{M},\tilde{N})$. 

We define $\Ext^p_{A}({M},-)$ to be the right derived functors of the 
left exact functor $\Hom_A(M,-)$ from the category of $A$-modules to itself 
(in general $M$ need not be finitely generated).
We will prove that $\Ext^p_{\x{O}_X}(\tilde{M},\tilde{N})\simeq \Ext^p_{A}({M},{N})$. 
This holds when $p=0$ because of the isomorphism  $\Hom_{\x{O}_X}(\tilde{M},\tilde{N})\simeq \Hom_{A}({M},{N})$. 
 On the other hand, for 
a fixed $N$, both $\Ext^p_{\x{O}_X}(-,\tilde{N})$ and $\Ext^p_{A}(-,{N})$ define $\delta$-functors from the category of 
finitely generated $A$-modules to the category of $A$-modules. Both $\delta$-functors are universal as one can show 
that they are coeffaceable by considering free resolutions of finitely generated $A$-modules. 
Therefore we get 
isomorphisms for all $p$.\\

\item\label{t-ext-qc-o-module} Let $\x{F}$ be coherent and assume that there is a resolution $\x{L}_\bullet\to \x{F}\to 0$  such that all the $\x{L}_i$ are coherent and $\xExt^p_f(\x{L}_i, \x{G})_{\mathfrak Q}=0$ for every $i\ge 0$ and $p\ge 1$. 
Then, $\xExt^p_f(\x{F}, \x{G})_{\mathfrak Q}$ are the cohomology objects of the complex 
$0\to f_*\xHom_{\x{O}_X}(\x{L}_\bullet, \x{G})$. This can be proved exactly 
as in the proof of Theorem \ref{t-acyclic-resolution}. 

In particular, if $\x{G}$ is also coherent, then under the assumptions of (4), the 
sequence $\x{L}_\bullet$ constructed in (4) also satisfies $\xExt^p_f(\x{L}_i, \x{G})_{\mathfrak Q}=0$. Therefore, if $f$ is projective the map $\xExt^p_f(\x{F}, -)_{\mathfrak Q}\to\xExt^p_f(\x{F}, -)$ in (3) 
can be shown, locally, to be an isomorphism because the cohomology objects of the complex $0\to f_*\xHom_{\x{O}_X}(\x{L}_\bullet, \x{G})$ calculate both sides of the map.\\ 
\end{enumerate}
\end{rem-thm}

\begin{thm}\label{t-ext-local-global-degen}
Let $f\colon X\to Y$ and $g\colon Y\to Z$ be projective morphisms of Noetherian schemes, $h=gf$
and $\x{F}$ and $\x{G}$ coherent sheaves on $X$. If $\x{O}_Y(1)$ is a very ample invertible sheaf 
for $g$, then for any $l\gg 0$ we have isomorphisms 
$$
g_*\xExt^p_f(\x{F}, \x{G}\otimes f^*\x{O}_Y(l)) \simeq \xExt^{p}_h(\x{F}, \x{G}\otimes f^*\x{O}_Y(l))
$$
\end{thm}
\begin{proof}
By Theorem \ref{t-ext-locally-free}, we have 
$$
E^{q,p}_2:=R^qg_*\xExt^p_f(\x{F}, \x{G}\otimes f^*\x{O}_Y(l))\simeq R^qg_*(\xExt^p_f(\x{F}, \x{G})\otimes \x{O}_Y(l))
$$
By Remark \ref{r-ext-coherent} the sheaves $\xExt^p_f(\x{F}, \x{G})$ are coherent and
$$
 R^qg_*(\xExt^p_f(\x{F}, \x{G})\otimes \x{O}_Y(l))=0
$$ 
whenever $l\gg 0$ and $q>0$. Now the spectral sequence of Theorem \ref{t-ext-spectral-sequence} gives the result
because $E^{q,p}_\infty=E^{q,p}_2=0$ unless $q=0$.\\
\end{proof}

In the theorem, if $Z$ is just a single point, then we get the isomorphism 
$$
H^0(Y,\xExt^p_f(\x{F}, \x{G}\otimes f^*\x{O}_Y(l))) \simeq \Ext^{p}_{\x{O}_X}(\x{F}, \x{G}\otimes f^*\x{O}_Y(l))
$$
\vspace{0.1cm}

\section{Ext and projective dimension}

\begin{defn}
Let $A$ be a ring and $M$ an $A$-module. The projective dimension of $M$ over $A$, denoted 
by $\pd(M)$ is the minimum of $r$ such that there is an exact sequence
$$
0\to P_r\to \cdots \to {P}_1 \to {P}_0\to M \to 0
$$ 
with each $P_i$ a projective 
$A$-module. If $M$ is already projective then obviously $\pd(M)=0$ by taking 
$P_0=M$ and $P_i=0$ for $i>0$. Conversely, if $\pd(M)=0$, then $M$ is projective.\\
\end{defn}

\begin{thm}\label{t-proj-dimension}
Let $A$ be a ring and $M$ an $A$-module. Then, $\pd(M)\le n$ if and only if 
$\Ext^p_A(M,N)=0$ for any $p> n$ and any $A$-module $N$.
\end{thm}
\begin{proof}
We do induction on $n$. First let $n=0$. If $\pd(M)=0$, then as mentioned above 
$M$ is projective. Let $N$ be an $A$-module and let $0\to N\to I\to K\to 0$ 
be a short exact sequence where $I$ is an injective $A$-module. 
Then, we have the long exact sequence 
$$
0 \to \Hom_A(M,N)\to \Hom_A(M,I)\to \Hom_A(M,K)\to \Ext^1_A(M,N)\to \Ext^1_A(M,I) \to \cdots 
$$ 
Since $I$ is injective, $\Ext^p_A(M,I)=0$ for any $p>0$. On the other hand, 
since $M$ is projective the functor $\Hom_A(M,-)$ is exact hence 
$\Ext^1_A(M,N)=0$ and $ \Ext^p(M,K)\simeq \Ext^{p+1}_A(M,N)$ for every $p>0$. 
But the same argument gives $\Ext^1_A(M,K)=0$ hence $\Ext^{2}_A(M,N)=0$ and eventually 
$\Ext^p_A(M,N)=0$ for any $p>0$.

Conversely, assume that $\Ext^p_A(M,N)=0$ for any $p>0$ and any $A$-module $N$. 
Then, the functor $\Hom_A(M,-)$ is exact hence $M$ is projective. Clearly, 
only the condition for $p=1$ is enough to imply the projectivity of $M$.

Now assume that $n>0$ and that the theorem holds up to $n-1$. Suppose that $\pd(M)\le n$. 
By induction, we may assume that $\pd(M)=n$. Then, we have a resolution 
$$
0\to P_n\to \cdots \to {P}_1 \to {P}_0\to M \to 0
$$
where each $P_i$ is projective. Let $E$ be the kernel of $P_0\to M$. 
We then have the exact sequence 
$$
0\to P_n\to \cdots \to {P}_2 \to {P}_1\to E \to 0
$$
which in particular means that $\pd(E)\le n-1$. So, by induction 
$\Ext^p_A(E,N)=0$ for any $p>n-1$ and any $A$-module $N$. On the other hand, 
the short exact sequence $0 \to E\to P_0\to M\to 0$ gives the long 
exact sequence  
$$
 \cdots \to \Ext^{p}_A(M,N)\to \Ext^p_A(P_0,N) \to \Ext^p_A(E,N) \to \Ext^{p+1}_A(M,N) \to\cdots 
$$ 
(proved similar to Theorem \ref{t-ext-long-sequence}) which implies that $\Ext^p(E,N)\simeq \Ext^{p+1}_A(M,N)$ for every $p>0$. 
Therefore, $\Ext^p_A(M,N)=0$ for any $p>n$.

Conversely, assume that $\Ext^p_A(M,N)=0$ for any $p>n$ and any $A$-module $N$. 
Let $P_0\to M$ be a surjective map with $P_0$ projective, and let $E$ be the 
kernel. Then, the latter long exact sequence implies that 
$\Ext^p_A(E,N)=0$ for any $p>n-1$ and any $A$-module $N$. By induction, 
$\pd(E)\le n-1$ and there is a resolution 
$$
0\to P_r\to \cdots \to {P}_2 \to {P}_1\to E \to 0
$$
by projective modules where $r\le n$. We then get the resolution
$$
0\to P_r\to \cdots \to {P}_1 \to {P}_0\to M \to 0
$$
which implies that $\pd(M)\le n$.\\
\end{proof}


\section{Cohomology with support}

\begin{defn}
Let $Z$ be a closed subset of a topological space $X$, and let $U=X\setminus Z$. For any sheaf 
$\x{F}$ on $X$, put $\Gamma_Z(X,\x{F})=\ker \x{F}(X)\to \x{F}(U)$. It is easy to see 
that the functor $\Gamma_Z(X,-)\colon \mathfrak{S}h(X)\to \mathfrak{Ab}$ from the category 
of sheaves on $X$ to the category of abelian groups is a left exact functor. We denote  
its right derived functors by $H^p_Z(X,-)$ which are called the cohomology groups 
with support in $Z$.
\end{defn}

The cohomology with support is actually a special case of ext. In fact, let 
$\x{O}_X$ be the ringed structure on $X$ given by the constant sheaf 
associated to $\Z$. The inclusion maps $j\colon U\to X$ and $i\colon Z\to X$ give 
an exact sequence 
$$
0 \to j_!j^*\x{O}_X \to \x{O}_X \to i_*i^{-1}\x{O}_X \to 0
$$ 
and for any 
sheaf $\x{F}$ on $X$ we have 
$$
\x{F}(X)=\Hom_{\x{O}_X}(\x{O}_X,\x{F}) ~\mbox{ and} ~\x{F}(U)=\Hom_{\x{O}_X}(j_!j^*\x{O}_X,\x{F})
$$ 
Therefore, the exact sequence 
$$
0\to \Hom_{\x{O}_X}(i_*i^{-1}\x{O}_X,\x{F}) \to \Hom_{\x{O}_X}(\x{O}_X,\x{F}) \to \Hom_{\x{O}_X}(j_!j^*\x{O}_X,\x{F})
$$
gives $\Gamma_Z(X,\x{F})=\Hom_{\x{O}_X}(i_*i^{-1}\x{O}_X,\x{F})$ hence the long exact 
sequence 
$$
\cdots \to \Ext^p_{\x{O}_X}(i_*i^{-1}\x{O}_X,\x{F})\to \Ext^p_{\x{O}_X}(\x{O}_X,\x{F})\to \Ext^p_{\x{O}_X}(j_!j^*\x{O}_X,\x{F})\to \cdots 
$$ 
 is the same as the exact sequence in the next theorem.\\

\begin{thm}\label{t-lc-exact-sequence} 
For every sheaf $\x{F}$ on $X$ we have an exact sequence 
$$
\cdots \to H^p_Z(X,\x{F}) \to H^p(X,\x{F})\to H^p(U,\x{F}|_U)\to \cdots
$$\\
\end{thm}

The next theorem says that cohomology with support in $Z$ reflects properties 
of the sheaf near $Z$.\\

\begin{thm}[Excision]\label{t-lc-excision}
Assume that $Z\subseteq W$ for some open subset $W\subseteq X$. Then, for every $p$, we have 
$H^p_Z(W,\x{F}|_W)\simeq H^p_Z(X,\x{F})$.
\end{thm}
\begin{proof}
Let 
$$
0 \to \x{F} \to \x{I}^0 \to  \x{I}^1 \to \cdots
$$
be an injective resolution in $\mathfrak{S}h(X)$. Then, $H^p_Z(X,\x{F})$ are 
by definition the cohomology objects of the complex
$$
0 \to \Gamma_Z(X,\x{I}^0) \to \Gamma_Z(X,\x{I}^1) \to \cdots  
$$ 
Now 
$$
0 \to \x{F}|_W \to \x{I}^0|_W \to  \x{I}^1|_W \to \cdots
$$
is an injective resolution in $\mathfrak{S}h(W)$ and so $H^p_Z(W,\x{F}|_W)$ are the cohomology objects of the complex
$$
0 \to \Gamma_Z(W,\x{I}^0|_W) \to \Gamma_Z(W,\x{I}^1|_W) \to \cdots  
$$ 

On the other hand, for every sheaf $\x{G}$ on $X$ we have a natural functorial map
$\alpha \colon \Gamma_Z(X,\x{G})\to \Gamma_Z(W,\x{G}|_W)$ induced by the restriction map 
$\x{G}(X)\to \x{G}(W)$. Note that the two open sets $U$ and $W$ give an open covering 
of $X$. Now if $\alpha(s)=0$, then $s=0$ because $s|_W=0$ and $s|_U=0$ hence $\alpha$ is 
injective. Moreover, if $t\in \Gamma_Z(W,\x{G}|_W)$, then $t|_{W\cap U}=0$. 
So, $t$ and the zero section in $\x{G}(U)$ glue together to give a section $s\in \x{G}(X)$. 
Since, $s|_U=0$, $s\in \Gamma_Z(X,\x{G})$, and $\alpha(s)=t$ hence $\alpha$ is surjective 
as well thus bijective which implies the theorem.\\  
\end{proof}

The following theorem is the analogue of the Mayer-Vietoris exact sequence in 
algebraic topology.\\

\begin{thm}\label{t-lc-Mayer-Vietoris}
Assume that $Z'$ is another closed subset of $X$. Then, for any sheaf $\x{F}$ 
on $X$ we have an exact sequence 
$$
\cdots \to H^p_{Z\cap Z'}(X,\x{F}) \to H^p_Z(X,\x{F})\oplus H^p_{Z'}(X,\x{F})\to 
H^p_{Z\cup Z'}(X,\x{F})\to \cdots
$$
\end{thm}
\begin{proof}
Note that in general for any closed subsets $Z\subseteq Z'$, any open subset $W$ and any sheaf 
$\x{F}$ we have an inclusion $\Gamma_Z(X,\x{F})\subseteq \Gamma_{Z'}(X,\x{F})$, and we have a 
natural map $\Gamma_Z(X,\x{F})\to \Gamma_{Z\cap W}(W,\x{F}|_W)$ which is surjective 
if $\x{F}$ is flasque and $X\setminus Z=W\setminus (Z\cap W)$. 

Now under the assumptions of the theorem, let $W=X\setminus (Z\cap Z')$. Then, 
for any sheaf $\x{I}$ on $X$ we have a commutative diagram  

$$
\xymatrix{
0\to   H^0_{Z\cap Z'}(X,\x{I}) \ar[d]\ar[r]^a & H^0_Z(X,\x{I})\oplus H^0_{Z'}(X,\x{I}) \ar[r]^{b}\ar[d]^{c} 
& H^0_{Z\cup Z'}(X,\x{I}) \ar[d]^{d} \\
 0\to  H^0_{\emptyset}(W,\x{I}|_W) \ar[r] & H^0_{Z\cap W}(W,\x{I}|_W)\oplus H^0_{Z'\cap W}(W,\x{I}|_W) \ar[r]
& H^0_{(Z\cup Z')\cap W}(W,\x{I}|_W) 
} 
$$ 
where the map $a$ sends $s$ to $(s,-s)$ and the map $b$ sends $(s,s')$ to $s+s'$ 
(similar definition for the maps in the lower row) hence making the rows exact. 

Now assume that $b$ is surjective for every injective sheaf $\x{I}$. Let 
$$
\x{I}^\bullet: \hspace{1cm} 0 \to \x{F} \to \x{I}^0 \to  \x{I}^1 \to \cdots
$$
be an injective resolution in $\mathfrak{S}h(X)$. Then, we get an exact sequence 
of complexes 
$$
0 \to \Gamma_{Z\cap Z'}(X,\x{I}^\bullet) \to \Gamma_Z(X,\x{I}^\bullet) \oplus \Gamma_{Z'}(X,\x{I}^\bullet)\to \Gamma_{Z\cup Z'}(X,\x{I}^\bullet)\to 0  
$$
whose associated long exact sequence gives the desired sequence in the theorem. 

Now we prove the surjectivity of $b$ for injective sheaves $\x{I}$. 
First, assume that in the above diagram the map in the lower row corresponding to $b$ is surjective.
 Since $\x{I}$ is flasque and since $X\setminus Z=W\setminus (Z\cap W)$, $X\setminus Z'=W\setminus (Z'\cap W)$, and $X\setminus (Z\cup Z')=W\setminus ((Z\cup Z')\cap W)$ the maps $c$ and $d$ are surjective. 
Let $t\in H^0_{Z\cup Z'}(X,\x{I})$. Then, there is $(s,s')$ such that $db(s,s')=d(t)$ hence 
$\overline{t}:=t-s-s'\in \ker d\subseteq\Gamma_{Z\cap Z'}(X,\x{I})$. So, $t=b(s+\overline{t},s')$ 
which means that $b$ is surjective.

Finally, we prove that  in the above diagram the map in the lower row corresponding to $b$ is an isomorphism. To do that 
we may simply assume that $W=X$ and $Z\cap Z'=\emptyset$.
Then, $ H^0_{\emptyset}(X,\x{I})=0$ which implies injectivity. 
Moreover, if $i\colon Z\to X$, $i'\colon Z'\to X$, and $e\colon Z\cup Z'\to X$ 
are the inclusions, then $e^{-1}\x{O}_X\simeq i^{-1}\x{O}_X\oplus {i'}^{-1}\x{O}_X$ 
where as usual $\x{O}_X$ is the constant sheaf on $X$ associated to $\Z$. Therefore, 
$e_*e^{-1}\x{O}_X\simeq i_*i^{-1}\x{O}_X\oplus i'_*{i'}^{-1}\x{O}_X$ and 
$$
\Hom_{\x{O}_X}(e_*e^{-1}\x{O}_X,\x{I})\simeq \Hom_{\x{O}_X}(i_*i^{-1}\x{O}_X,\x{I})
\oplus \Hom_{\x{O}_X}(i'_*{i'}^{-1}\x{O}_X,\x{I})
$$
which implies the claim.\\
\end{proof}

Cohomology with support can be quite useful for example in the context of \'etale cohomology.
 Suppose that 
$X$ is a smooth curve over an algebraically closed field, $Z$ is a finite set of closed points 
in $X$, and $U=X\setminus Z$. Suppose that $\x{F}$ is a sheaf on the \'etale site of $X$ 
and suppose that we know how to compute $H^p(U_{et},\x{F}|_U)$. Then using a fact similar to Theorem \ref{t-lc-exact-sequence} the study of the cohomology groups $H^p(X_{et},\x{F})$ can be reduced to the 
study of the cohomology groups $H^p(U_{et},\x{F}|_U)$ and $H^p_Z(X_{et},\x{F})$. For the latter, using a fact similar to Theorem 
\ref{t-lc-Mayer-Vietoris} we reduce the problem to the case when $Z$ is a single point.
 Here one may use another fact similar to 
 Theorem \ref{t-lc-excision} to replace $X$ with any \'etale
neighborhood $V$ of $Z$. In fact, we can write $H^p_Z(X_{et},\x{F})=\varinjlim H^p_Z(V_{et},\x{F}|_V)$
where $V$ runs through the \'etale neighborhoods of $Z$. In this way one can use a 
theorem of Grothendieck which says that 
$$
\varinjlim H^p_Z(V_{et},\x{F}|_V)=H^p_Z(\varprojlim V_{et},\x{F})
$$
But $\varprojlim V=\Spec \x{O}^h_{X,Z}$, that is, the Henselisation of the local ring at $Z$ which is rather well-known. For more information see Milne [\ref{Milne}, Theorem 14.3, and the proof of Poincar\'e duality 14.7].\\

\begin{rem}\label{r-lc-sheaves}
Let $X$ be a topological space and $Z$ a locally closed subset, i.e., there is an open subset 
$V$ of $X$ such that $Z$ is a closed subset of $V$. For any sheaf $\x{F}$ on $X$, one may 
define $\Gamma_Z(X,\x{F}):=\Gamma_Z(V,\x{F}|_V)$. It is easy to see that this definition 
does not depend on the choice of $V$. 

Now if $i\colon Z\to X$ is the inclusion map, then we can define a presheaf on $Z$ by 
assigning to an open subset $W$ of $Z$ the group $\Gamma_W(X,\x{F})$. It turns out that 
this is a subsheaf of $i^{-1}\x{F}$ which is denoted by $i^!\x{F}$. Note that if $Z$ is an 
open subset, then $i^!\x{F}=\x{F}|_Z$. On the other hand, by assigning to each open 
subset $U$ of $X$ the group $\Gamma_{Z\cap U}(U,\x{F}|_U)$ we obtain a sheaf on $X$ 
denoted by $\underline{\Gamma}_Z(\x{F})$. One can immediately check that 
$i_*i^!\x{F}=\underline{\Gamma}_Z(\x{F})$. 

When $Z$ is a closed subset 
$\underline{\Gamma}_Z(\x{F})\simeq \xHom_{\x{O}_X}(i_*i^{-1}\x{O}_X,\x{F})$ 
where $\x{O}_X$ is the constant sheaf associated to $\Z$. Moreover, in this case 
$i^!$ is a right adjoint of $i_*$, that is, for any sheaf $\x{F}$ on $X$ and 
any sheaf $\x{G}$ on $Z$ we have a functorial isomorphism 
$$
\Hom(i_*\x{G},\x{F})\simeq \Hom(\x{G},i^!\x{F})
$$

For more details see Grothendieck's SGA 2 [\ref{SGA2}, expos\'e I].\\

\end{rem}

\begin{rem}
Let $f\colon (X,\x{O}_X)\to (Y,\x{O}_Y)$ be a morphism of ringed spaces, and let $Z$ 
be a closed subset of $X$. 
One can use the functor $\underline \Gamma_Z(-)$ to define an even more general 
ext sheaf. More precisely, for any $\x{O}_X$-module $\x{F}$ the functor 
$$
f_*\underline \Gamma_Z(\xHom_{\x{O}_X}(\x{F},-))\colon \mathfrak{M}(X) \to \mathfrak{S}h(Y)
$$
is left exact as it is the composition of three left exact functors. The right derived 
functors of this functor are denoted by $\xExt^p_{f,Z}(\x{F},-)$. When $f$ is the identity 
map or the constant map to a single point, these functors are studies in 
Grothendieck's SGA 2 [\ref{SGA2}, expos\'e VI].\\
\end{rem}

\section{Local cohomology} 

\begin{defn}
Let $A$ be a ring, $I$ an ideal, and $M$ an $A$-module. We define the local 
cohomology groups of $M$ with respect to $I$ as $H^p_I(M):=H^p_{V(I)}(\Spec A, \tilde{M})$.\\ 
\end{defn}

\begin{thm}\label{t-lc-description} 
If $A$ is Noetherian then we have the following properties:\\
 
$(1)$ $H^0_I(M)=\{m\in M\mid I^lm=0 ~\mbox{for some $l$}\}$,

$(2)$ $H^p_I(-)$ are the right derived functors of the left exact functor\\ 
$$
H^0_I(-)\colon \mathfrak{M}(A)\to \mathfrak{M}(A)$$ 
from the category of $A$-modules to itself,   

$(3)$ $H^0_I(H^p_I(M))=H^p_I(M)$ for every $p$.\\

\end{thm} 
\begin{proof}
Put $X=\Spec A$, $Z=V(I)$, and $U=X\setminus Z$. (1) By definition 
$$
H^0_I(M)=\ker H^0(X,\tilde{M})\to H^0(U,\tilde{M}|_U)
$$
If $I=\langle b_1,\dots,b_r \rangle$, then $U=D(b_1)\cup \cdots \cup D(b_r)$ where 
$D(b_i)=X\setminus V(b_i)$. Thus, 
$$
H^0_I(M)=\bigcap_i\ker H^0(X,\tilde{M})\to H^0(D(b_i),\tilde{M}|_{D(b_i)})
$$
Since $H^0(X,\tilde{M})=M$ and $H^0(D(b_i),\tilde{M}|_{D(b_i)})=M_{b_i}$, 
$$
\ker H^0(X,\tilde{M})\to H^0(D(b_i),\tilde{M}|_{D(b_i)})=\{m\in M\mid b_i^lm=0 ~\mbox{for some $l$}\}
$$
which implies that 
$$
H^0_I(M)=\{m\in M\mid I^lm=0 ~\mbox{for some $l$}\}
$$\\

(2) Let 
$$
 0 \to M \to {L}^0 \to  {L}^1 \to \cdots
$$
be an injective resolution in $\mathfrak{M}(A)$. Then, 
$$
 0 \to \tilde{M} \to \tilde{L}^0 \to  \tilde{L}^1 \to \cdots
$$
is a flasque resolution hence in view of Exercise \ref{exe-flasque-cohomology-support} 
the cohomology objects of the complex 
$$
 0 \to \Gamma_Z(X,\tilde{L}^0) \to  \Gamma_Z(X,\tilde{L}^1) \to \cdots
$$
are the cohomology groups $H^p_I(M)=H^p_Z(X,\tilde{M})$. The latter 
complex coincides with 
$$
 0 \to H^0_I({L}^0) \to  H^0_I({L}^1) \to \cdots
$$
whose cohomology objects are those given by the right derived functors of 
$H^0_I(-)$.\\

(3) This follows immediately from (1) and the complexes in (2). In fact, 
$H^p_I(M)$ is the quotient of some submodule of $H^0_I({L}^p)$. By (1), every element $t\in H^0_I({L}^p)$ 
satisfies $I^lt=0$ for some $l$ so the same is true for every element of 
$H^p_I(M)$ which implies the equality $H^0_I(H^p_I(M))=H^p_I(M)$.\\
\end{proof} 

\begin{thm}\label{t-lc-ext}
Let $A$ be a Noetherian ring. Then, for any $A$-module $M$ and any $p$ we have natural isomorphisms 
$$
\varinjlim \Ext^p_A({A}/{I^l},M) \simeq H^p_I(M) 
$$
\end{thm}
\begin{proof}
For every $l$, the surjective map $A/I^{l+1}\to A/I^l$ gives an injective map 
$$
\Hom_A(A/I^l,M)\to  \Hom_A(A/I^{l+1},M)
$$ 
making $\{\Hom_A(A/I^l,M)\}_{l\in \N}$
into a directed set of $A$-modules. For any $l$, we define a map $\phi_l\colon \Hom_A(A/I^l,M)\to H^0_I(M)$ 
which sends $\alpha\colon A/I^l\to M$ to $\alpha(1)$. This determines a 
directed system of $A$-homomorphisms hence by the universal property of direct 
limits we get an $A$-homomorphism $\phi \colon \varinjlim \Hom_A({A}/{I^l},M) \to H^0_I(M)$ 
out of the $\phi_l$. 

If $I^lm=0$ for some $m\in H^0_I(M)$, then the map $\alpha\colon A/I^l\to M$ which sends 
$a$ to $am$ is well-defined with $\alpha(1)=m$ hence $m$ belongs to the image of $\phi_l$.  
Thus, $\phi$ is surjective. On the other hand, the $\phi_l$ are all injective hence 
$\phi$ is injective as well. Therefore, $\phi$ is an isomorphism of $A$-modules. 

Now let 
$$
 0 \to M \to {L}^0 \to  {L}^1 \to \cdots
$$
be an injective resolution in $\mathfrak{M}(A)$. Then the complex 
$$
 0 \to H^0_I({L}^0)\simeq \varinjlim\Hom_A(A/I^{l},L^0) \to  H^0_I({L}^1)\simeq \varinjlim\Hom_A(A/I^{l},L^1) \to \cdots
$$
which calculates the local cohomology groups $H^p_I(M)$, is the direct limit of the complexes 
$$
 0 \to \Hom_A(A/I^{l},L^0) \to \Hom_A(A/I^{l},L^1) \to \cdots
$$
which implies the result as $\Ext^p_A({A}/{I^l},-)$ are the right 
derived functors of $\Hom_A({A}/{I^l},-)$.
\end{proof}

\begin{rem}
There is a global version of the previous theorem which goes as follows. 
 Let $(X,\x{O}_X)$ be a ringed space, $Z$ be a closed subset and $i\colon Z\to X$ be the 
inclusion map. We have an exact sequence 
$$
0 \to j_!j^*\x{O}_X \to \x{O}_X \to {i}_*i^{-1}\x{O}_X \to 0
$$ 
of $\x{O}_X$-modules where $j\colon U=X\setminus Z\to X$ is the inclusion map.
For any $\x{O}_X$-module $\x{F}$, we have 
$\Gamma_Z(X,\x{F})=\Hom_{\x{O}_X}({i}_*i^{-1}\x{O}_X, \x{F})$ hence the right derived 
functors of $\Gamma_Z(X,-)\colon \mathfrak{M}(X)\to \mathfrak{Ab}$ coincide with 
those of $\Hom_{\x{O}_X}({i}_*i^{-1}\x{O}_X,-)$. On the other hand, according to 
Exercise \ref{exe-flasque-cohomology-support},  the right derived 
functors of $\Gamma_Z(X,-)\colon \mathfrak{M}(X)\to \mathfrak{Ab}$ coincide with 
the restriction of the the right derived 
functors of $\Gamma_Z(X,-)\colon \mathfrak{S}h(X)\to \mathfrak{Ab}$. This then implies 
that $H^p_Z(X,\x{F})\simeq \Ext^p_{\x{O}_X}({i}_*i^{-1}\x{O}_X, \x{F})$ for any 
$\x{O}_X$-module $\x{F}$.

Now assume that $X$ is a Noetherian scheme, $Z$ is a closed subset with a subscheme structure 
corresponding to a quasi-coherent ideal sheaf $\x{I}_Z$. For any $l$, let $Z_l$ be the closed subscheme defined by $\x{I}_Z^l$.
Note that if $m\ge l$, we have a natural surjective morphism 
$\x{O}_{Z_m}\to \x{O}_{Z_l}$. This makes the sheaves $\x{O}_{Z_l}$ 
into an inverse system. 
We also have natural morphisms $i_*i^{-1}\x{O}_X\to \x{O}_{Z_l}$ for every $l$.
So, for every $\x{O}_X$-module $\x{F}$ and for every $l$ and $p$ we get a map 
$$
\Ext^p_{\x{O}_X}(\x{O}_{Z_l}, \x{F})\to \Ext^p_{\x{O}_X}({i}_*i^{-1}\x{O}_X, \x{F})\simeq H^p_Z(X,\x{F})
$$
which give rise to a map 
$$
\phi \colon \varinjlim_{l}\Ext^p_{\x{O}_X}(\x{O}_{Z_l}, \x{F})\to H^p_Z(X,\x{F})
$$

When $\x{F}$ is quasi-coherent the map $\phi$ is an isomorphism by Grothendieck's SGA 2 [\ref{SGA2}, expos\'e II, Theorem 6].\\
\end{rem}

\section{Local cohomology and depth}

Assume that $A$ is a Noetherian ring, $M$ is finitely generated $A$-module, and $IM\neq M$ 
for some ideal $I$ of $A$.  
An $M$-regular sequence in $I$ is a sequence $a_1,\dots, a_n$ of elements of $I$ such that 
 $a_i$ is not a zero-divisor of $\frac{M}{\langle a_1,\dots,a_{i-1}\rangle M}$ for any $i=1, \dots,n$. It is well-known that any two maximal $M$-regular sequences in $I$ have the same length (cf. [\ref{Eisenbud}, page 429]). One then defines the depth of $M$ with respect to $I$, denoted by $\depth_IM$, to be the common length of maximal $M$-regular sequences in $I$. 

The notion of depth is closely related to local cohomology as the next theorem 
shows.\\

\begin{thm}[Characterisation of depth]\label{t-lc-depth}
Assume that $A$ is a Noetherian ring, $M$ is a finitely generated $A$-module, and $I$ is an
 ideal of $A$ such that $IM\neq M$. Then, $\depth_IM\ge n$ if and only if $H^p_I(M)=0$ for any $p<n$.  
\end{thm} 
\begin{proof}
We use induction on $n$. The case $n=0$ is trivial so we may assume that $n>0$ and that 
the theorem holds for $n-1$. 

Assume that $\depth_IM=d$ and that $d\ge n$. Let $a_1,\dots, a_n$ be an $M$-regular sequence 
in $I$. Consider the short exact sequence 
$$
0 \to M\to M \to \frac{M}{a_1M} \to 0
$$
where $M\to M$ is given by multiplication by $a_1$. Then, $\depth_I(\frac{M}{a_1M})\ge n-1$ since $a_2,\dots, a_n$ is an $\frac{M}{a_1M}$-regular sequence in $I$. So, by induction 
 $H^p_I(\frac{M}{a_1M})=0$ for any $p<n-1$. By Theorem \ref{t-lc-description}, we have 
 a long exact sequence 
$$
\cdots \to H^p_I(M)\to H^p_I(M) \to H^p_I(\frac{M}{a_1M}) \to H^{p+1}_I(M) \to \cdots
$$
which implies that $\alpha\colon H^{n-1}_I(M)\to H^{n-1}_I(M)$ is injective. But 
$\alpha$ is just the map given by multiplication by $a_1$, and for every element $t \in H^{n-1}_I(M)$ 
we have $a_1^lt=0$ for some $l$ by Theorem \ref{t-lc-description}. Thus, $\alpha^l(t)=0$ for such a $t$ hence 
$\alpha$ can be injective only if  $H^{n-1}_I(M)=0$. So, $H^{p}_I(M)=0$ for any $p<n-1$.

 Conversely, assume that  $H^{p}_I(M)=0$ for any $p<n$. A well-known theorem in 
 commutative algebra states that the set of zero-divisors of $M$ together with 
 $0$ is the union of the finitely many associated primes of $M$. Now $I$ cannot 
 be a subset of that union otherwise by elementary properties of primes ideals 
 there is an associated prime $P$ of $M$ such that $I\subseteq P$. So, 
 there is a nonzero element $m\in M$ such that $Im=0$. But this contradicts the 
 assumption $H^{0}_I(M)=0$. Therefore,  
 there is an element $a_1\in I$ which is not a zero-divisor of $M$.
 
 The above long exact sequence implies that $H^p_I(\frac{M}{a_1M})=0$ for any $p<n-1$. 
 Thus, by induction $\depth_I(\frac{M}{a_1M})\ge n-1$ hence there is an $\frac{M}{a_1M}$-regular 
 sequence $a_2,\dots,a_{n}$ in $I$. But $a_1$ is already an $M$-regular sequence 
 so $a_1,\dots,a_{n}$ form an $M$-regular sequence in $I$ which in particular 
 implies that $\depth_IM\ge n$.\\
\end{proof} 
 
Depth actually has a geometric interpretation. Indeed, 
 using the exact sequence of Theorem \ref{t-lc-exact-sequence} we see that if $\depth_IM\ge n$ 
then the map $H^p(X,\tilde{M})\to H^p(U,\tilde{M}|_U)$ is injective for every $p<n$ 
and an isomorphism for every $p<n-1$.  
In other words, removing $Z$ from $X$ does not change the cohomology groups up to 
$n-2$ when $\depth_IM=n$.\\


\section{Local duality}

\begin{rem}
Let $A$ be a Noetherian ring. Let $L$ be an injective $A$-module. Then, 
it is well-known that $L$ can be written as a direct sum 
$$
L=\bigoplus_t E(A/P_t)
$$
where each $P_t$ is a prime ideal of $A$ and $E(A/P_t)$ is the injective 
envelope (=hull) of the $A$-module $A/P_t$. Moreover,  
for any ideal $I$ of $A$, we have 
$$
H^0_I(L)=\bigoplus_{I\subseteq P_t} H^0_I(E(A/P_t))=\bigoplus_{I\subseteq P_t} E(A/P_t)
$$

Now further assume that $A$ is local with maximal ideal $\mathfrak{m}$. Then, 
$$
H^0_{\mathfrak{m}}(L)=\bigoplus_{\mathfrak{m}=P_t} E(A/P_t)
$$
Further assume that $M$ is a finitely generated $A$-module. Then, using injective 
envelopes we can find an injective resolution 
$$
 0 \to M \to {L}^0 \to  {L}^1 \to \cdots
$$
such that in each $L^i$ only a finite number of copies of $E:=E(A/\mathfrak{m})$ 
occur. In particular, when we apply the functor $H^0_{\mathfrak{m}}(-)$ 
to this sequence each term would be just a finite direct sum of copies of 
$E(A/\mathfrak{m})$. On the other hand, it is known that $E(A/\mathfrak{m})$ 
is Artinian hence so the modules $H^p_{\mathfrak{m}}(M)$ would also be 
Artinian. Finally, it is well-known that for any Artinian $A$-module $N$ we have a natural isomorphism $N\simeq \Hom_A(\Hom_A(N,E),E)$. 
For more information see [\ref{CM-rings}, section 3.2].\\
\end{rem}

\begin{thm}
Let $A$ be a regular local Noetherian ring of dimension $d$ with maximal ideal 
$\mathfrak{m}$. Then, $H^d_\mathfrak{m}(A)\simeq E(A/\mathfrak{m})$ where 
$E(A/\mathfrak{m})$ is the injective envelope of $A/\mathfrak{m}$.
\end{thm}
\begin{proof}
See Grothendieck's SGA 2 [\ref{SGA2}, expos\'e IV, Theorem 4.7, Theorem 5.4].
\end{proof}

We will use this theorem only to the effect that $H^d_{\mathfrak{m}}(A)$ 
is an injective $A$-module. 

The following theorem is a local version of the duality theorem in the next 
chapter.\\

\begin{thm}[Local duality]\label{t-lc-duality}
Let $A$ be a regular local Noetherian ring of dimension $d$ with maximal ideal $\mathfrak{m}$.
Then, for any finitely generated $A$-module $M$ we have an isomorphism 
$$
H^p_\mathfrak{m}(M)\simeq \Hom_A(\Ext^{d-p}_A(M,A),H^d_\mathfrak{m}(A))
$$
\end{thm}
\begin{proof}
We have a natural pairing 
$$
\Hom_A(A/\mathfrak{m}^l,M)\times \Hom_A(M,A)\to \Hom_A(A/\mathfrak{m}^l,A)
$$ 
which induces a Yoneda pairing 
$$
\Ext^p_A(A/\mathfrak{m}^l,M)\otimes_A \Ext^{d-p}_A(M,A)\to \Ext^d_A(A/\mathfrak{m}^l,A)
$$ 
One way to imagine the pairing is that an element $\alpha\in \Ext^p_A(A/\mathfrak{m}^l,M)$ 
corresponds to an exact sequence 
$$
 0 \to M \to {L}^1 \to  {L}^2 \to \dots \to L^p\to A/\mathfrak{m}^l \to 0
$$
and an element $\beta\in \Ext^{d-p}_A(M,A)$ corresponds to an exact sequence 
$$
 0 \to A \to {N}^1 \to  {N}^2 \to \dots \to N^{d-p}\to M \to 0
$$
giving rise to an exact sequence 
$$
 0 \to A \to {N}^1 \to  {N}^2 \to \dots \to N^{d-p}\to {L}^1 \to  {L}^2 \to \dots \to L^p\to A/\mathfrak{m}^l \to 0
$$
hence an element $\gamma \in \Ext^d_A(A/\mathfrak{m}^l,A)$ (see [\ref{Eisenbud}, Exercise A3.27]).

Now the Yoneda pairing induces a map 
$$
\psi_l\colon \Ext^p_A(A/\mathfrak{m}^l,M)\to \Hom_A(\Ext^{d-p}_A(M,A), \Ext^d_A(A/\mathfrak{m}^l,A))
$$ 
and since $\Ext^{d-p}_A(M,A)$ is finitely generated over $A$ by Remark-Theorem 
\ref{r-ext-coherent}, by taking direct limits we obtain a map
$$
\psi\colon \varinjlim\Ext^p_A(A/\mathfrak{m}^l,M)\to \Hom_A(\Ext^{d-p}_A(M,A), \varinjlim\Ext^d_A(A/\mathfrak{m}^l,A))
$$ 
Using Theorem \ref{t-lc-ext} we can consider $\psi$ as a map
$$
\psi\colon H^p_\mathfrak{m}(M)\to D^p(M):=\Hom_A(\Ext^{d-p}_A(M,A), H^d_\mathfrak{m}(A))
$$ 

We prove that $\psi$ is an isomorphism. First assume that $p>d$. Then, the exact 
sequence of Theorem \ref{t-lc-exact-sequence} and Grothendieck theorem on vanishing of 
cohomology imply that $H^p_{\mathfrak{m}}(M)=0$ if $p>d+1$. However, the vanishing 
$H^{d+1}_\mathfrak{m}(M)=0$ is more subtle: since $A$ is a regular local ring, 
$\pd(A/\mathfrak{m}^l)\le \dim A=d$ hence by Theorem \ref{t-proj-dimension} $\Ext^{d+1}_A(A/\mathfrak{m}^l,M)=0$ 
thus $H^{d+1}_\mathfrak{m}(M)=\varinjlim\Ext^{d+1}_A(A/\mathfrak{m}^l,M)=0$. 
On the other hand, since $d-p<0$, $\Ext^{d-p}_A(M,A)=0$ hence $\psi$ 
is an isomorphism in this case. 

Now assume that $p=d$. Since $M$ is finitely generated, there is an exact sequence 
$A^r\to A^s\to M$ giving rise to an exact sequence 
$$
0 \to \Ext^{0}_A(M,A) \to \Ext^{0}_A(A^s,A) \to \Ext^{0}_A(A^r,A) 
$$
hence a commutative diagram 
$$
\xymatrix{
H^d_\mathfrak{m}(A^r) \ar[r]\ar[d] & H^d_\mathfrak{m}(A^s)\ar[r]\ar[d] & H^d_\mathfrak{m}(M)\ar[r]\ar[d]& 0\\
D^d(A^r) \ar[r] & D^d(A^s) \ar[r] & D^d(M) \ar[r]& 0
}
$$
This reduces the problem to the case $M=A^r$. This case in turn follows from the facts 
that $H^d_\mathfrak{m}(A^r)\simeq H^d_\mathfrak{m}(A)^r$, $\Ext^{0}_A(A^r,A)\simeq A^r$, and 
$\Hom_A(A^r, H^d_\mathfrak{m}(A))\simeq H^d_\mathfrak{m}(A)^r$.

Now assume that $p< d$ and assume that the theorem holds for values more than 
$p$. Again since $M$ is finitely generated, there is an exact sequence 
$0\to K\to A^s\to M$ giving rise to an exact sequence 
$$
\cdots \to \Ext^{d-p-1}_A(K,A) \to \Ext^{d-p}_A(M,A) \to \Ext^{d-p}_A(A^s,A) \to \Ext^{d-p}_A(K,A)  \to \cdots 
$$
hence a commutative diagram 
$$
\xymatrix{
\cdots \ar[r] & H^p_\mathfrak{m}(K) \ar[r]\ar[d] & H^p_\mathfrak{m}(A^s)\ar[r]\ar[d] & H^p_\mathfrak{m}(M)\ar[r]\ar[d]& H^{p+1}_\mathfrak{m}(K) \ar[d]\\
\cdots \ar[r] & D^p(K) \ar[r] & D^p(A^s) \ar[r] & D^p(M) \ar[r]& D^{p+1}(K) 
}
$$ 
By Theorem \ref{t-lc-depth}, $H^p_\mathfrak{m}(A^s)\simeq H^p_\mathfrak{m}(A)^s=0$ 
since $\depth_\mathfrak{m}A=\dim A=d$. On the other hand, $\Ext^{d-p}_A(A^s,A)=0$ 
since $A$ is projective. So the result for $M$ follows from the above diagram.\\
\end{proof}

\section*{Exercises}
\begin{enumerate}
\item\label{exe-injective-restriction-tensor} Let $(X,\x{O}_X)$ be a dinged space, and let $\x{I}\in\mathfrak{M}(X)$ be an 
injective object. Show that $\x{I}|_U$ is an injective object of $\mathfrak{M}(X)$ 
for any open subset $U\subseteq X$. Moreover, show that $\x{I}\otimes \x{L}$ is also 
injective for any locally free sheaf $\x{L}\in\mathfrak{M}(X)$ of finite rank.\\ 

\item\label{exe-flasque-cohomology-support} Let $(X,\x{O}_X)$ be a ringed space and 
$Z$ a closed subset of $X$. Show that if $\x{F}$ is a flasque sheaf on $X$, then
$H^p_Z(X,\x{F})=0$ for any $p>0$. Deduce that the right derived functors of the 
left exact functor $\Gamma_Z(X,-)\colon \mathfrak{M}(X)\to \mathfrak{Ab}$ coincide 
with the functors $H^p_Z(X,-)$ restricted to $\mathfrak{M}(X)$.\\

\item Let $X$ be a topological space and $Z\subseteq Z'$ closed subsets and 
$W=Z'\setminus Z$. In view of Remark \ref{r-lc-sheaves}, show that there is a left  
exact sequence 
$$
0\to \underline{\Gamma}_Z(\x{F})\to \underline{\Gamma}_{Z'}(\x{F})\to \underline{\Gamma}_W(\x{F})\to 0
$$
Moreover, show that if $\x{F}$ is flasque, then the sequence is also exact on the right.\\

\item Let $X, Z,Z',W$ be as in the previous exercise. Show that 
$\Gamma_W(X,-)\colon \mathfrak{S}h(X)\to \mathfrak{Ab}$ is a left exact functor (see Remark \ref{r-lc-sheaves}). 
Its right derived functors are denoted by $H^p_W(X,-)$.  
Show that for any sheaf $\x{F}$ on $X$ there is an exact sequence
$$
\cdots \to H^p_Z(X,\x{F})\to H^p_{Z'}(X,\x{F})\to H^p_W(X,\x{F})\to H^{p+1}_Z(X,\x{F}) \to \cdots
$$\\

\item Let $X$ be a topological space and $\x{F}$ a sheaf on $X$. Show that $\x{F}$ is 
flasque if and only if $H^p_Z(X,\x{F})=0$ for any $p>0$ and any locally closed subset $Z$.\\

\end{enumerate}

\chapter{\tt Relative duality}

In this chapter the relative duality theorem is proved based on ideas and formulation of Kleiman [\ref{Kleiman}] which 
is in turn based on Grothendieck duality theory [\ref{Hartshorne-duality}] (see also FGA [\ref{FGA}, section 2]). We do not use derived categories but it comes with the price of putting certain 
strong assumptions on the sheaves involved. On the other hand, the proof is relatively short 
and neat. The main ingredient is a spectral sequence which is discussed in 3.2.\footnote{In this chapter, flatness and base change is used which will be discussed in Chapter 3. The 
theorem on the formal functions is also used for which we refer to Hartshorne [\ref{Hartshorne}, III, \S 11].}

Let $f\colon X\to Y$ be a morphism of schemes and let $\x{F}$ be an $\x{O}_X$-module. 
Then, for any point $y\in Y$ there is a natural base change map 
$$
(R^pf_*\x{F})\otimes k(y) \to H^p(X_y,\x{F}_y)
$$
where $k(y)$ is the residue field at $y$ and $\x{F}_y$ is the pullback of $\x{F}$ 
on the fibre $X_y$. We say that $R^pf_*\x{F}$ commutes with base change if the 
above map is an isomorphism for every $y$.\\

\section{Duality for $\PP^n_Y$}
We first prove the duality theorem on the projective space. The general case will be 
reduced to this case.\\

\begin{thm}\label{duality-projective-space}  
Let $Y$ be a Noetherian scheme, and $\pi \colon \PP^n_Y\to Y$ the projection. Then, for any $p$ and any coherent $\x{O}_X$-module $\x{F}$
 and coherent $\x{O}_Y$-module $\x{G}$ there are functorial morphisms 
$$
\theta^p\colon \xExt^p_\pi(\x{F},\omega_\pi\otimes \x{G}) \to \xHom_{\x{O}_Y}(R^{n-p}\pi_*\x{F},\x G)
$$  
where $\omega_\pi=\x{O}_{\PP^n_Y}(-n-1)$. For $p=0$, the morphism is always an isomorphism. Moreover, for a fixed $m$ and $\x{F}$ flat over $Y$, $R^{n-p}f_*\x{F}$  commutes with base change for every $p\le m$ if and only if $\theta^p$ is an isomorphism for every $\x{G}$ and every $p\le m$. 
\end{thm}
\begin{proof}
First, it is well-known that we have a natural isomorphism $R^n\pi_*\omega_\pi\simeq \x{O}_Y$ 
(cf. [\ref{Hartshorne}, III, Exercise 8.4]). There is an open covering $\mathcal{U}$ of $\PP^n_Y$ by $n+1$ subsets 
such that for any open affine $U=\Spec A\subseteq Y$, the covering restricts to the 
standard covering of $\PP^n_A$ by the $n+1$ copies of $\A^n_A$. Let 
$$
\x{C}^\bullet:\hspace{1cm} 0 \to \omega_\pi \to \x{C}^0(\mathcal{U},\omega_\pi) \to \cdots \to \x{C}^{n-1}(\mathcal{U},\omega_\pi) \to \x{C}^n(\mathcal{U},\omega_\pi) \to 0
$$
be the associated \v{C}ech resolution. From this we get a complex 
$$
0 \to \pi_*\omega_\pi \to \pi_*\x{C}^0(\mathcal{U},\omega_\pi) \to \cdots \to \pi_*\x{C}^{n-1}(\mathcal{U},\omega_\pi) \to \pi_*\x{C}^n(\mathcal{U},\omega_\pi) \to 0
$$
inducing an exact sequence 
$$
 \pi_*\x{C}^{n-1}(\mathcal{U},\omega_\pi) \to^{d^{n-1}} \pi_*\x{C}^n(\mathcal{U},\omega_\pi) \to R^n\pi_*\omega_\pi \to 0
$$
because on any open affine $U=\Spec A\subseteq Y$ the module of sections of $\rm{coker}~ d^{n-1}$ 
is $\check{H}^n(\mathcal{U}|_{\pi^{-1}U}, \omega_{\pi}|_{\pi^{-1}U})$ which is 
isomorphic to $R^n\pi_*\omega_\pi(U)$.\footnote{Strictly speaking, we need to take an injective resolution $\x{I}^\bullet:\hspace{1cm}  0\to \omega_\pi \to \x{I}^0\to \cdots$ from which we get a morphism of complexes $\x{C}^\bullet \to \x{I}^\bullet$ and finally a morphism $\rm{coker}~d^{n-1}\to R^n\pi_*\omega_\pi$ which is an isomorphism as verified locally.}

On the other hand, we have the \v{C}ech resolution of 
$\omega_\pi\otimes \x{G}$:
$$
0 \to \omega_\pi\otimes \x{G} \to \x{C}^0(\mathcal{U},\omega_\pi\otimes \x{G}) \to \cdots \to \x{C}^{n-1}(\mathcal{U},\omega_\pi\otimes \x{G}) \to \x{C}^n(\mathcal{U},\omega_\pi\otimes \x{G}) \to 0
$$ 
and similarly we get an exact sequence  
$$
 \pi_*\x{C}^{n-1}(\mathcal{U},\omega_\pi\otimes \x{G}) \to \pi_*\x{C}^n(\mathcal{U},\omega_\pi\otimes \x{G}) \to R^n\pi_*(\omega_\pi\otimes \x{G}) \to 0
$$

On the other hand, we have natural morphisms 
$\pi_*\x{C}^{p}(\mathcal{U},\omega_\pi)\otimes \x{G}\to \pi_*\x{C}^{p}(\mathcal{U},\omega_\pi\otimes \x{G})$ which are isomophisms as can be checked locally hence 
in view of the exact sequence 
$$
 \pi_*\x{C}^{n-1}(\mathcal{U},\omega_\pi)\otimes \x{G} \to \pi_*\x{C}^n(\mathcal{U},\omega_\pi) \otimes \x{G} \to R^n\pi_*\omega_\pi \otimes \x{G}\to 0
$$
we get an isomorphism\footnote{Alternatively, one can use the fact that there is a 
natural morphism $R^n\pi_*(\omega_\pi)\otimes \x{G}\to R^n\pi_*(\omega_\pi\otimes \x{G})$ 
(see Grothendieck's EGA III [\ref{EGA}], Chapter 0, 12.2.2)
and then try to prove that this is an isomorphism by locally taking an exact sequence 
$\x{E}'\to \x{E}\to \x{G}\to 0$ where $\x{E}',\x{E}$ are free.}  

$$
R^n\pi_*(\omega_\pi)\otimes \x{G}\simeq R^n\pi_*(\omega_\pi\otimes \x{G})
$$

Now any morphism $\x{F}\to \omega_\pi\otimes \x{G}$ over any open 
set $U$ in $Y$ induces morphisms 
$$
R^n\pi_*\x{F}\to R^n\pi_*(\omega_\pi\otimes \x{G})\simeq R^n\pi_*(\omega_\pi)\otimes \x{G}
\simeq \x{G}
$$ 
over $U$ hence giving a map 
$$
\sigma \colon \pi_*\xHom_{\x{O}_{\PP^n_Y}}(\x{F}, \omega_\pi\otimes \x{G})\to \xHom_{\x{O}_Y}(R^n\pi_*\x{F},\x{G})
$$
To see that this is an isomophism we may assume that $Y=\Spec A$. In that case, 
there is an exact sequence $\x{L}'\to \x{L} \to \x{F} \to 0$ where $\x{L}',\x{L}$ 
are finite direct sums of sheaves of the form ${\x{O}_{\PP^n_Y}(-l)}$ for 
sufficiently large $l$. We get a commutative diagram 
{\tiny{
$$
\xymatrix{
0 \to \Hom_{\x{O}_{\PP^n_Y}}(\x{F}, \omega_\pi\otimes \x{G}) \ar[d]\ar[r] &
\Hom_{\x{O}_{\PP^n_Y}}(\x{L}, \omega_\pi\otimes \x{G}) \ar[d]\ar[r] &
\Hom_{\x{O}_{\PP^n_Y}}(\x{L}', \omega_\pi\otimes \x{G}) \ar[d]\\
0\to \Hom_A(H^n(\PP^n_Y,\x{F}),\x{G}(Y)) \ar[r]  &
\Hom_A(H^n(\PP^n_Y,\x{L}),\x{G}(Y)) \ar[r] &
\Hom_A(H^n(\PP^n_Y,\x{L}'),\x{G}(Y))
}
$$}} 
which reduces the problem to the case $\x{F}={\x{O}_{\PP^n_Y}(-l)}$ for 
sufficiently large $l$.\footnote{Note that here we have used the fact that 
 $R^{n+1}\pi_*\x{H}=0$ for any coherent sheaf $\x{H}$ which follows from 
 \v{C}ech cohomology and the fact that $\PP^n_A$ can be covered by 
 $n+1$ open affine sets.} 
 In this case, 
$\Hom_{\x{O}_{\PP^n_Y}}(\x{F}, \omega_\pi\otimes \x{G})\simeq H^0(\x{O}_{\PP^n_Y}, 
\omega_\pi\otimes \x{G}(l))$. 
Since $Y$ is affine and $\x{G}$ coherent we have an exact 
sequence $\x{H}'\to \x{H} \to \x{G} \to 0$
where the first two sheaves are free sheaves of finite rank. If $l$ is large enough we get an exact 
sequence 
$$
H^0(\x{O}_{\PP^n_Y}, \omega_\pi\otimes \x{H}'(l))\to 
H^0(\x{O}_{\PP^n_Y}, \omega_\pi\otimes \x{H}(l))\to 
H^0(\x{O}_{\PP^n_Y}, \omega_\pi\otimes \x{G}(l)) \to 0
$$
Comparing this with {\tiny
$$
\Hom_A(H^n(\PP^n_Y,\x{F}),\x{H'}(Y)) \to
\Hom_A(H^n(\PP^n_Y,\x{F}),\x{H}(Y)) \to
\Hom_A(H^n(\PP^n_Y,\x{F}),\x{G}(Y)) \to 0
$$}
implies that $\sigma$ is an isomorphism. To get the surjectivity in the latter 
sequence we have used the fact that $H^n(\PP^n_Y,\x{F})$ is a free $A$-module.

Now a Yoneda pairing as in the proof of Theorem \ref{t-lc-duality} combined 
with the presheaf description of Theorem \ref{t-ext-sheaf-associated} gives a 
sheaf version of the Yoneda pairing: 
$$
\xExt^{n-p}_\pi(\x{O}_{\PP^n_Y},\x{F})\otimes \xExt^p_\pi(\x{F},\omega_\pi\otimes\x{G})\to 
\xExt^n_\pi(\x{O}_{\PP^n_Y},\omega_\pi\otimes\x{G})
$$
and since 
$$
\xExt^n_\pi(\x{O}_{\PP^n_Y},\omega_\pi\otimes\x{G})\simeq  R^n\pi_*(\omega_\pi\otimes \x{G})
\simeq R^n\pi_*(\omega_\pi)\otimes \x{G}\simeq \x{G}
$$
and 
$$
\xExt^{n-p}_\pi(\x{O}_{\PP^n_Y},\x{F})\simeq  R^{n-p}\pi_*\x{F}
$$
we get a natural morphism 
$$
\theta^p \colon \xExt^p_\pi(\x{F},\omega_\pi\otimes\x{G}) \to \xHom_{\x{O}_Y}(R^{n-p}f_*\x{F},\x G)
$$
which coincides with $\sigma$ when $p=0$. 

Now fix a coherent $\x{F}$ flat over $Y$ and fix a number $m$. First assume that $R^{n-p}f_*\x{F}$  commutes with base change for every $p\le m$. This in particular implies that $R^{n-p}f_*\x{F}$ is locally free for every $p\le m-1$.
To prove that $\theta^p$ are isomorphisms for $p\le m$, we may assume that $Y$ is affine, 
say $\Spec A$. In that case, there is an exact sequence $0 \to \x{K}\to \x{L} \to \x{F} \to 0$ where $\x{L}$ 
is a finite direct sum of sheaves of the form ${\x{O}_{\PP^n_Y}(-l)}$ for 
sufficiently large $l$. Since $\x{F}$ and $\x{L}$ are flat over $Y$, so is $\x{K}$. 
We get a commutative diagram 
{\tiny
$$
\xymatrix{
\Ext^{p-1}_{\x{O}_{\PP^n_Y}}(\x{L}, \omega_\pi\otimes \x{G}) \ar[d]\ar[r] &
\Ext^{p-1}_{\x{O}_{\PP^n_Y}}(\x{K}, \omega_\pi\otimes \x{G}) \ar[d] \ar[r] &
\Ext^{p}_{\x{O}_{\PP^n_Y}}(\x{F}, \omega_\pi\otimes \x{G}) \ar[d] \ar[r]& 0\\
\Hom_A(H^{n-p+1}(\x{L}),\x{G}(Y)) \ar[r] &
\Hom_A(H^{n-p+1}(\x{K}),\x{G}(Y))\ar[r] &
\Hom_A(H^{n-p}(\x{F}),\x{G}(Y)) \ar[r] & 0
}
$$}
The zero's in the diagram come from the facts 
$$
\Ext^p_{\x{O}_{\PP^n_Y}}(\x{O}_{\PP^n_Y(-l)}, \omega_\pi\otimes \x{G})=
H^p({\PP^n_Y},\omega_\pi\otimes \x{G}(l))=0
$$
and 
$H^{n-p}({\PP^n_Y},{\x{O}_{\PP^n_Y}(-l)})=0$ when $p>0$ and $l\gg 0$. The upper 
row is exact by the properties of ext sheaves. However, the lower row is not in general 
exact but we prove that it is exact under our assumptions. First note that 
if $p>1$, then we get isomorphisms $R^{n-p+1}\pi_*\x{K}\simeq R^{n-p}\pi_*\x{F}$.
So,  $R^{n-p+1}\pi_*\x{K}$ is locally free for any $1<p<m$. Moreover, considering the exact sequence 
$$
(*) \hspace{1cm} 0\to R^{n-1}\pi_*\x{F} \to R^{n}\pi_*\x{K}\to R^n\pi_*\x{L}\to R^{n}\pi_*\x{F} \to 0
$$
if $m>1$ then the fact that $R^{n-1}\pi_*\x{F}$, $R^n\pi_*\x{L}$, and $R^{n}\pi_*\x{F}$ are 
all locally free implies that $R^{n}\pi_*\x{K}$ is also locally free. Therefore, 
$R^{n-p+1}\pi_*\x{K}$ commutes with base change for any $p\le m$ (if $m=1$, this is true 
for any coherent sheaf flat over $Y$).

First assume that $1<p\le m$. Then the left objects in the upper and lower rows vanish 
and the lower row also becomes exact. The result follows from induction on $m$ 
applied to the sheaf $\x{K}$. Now assume that $p=1$. In this case, the lower row in the 
diagram is again exact because $H^{n}({\PP^n_Y},\x{L})$ and $H^{n}({\PP^n_Y},\x{F})$ are projective $A$-modules which split the sequence $(*)$.
Thus, the result follows from the diagram and the case $p=0$.

Conversely, for the fixed $m$ and $\x{F}$ flat over $Y$, assume that $\theta^p$ is an isomorphism for every $\x{G}$ and every $p\le m$. We may assume that $m>0$. Let $T$ be the kernel of 
of $H^{n}({\PP^n_Y},\x{L})\to H^{n}({\PP^n_Y},\x{F})\to 0$. The assumptions imply that 
the lower row of the diagram is always exact for $p\le m$ and every coherent $\x{G}$. 
This in particular implies that $T$ is projective in the category of finitely generated 
$A$-modules. Since $T$ is itself finitely generated, it is then projective in the 
category of $A$-modules. This forces $R^n\pi_*\x{F}$ to be locally free which in turn 
implies that $R^{n-1}\pi_*\x{F}$ commutes with base change. If $m>1$ we use the 
isomorphisms $R^{n-p+1}\pi_*\x{K}\simeq R^{n-p}\pi_*\x{F}$ and induction on $m$ 
to be applied to $\x{K}$.\\
\end{proof}

\section{A spectral sequence}\label{duality-spectral-sequence} 
Let $f \colon X\to Y$ be a projective morphism of Noetherian schemes. Fix a closed immersion $e\colon X\to T$ such that $f=g e$ where 
$g\colon T\to Y$ is a projective morphism. Let $\x{F}$ be a coherent sheaf on 
$X$. For any quasi-coherent sheaf $\x{N}$ on $T$, the natural morphism 
$\x{O}_{T} \to \x{O}_X$ induces a morphism 
$$
\xHom_{\x{O}_{T}}(\x{O}_X, \x{N}) \to \xHom_{\x{O}_{T}}(\x{O}_{T}, \x{N})\simeq \x{N}
$$
which in turn induces a natural morphism 
$$
\phi \colon \xHom_{\x{O}_{T}}(\x{F},\xHom_{\x{O}_{T}}(\x{O}_X, \x{N})) \to \xHom_{\x{O}_{T}}(\x{F}, \x{N})
$$
which is an isomorphism as can be seen locally: on any open affine subscheme $U=\Spec A$ of 
$T$, $X$ is defined by an ideal $I$ and the above morphism corresponds to a homomorphism 
$$
\Hom_A(F,\Hom_A({A}/{I},N))\to \Hom_A(F,N)
$$
which is an isomorphism where $\x{F}=\tilde{F}$ and $\x{N}=\tilde{N}$ on $U$. 
Note that 
$$
\xHom_{\x{O}_{T}}(\x{F},\xHom_{\x{O}_{T}}(\x{O}_X, \x{N})) \simeq 
\xHom_{\x{O}_X}(\x{F},\xHom_{\x{O}_{T}}(\x{O}_X, \x{N}))
$$
since $\x{F}$ and $\xHom_{\x{O}_{T}}(\x{O}_X, \x{N})$ are both $\x{O}_X$-modules.
Thus, we get the isomorphism 
$$
\psi \colon f_*\xHom_{\x{O}_{X}}(\x{F},\xHom_{\x{O}_{T}}(\x{O}_X, \x{N})) \to g_*\xHom_{\x{O}_{T}}(\x{F}, \x{N})
$$

For 
similar reasons we have a natural isomorphism 
$$
\phi' \colon \Hom_{\x{O}_{X}}(\x{F}',\xHom_{\x{O}_{T}}(\x{O}_X, \x{N})) \to \Hom_{\x{O}_{T}}(\x{F}', \x{N})
$$
where $\x{F}'$ is any quasi-coherent sheaf on $X$.

Now we have a commutative diagram 
$$
\xymatrix{
\mathfrak{Q}(T) \ar[rd]^{\alpha} \ar[r]^{\beta} & \mathfrak{Q}(X)\ar[d]^{\gamma}\\
& \mathfrak{Q}(Y)
}
$$
where $\alpha$ is the functor  $g_*\xHom_{\x{O}_{T}}(\x{F},-)$, $\beta$ is 
the functor $\xHom_{\x{O}_{T}}(\x{O}_X,-)$ and $\gamma$ is the functor 
$f_*\xHom_{\x{O}_{X}}(\x{F},-)$. 

On the other hand, if $\x{N}$ is injective in $\mathfrak{Q}(T)$, then 
$\xHom_{\x{O}_{T}}(\x{O}_X, \x{N})$ is  injective in $\mathfrak{Q}(X)$
because the functor $\Hom_{\x{O}_{T}}(-, \x{N})$ is exact on 
$\mathfrak{Q}(T)$ and given the isomorphism $\phi'$, the functor 
$\Hom_{\x{O}_{X}}(-, \xHom_{\x{O}_{T}}(\x{O}_X, \x{N}))$ will be exact 
on $\mathfrak{Q}(X)$.

In view of the above, we get a spectral sequence 
$$
E_2^{p,q}=\xExt^p_{f}(\x{F},\xExt^q_{\x{O}_{T}}(\x{O}_X, \x{N})_{\mathfrak{Q}})_{\mathfrak{Q}} \implies \xExt^{p+q}_{g}(\x{F}, \x{N})_{\mathfrak{Q}}
$$
which by Remark-Theorem \ref{r-ext-coherent}(\ref{t-ext-qc-o-module}) translates into 
$$
E_2^{p,q}=\xExt^p_{f}(\x{F},\xExt^q_{\x{O}_{T}}(\x{O}_X, \x{N})) \implies \xExt^{p+q}_{g}(\x{F}, \x{N})
$$
if $\x{N}$ is coherent.\\

\section{Dualising pairs}

\begin{defn}
Let $f \colon X\to Y$ be a projective morphism of Noetherian schemes, and 
let $r=\max\{\dim X_y \mid y\in Y\}$. A pair 
$(f^!,t_f)$ is called a dualising pair for $f$ if $f^!$ is a covariant functor from the 
category of coherent sheaves on $Y$ to the category of coherent sheaves on $X$, 
and $t_f\colon R^rf_*f^!\to \rm{id}$ is a natural transformation  
inducing a bifunctorial isomorphism 
$$
\theta^0\colon f_*\xHom_{\x{O}_X}(\x{F},f^{!}\x{G}) \to \xHom_{\x{O}_Y}(R^{r}f_*\x{F},\x G)
$$ 
for any coherent sheaf $\x{F}$ on $X$ and any coherent sheaf $\x{G}$ on $Y$. 
We say that a dualising sheaf $\omega_f$ exists for $f$ if $f^!\x{G}\simeq \omega_f\otimes \x{G}$ 
for every $\x{G}$. If such a sheaf exists then obviously $\omega_f:=f^!{\x{O}_{_Y}}$.\\
\end{defn}

\begin{lem}
Let $f \colon X\to Y$ be a projective morphism of Noetherian schemes. 
If a dualising pair exists for $f$, then it is unique (up to isomorphism).
\end{lem}
\begin{proof}
Let $r=\max\{\dim X_y \mid y\in Y\}$ .If $(f^!,t_f)$ is a dualising pair for $f$, then from the definition we get 
a bifunctorial isomorphism 
$$
\Hom_{\x{O}_X}(\x{F},f^{!}\x{G}) \simeq \Hom_{\x{O}_Y}(R^{r}f_*\x{F},\x G)
$$ 
which means that $f^!$ is a left adjoint to $R^rf_*$. Since a left adjoint is 
unique, this uniquely determines $f^!$ up to isomorphism. On the other hand, 
the adjoint property applied to $\x{F}=f^{!}\x{G}$ implies that the identity 
morphism $f^{!}\x{G}\to f^{!}\x{G}$ corresponds exactly to the morphism 
$R^{r}f_*f^!\x{G}\to\x G$ given by $t_f$ hence $t_f$ is also uniquely 
determined up to isomorphism.\\
\end{proof}

\begin{thm}\label{t-dualising-pair}
 Let $f \colon X\to Y$ be a flat projective morphism of Noetherian schemes, and let $r=\max\{\dim X_y \mid y\in Y\}$. Fix a closed immersion $e\colon X\to \mathbb{P}^n_Y$ such that $f=\pi e$ where 
$\pi\colon \mathbb{P}^n_Y\to Y$ is the projection. Then, a dualising pair $(f^!,t_f)$ exists for $f$. More precisely, for any $p$, any coherent $\x{O}_X$-module $\x{F}$, and any coherent $\x{O}_Y$-module $\x{G}$ there is a morphism  
$$
\theta^p\colon \xExt^p_f(\x{F},f^{!}\x{G}) \to \xHom_{\x{O}_Y}(R^{r-p}f_*\x{F},\x G)
$$ 
which is functorial in $\x{F}$ and $\x{G}$ where $\theta^0$ is always an isomorphism  
and
$$
f^!\x{G}=\xExt^{n-r}_{\x{O}_{\PP^n_Y}}(\x{O}_X, \omega_\pi\otimes \x{G})
$$ 
\end{thm}
\begin{proof}

Let $\x{M}$ be a coherent sheaf on $X$, flat over $Y$. If $p<n-r$ we will prove that  
$$
\xExt^p_{\x{O}_{\PP^n_Y}}(\x{M},\omega_\pi\otimes\x{G})=0
$$
 Since the statement is local over $Y$, we may assume that $Y$ is affine say $Y=\Spec A$. We will prove that 
$$
\xExt^p_{\x{O}_{\PP^n_Y}}(\x{M},\omega_\pi\otimes\x{G})\otimes\x{O}_{\PP^n_Y}(l)=0
$$ 
if $l\gg 0$. Since the sheaves involved are all coherent, the latter sheaf is generated 
by global sections hence it is enough to prove that 
$$
H^0(\PP^n_Y,\xExt^p_{\x{O}_{\PP^n_Y}}(\x{M},\omega_\pi\otimes \x{G})\otimes\x{O}_{\PP^n_Y}(l))=0
$$
By Theorem \ref{t-ext-local-global-degen}, we have 
$$
H^0(\PP^n_Y,\xExt^p_{\x{O}_{\PP^n_Y}}(\x{M},\omega_\pi\otimes\x{G})\otimes\x{O}_{\PP^n_Y}(l))=
\Ext^p_{\x{O}_{\PP^n_Y}}(\x{M},\omega_\pi\otimes\x{G}(l))
$$
$$
=\Ext^p_{\x{O}_{\PP^n_Y}}(\x{M}(-l),\omega_\pi\otimes\x{G})
$$
for large $l$ hence it is enough to prove that the right hand side vanishes. 
If $p<n-r$, then $r<n-p$ hence by a general theorem on higher direct images we have 
$R^{n-p}\pi_*\x{M}(-l)=0$. Therefore $R^{n-p}\pi_*\x{M}(-l)$ commutes with base change 
for every $p\le n-r$. Now by the duality on $\PP^n_Y$ (\ref{duality-projective-space}),  for 
$p<n-r$ we have   
$$
\Ext^p_{\x{O}_{\PP^n_Y}}(\x{M}(-l),\omega_\pi\otimes\x{G})\simeq 
\Hom_{\x{O}_Y}(R^{n-p}\pi_*\x{M}(-l),\x G)=0
$$
If $p<n-r$, another application of duality on $\PP^n_Y$ proves that $\xExt^p_\pi(\x{M},\omega_\pi\otimes\x{G})=0$.\\

From section \ref{duality-spectral-sequence}, we get the 
spectral sequence 
$$
E_2^{p,q}=\xExt^p_{f}(\x{F},\xExt^q_{\x{O}_{\PP^n_Y}}(\x{O}_X, \omega_\pi\otimes\x{G})) \implies \xExt^{p+q}_{\pi}(\x{F}, \omega_\pi\otimes\x{G})
$$
which satisfies $E^{p,q}_2=0$ if $q<n-r$. Thus for every $p$  we get a morphism  
$$
\mu^p\colon E^{p,n-r}_2=\xExt^p_f(\x{F},f^!\x{G}) \to \xExt^{n-r+p}_{\pi}(\x{F}, \omega_\pi\otimes\x{G})
$$
which is an isomorphism for $p=0$ where we have put $f^!\x{G}=\xExt^{n-r}_{\x{O}_{\PP^n_Y}}(\x{O}_X, \omega_\pi\otimes\x{G})$. 
This  in turn induces the morphism
$$
\theta^p\colon \xExt^p_f(\x{F},f^{!}\x{G}) \to \xHom_{\x{O}_Y}(R^{r-p}f_*\x{F},\x G)
$$ 
by composing $\mu$ and $\theta$ on $\PP^n_Y$ as in the proof 
of Theorem \ref{duality-projective-space}. Functoriality follows from the construction.
 
To see that $\theta^0$ is an isomorphism we may assume that $Y=\Spec A$. In that case, 
there is an exact sequence $\x{L}'\to \x{L} \to \x{F} \to 0$ where $\x{L}',\x{L}$ 
are finite direct sums of sheaves of the form ${\x{O}_{X}(-l)}$ for 
sufficiently large $l$. We get a commutative diagram 
{\tiny{
$$
\xymatrix{
0 \to \Ext^{n-r}_{\x{O}_{\PP^n_Y}}(\x{F}, \omega_\pi\otimes \x{G}) \ar[d]\ar[r] &
\Ext^{n-r}_{\x{O}_{\PP^n_Y}}(\x{L}, \omega_\pi\otimes \x{G}) \ar[d]\ar[r] &
\Ext^{n-r}_{\x{O}_{\PP^n_Y}}(\x{L}', \omega_\pi\otimes \x{G}) \ar[d]\\
0\to \Hom_A(H^r(\PP^n_Y,\x{F}),\x{G}(Y)) \ar[r]  &
\Hom_A(H^r(\PP^n_Y,\x{L}),\x{G}(Y)) \ar[r] &
\Hom_A(H^r(\PP^n_Y,\x{L}'),\x{G}(Y))
}
$$}} 
where the exactness of the upper row follows from the exactness of the 
corresponding sequence 
$$
0 \to \Hom_{\x{O}_{X}}(\x{F}, \omega_\pi\otimes \x{G}) \to 
\Hom_{\x{O}_{X}}(\x{L}, \omega_\pi\otimes \x{G}) \to
\Hom_{\x{O}_{X}}(\x{L}', \omega_\pi\otimes \x{G})
$$
Now by applying the duality on $\PP^n_Y$ to the flat sheaves $\x{L}',\x{L}$ we 
get the result that $\theta^0$ is an isomorphism. 

 Finally, the identity morphisms $f^!\x{G}\to f^!\x{G}$ induce 
morphisms $R^rf_*f^!\x{G} \to \x{G}$ giving $t_f$.\\
\end{proof}

\begin{rem}\label{r-dualising-sheaf}
In the setting of Theorem \ref{t-dualising-pair}, if $\x{G}$ is locally free, then 
$$
f^!\x{G}=\xExt^{n-r}_{\x{O}_{\PP^n_Y}}(\x{O}_X, \omega_\pi\otimes\x{G})\simeq 
\xExt^{n-r}_{\x{O}_{\PP^n_Y}}(\x{O}_X, \omega_\pi)\otimes\x{G}\simeq\omega_f\otimes \x{G}
$$
where we define $\omega_f:=\xExt^{n-r}_{\x{O}_{\PP^n_Y}}(\x{O}_X, \omega_\pi)$. In particular, if $Y=\Spec k$ where $k$ is a field, 
then $\x{G}$ is always free hence the formula holds.\\
\end{rem}

\section{General case of duality}

Let $f\colon X\to Y$ be a morphism of schemes. We say that the fibre $X_y$ over $y\in Y$ is of pure dimension $r$ if every irreducible component of $X_y$ has dimension $r$. In the following duality theorem we are interested in the case when the fibres $X_y$ are Cohen-Macaulay schemes (see Definition \ref{d-cm-schemes}).\\

\begin{thm}[Relative duality]\label{t-duality}
Let $f \colon X\to Y$ be a flat projective morphism of Noetherian schemes, with fibres 
of pure dimension $r$, and let $(f^!,t_f)$ and $\theta^p$ be as in Theorem \ref{t-dualising-pair}. 
Then the following are equivalent:

(i) the fibres of $f$ are Cohen-Macaulay schemes, 

(ii) if $\x{M}$ is a coherent locally free $\x{O}_X$-module, then 
 for every sufficiently large $l$,  the sheaf $R^{q}f_*\x{M}(-l)$ 
commutes with base change for every $q$ and it vanishes for $q\neq r$,

(iii)  a dualising sheaf $\omega_f$ exists for $f$, and it is flat over $Y$. Moreover, for fixed $m$ and $\x{F}$ flat over $Y$, $R^{r-p}f_*\x{F}$  commutes with base change 
for every $p\le m$ if and only if $\theta^p$ is an isomorphism for every $\x{G}$ and every $p\le m$. 
 \\

\end{thm}
\begin{proof}

\emph{Step 1.} Fix a closed immersion $e\colon X\to \mathbb{P}^n_Y$ such that $f=\pi e$ where 
$\pi\colon \mathbb{P}^n_Y\to Y$ is the projection.
From section \ref{duality-spectral-sequence}, we get the 
spectral sequence 
$$
E_2^{p,q}=\xExt^p_{f}(\x{F},\xExt^q_{\x{O}_{\PP^n_Y}}(\x{O}_X, \omega_\pi\otimes\x{G})) \implies \xExt^{p+q}_{\pi}(\x{F}, \omega_\pi\otimes\x{G})
$$
which satisfies  $E^{p,q}_2=0$ if $q<n-r$ by the proof of Theorem \ref{t-dualising-pair}.\\

\emph{Step 2.} Before we prove the equivalence of the conditions in the theorem we need some preparations. Let $\x{M}$ be a coherent sheaf on $X$, flat over $Y$, such that 
for every sufficiently large $l$, the sheaf $R^{q}f_*\x{M}(-l)$ 
commutes with base change for every $q$ and it vanishes for $q\neq r$.
We will show that 
 $$
 {L}^{q}(\x{G}):=\xExt^{n-q}_{\x{O}_{\PP^n_Y}}(\x{M},\omega_\pi\otimes\x{G})=0
 $$ 
 for every $q<r$ and every $\x{G}$. The problem is local hence we could take $Y$ affine.
Note that $L^q(\x{G})=0$ if and only if 
$L^{q}(\x{G})\otimes \x{O}_{\PP^n_Y(l)}=0$. If $l\gg 0$, then $L^{q}(\x{G})\otimes \x{O}_{\PP^n_Y}(l)=0$ if and only if $H^0(\PP^n_Y,L^q(\x{G})\otimes \x{O}_{\PP^n_Y(l)})=0$. 
By Theorem \ref{t-ext-local-global-degen} and by duality on $\PP^n_Y$, 
$$
H^0(\PP^n_Y,L^q(\x{G})\otimes \x{O}_{\PP^n_Y}(l))\simeq \Ext^{n-q}_{\x{O}_{\PP^n_Y}}(\x{M},\omega_\pi\otimes \x{G}(l))\simeq 
$$
$$
\Ext^{n-q}_{\x{O}_{\PP^n_Y}}(\x{M}(-l),\omega_\pi\otimes \x{G})\simeq \Hom_{\x{O}_{Y}}(R^q\pi_*\x{M}(-l),\x{G})=0
$$
if $q<r$.\\

\emph{Step 3.} (i) $\implies$ (ii): 
Let $\x{M}$ be a coherent locally free $\x{O}_X$-module. Since $f$ is flat, any coherent locally 
free sheaf on $X$ is also flat over $Y$. By base change theory if $H^q(X_y,\x{M}(-l)_y)=0$ for some $y\in Y$, then 
$R^{q}f_*\x{M}(-l)=0$ in a neighborhood of $y$. Since $Y$ is Noetherian, it is then 
enough to assume that $Y=\Spec k$ for some field $k$. By duality on $\PP^n_Y$, 
$H^q(X,\x{M}(-l))$ is dual to $\Ext^{n-q}_{\x{O}_{\PP^n_Y}}(\x{M}(-l),\omega_\pi)$.
By Theorem \ref{t-ext-local-global-degen}, the vanishing of the latter group follows from the vanishing of the 
sheaf $\xExt^{n-q}_{\x{O}_{\PP^n_Y}}(\x{M},\omega_\pi)=0$ for any large $l$. 
Let $U=\Spec B$ be an open affine subscheme of $\PP^n_Y$ and assume that $X$ is 
defined on $U$ by an ideal $I$ of $B$. We may in addition assume that $\omega_\pi$ is 
given by the module $B$ on $U$. By Remark-Theorem \ref{r-ext-coherent}, we then need to prove that 
$\Ext^{n-q}_B(B/I,B)=0$. By localising we may assume that $B$ is local with maximal ideal 
$\mathfrak{m}$. By the local duality Theorem \ref{t-lc-duality}, 
$$
\Ext^{n-q}_B(B/I,B)=\Hom_B(H^q_{\mathfrak m}(B/I),H^n_{\mathfrak m}(B))
$$
Since $B/I$ is Cohen-Macaulay, $\depth_{\mathfrak{m}/I} B/I=\dim B/I=r$ 
which in particular implies that $\depth_{\mathfrak m} B/I\ge r$. Now 
use Theorem \ref{t-lc-depth} to get the vanishing $H^q_{\mathfrak{m}}(B/I)=0$.\\

\emph{Step 4.} (ii) $\implies$ (iii): 
We prove the second statement first. 
By Step 2,
 $$
 \xExt^{n-q}_{\x{O}_{\PP^n_Y}}(\x{O}_X,\omega_\pi\otimes\x{G})=0
 $$ 
for every $q<r$ and every coherent $\x{G}$. 
Thus in  the spectral sequence 
$$
E_2^{p,q}=\xExt^p_{f}(\x{F},\xExt^q_{\x{O}_{\PP^n_Y}}(\x{O}_X, \omega_\pi\otimes\x{G})) \implies \xExt^{p+q}_{\pi}(\x{F}, \omega_\pi\otimes\x{G})
$$
$E_2^{p,q}=0$ unless $q=n-r$ in view of the vanishings obtained above (including the proof of Theorem \ref{t-dualising-pair}), and so we get the isomorphisms  
$$
\mu^p\colon \xExt^p_{f}(\x{F},f^!\x{G}) \to \xExt^{n-r+p}_{\pi}(\x{F}, \omega_\pi\otimes\x{G})
$$
where as usual $f^!\x{G}=\xExt^{n-r}_{\x{O}_{\PP^n_Y}}(\x{O}_X, \omega_\pi\otimes\x{G})$. 

Fix $m$ and a coherent $\x{F}$ flat over $Y$. Assume that $R^{r-p}f_*\x{F}$  commutes with base change for every $p\le m$.
Now apply duality on $\PP^n_Y$ to deduce that  
$$
\theta^p\colon \xExt^p_f(\x{F},f^{!}\x{G}) \to \xHom_{\x{O}_Y}(R^{r-p}f_*\x{F},\x G)
$$ 
is an isomorphism for every $p\le m$ and every coherent $\x{G}$. Conversely, 
assume that $\theta^p$ is an isomorphism for every $p\le m$ and every coherent $\x{G}$. 
Therefore, the morphisms 
$$
\xExt^{n-r+p}_{\pi}(\x{F}, \omega_\pi\otimes\x{G}) \to \xHom_{\x{O}_Y}(R^{r-p}\pi_*\x{F},\x G)
$$
are isomorphisms for every $p\le m$ and every coherent $\x{G}$. Moreover, by the 
proof of Theorem \ref{t-dualising-pair} the maps are also isomorphisms for any 
$p<0$ since both sides simply vanish. Now apply the duality theorem on $\PP^n_Y$ 
to deduce that $R^{r-p}f_*\x{F}$  commutes with base change for every $p\le m$.

Proof of the first statement: the sheaf version of the Yoneda pairing gives a morphism  
$$
\xExt^{n-r}_{\x{O}_{\PP^n_Y}}(\x{O}_{X},\x{O}_{\PP^n_Y})\otimes 
\xHom_{\x{O}_{\PP^n_Y}}(\x{O}_{\PP^n_Y},\x{G})\to 
\xExt^{n-r}_{\x{O}_{\PP^n_Y}}(\x{O}_{X},\x{G})
$$
and by tensoring $\omega_\pi$ we get a morphism 
$$
\phi\colon \omega_f\otimes \x{G}\to 
\xExt^{n-r}_{\x{O}_{\PP^n_Y}}(\x{O}_{X},\omega_\pi\otimes \x{G})
$$
which we prove to be an isomorphism where $\omega_f:=\xExt^{n-r}_{\x{O}_{\PP^n_Y}}(\x{O}_{X},\omega_\pi)$. Since the problem is local we may assume that $Y$ is affine. 
Since $\xExt^{n-q}_{\x{O}_{\PP^n_Y}}(\x{O}_X,\omega_\pi\otimes\x{G})=0$ unless 
$q=r$, the functor $f^!$ is exact. Now let $\x{H}'\to \x{H} \to \x{G}\to 0$ 
be an exact sequence where $\x{H}'$ and $\x{H}$ are free sheaves of finite rank.
Thus,  $f^!\x{H}'\to f^!\x{H} \to f^!\x{G}\to 0$ is also exact. By 
Remark \ref{r-dualising-sheaf}, we have $f^!\x{H}'\simeq \omega_f\otimes \x{H}'$ 
and $f^!\x{H}\simeq \omega_f\otimes \x{H}$. Now the exact sequence 
$\omega_f\otimes \x{H}'\to \omega_f\otimes \x{H} \to \omega_f\otimes \x{G}\to 0$  
implies the isomorphism $f^!\x{G}\simeq \omega_f\otimes \x{G}$. Flatness of 
$\omega_f$ follows from exactness of $f^!$.\\

\emph{Step 5.} (iii) $\implies$ (i): We show that the sheaves $R^{q}f_*\x{O}_X(-l)$ commute 
with base change for every $q$ and $l\gg 0$. Since this is a local problem we may take $Y$ to be affine, 
say $\Spec A$. 
By assumptions the functor $f^!$ is exact. 
An argument similar to Step 2 show that if $l$ is sufficiently large, then 
$H^q(X,f^!\x{G}\otimes \x{O}(l))=0$ for any coherent $\x{G}$ and 
any $q>0$ (the bigness of $l$ 
does not depend on $\x{G}$ - see the exercises). Thus, $H^0(X,\x{O}_X(l)\otimes f^!-)$ is an exact functor on $\mathfrak{C}(Y)$. 
On the other hand, for any coherent $\x{G}$ we have isomorphisms 
$$
H^0(X,\x{O}_X(l)\otimes f^!\x{G})\simeq \Hom_{\x{O}_X}(\x{O}_X(-l),f^{!}\x{G}) \simeq \Hom_{\x{O}_Y}(R^{r}f_*\x{O}_X(-l),\x G)
$$
which implies that the functor $\Hom_{\x{O}_Y}(R^{r}f_*\x{O}_X(-l),-)$ is exact 
on $\mathfrak{C}(Y)$. This also implies that the same functor is exact on 
$\mathfrak{Q}(Y)$ too hence $H^r(X,\x{O}_X(-l))$ is a projective $A$-module. 
Therefore, the sheaf $R^{r}f_*\x{O}_X(-l)$ is locally free and 
it commutes with base change since $R^{q}f_*\x{O}_X(-l)=0$ if $q>r$. 
Base change theory implies that $R^{r-1}f_*\x{O}_X(-l)$ also 
commutes with base change and using the isomorphism 
$$
0=\Ext^1_{\x{O}_X}(\x{O}_X(-l),f^{!}\x{G}) \simeq \Hom_{\x{O}_Y}(R^{r-1}f_*\x{O}_X(-l),\x G)
$$
we get the vanishing $R^{r-1}f_*\x{O}_X(-l)=0$. By continuing this process 
one proves that $R^{q}f_*\x{O}_X(-l)=0$ and that it commutes with base change for every $q<r$ and $l\gg 0$.
In particular, if $y\in Y$, then $H^q(X_y,\x{O}_{X_y}(-l))=0$ for every $q<r$ and every $l\gg 0$.
To prove that $X_y$ is Cohen-Macaulay, we may then simply assume that $X=X_y$, that is, 
$Y=\Spec k$ for some field $k$. By Step 2,  $\xExt^{n-q}_{\x{O}_{\PP^n_Y}}(\x{O}_X,\omega_\pi)=0$ for every $q<r$. Locally, as in Step 2, this can be translated into the vanishing 
$$
0=\Ext^{n-q}_B(B/I,B)=\Hom_B(H^q_{\mathfrak m}(B/I),H^n_{\mathfrak m}(B))
$$
hence $H^q_{\mathfrak{m}}(B/I)=0$ for every $q<r$. 
Thus, $\depth_{\mathfrak m} B/I\ge r$ and so $\depth_{\mathfrak{m}/I} B/I=\dim B/I=r$ 
which implies that $B/I$ is Cohen-Macaulay. Therefore, $X$ is Cohen-Macaulay.\\
\end{proof}

\begin{rem}
In the relative duality theorem, if $Y=\Spec k$ where $k$ is a field, then the 
conditions on flatness and base change are automatically satisfied. This situation 
is where the duality theorem is most frequently used.\\
\end{rem}

\begin{rem}
Let $X$ be a normal projective variety of dimension $r$ over a field $k$ with the structure morphism 
$f\colon X\to \Spec k$, and let $X_s\subseteq X$ be the set of smooth points of 
$X$ and $j\colon X_s\to X$ the inclusion morphism. By Remark \ref{r-dualising-sheaf}, a dualising sheaf 
$\omega_f$ exists for $f$. However, in this case one can calculate $\omega_f$ in  
a more direct way. It is well-known that 
$\omega_f\simeq j_*\omega_{X_s}$ where $\omega_{X_s}$ is the canonical sheaf of $X_s$
defined as the exterior power $\wedge^r \Omega_{X_s}$ of the sheaf of differential forms 
(cf. Koll\'ar-Mori [\ref{Kollar-Mori}, Theorem 5.75]).\\
\end{rem}

\section{Applications.} We present some of the consequences of the 
ideas we have developed in relation with duality theory.

\begin{exa}[Riemann-Roch]
In fact, the duality theorem is 
an indispensable tool in the geometry of varieties. For example, it immediately implies the 
Riemann-Roch theorem for curves (cf. Hartshorne [\ref{Hartshorne}, IV, Theorem 1.3]). 
On a smooth projective curve $X$ over an algebraically closed field $k$,  
if $D$ is a divisor, then using the duality theorem one can write 
the Euler characteristic $\mathcal{X}(\x{O}_X(D))$ as $\dim_kH^0(X,\x{O}_X(D))-\dim_kH^0(X,\omega_X\otimes \x{O}_X(-D))$, and using induction in a sense on the degree of $D$ one can get the Riemann-Roch.  
With some more efforts the Riemann-Roch theorem for surfaces 
is also derived from the duality theorem (cf. Hartshorne [\ref{Hartshorne}, V, Theorem 1.6]).\\
\end{exa}

\begin{exa}[Kodaira vanishing]\label{exa-Kodaira-vanishing}
The Kodaira vanishing theorem states that on a smooth projective variety $X$ of 
dimension $r$ over 
the complex numbers, $H^p(X,\omega_X\otimes \x{L})=0$ if $p>0$ and $\x{L}$ is 
an ample invertible sheaf. The theorem is a fundamental tool in birational geometry 
over the complex numbers. In fact, it is so essential that it has prevented birational 
geometry over fields of positive characteristic to move any further because 
Kodaira vanishing fails in this case.
By duality, the above vanishing is equivalent to the vanishing 
$H^p(X,\x{L}^{\vee})=0$ for $p<r$. Using Hodge theory one proves that 
for some $m\gg 0$ the map 
$$
H^p(X,\x{L}^{-m})\to H^p(X,\x{L}^{\vee})
$$
is surjective for any $p$  where $\x{L}^{-m}$ denotes $m$ times the tensor product of $\x{L}^{\vee}$ with itself. By another application of duality, 
$$
\dim_k H^p(X,\x{L}^{-m})=\dim_k H^{r-p}(X,\omega_X\otimes \x{L}^{m})
$$ 
The latter dimension is zero if $p<r$ by the 
Serre vanishing theorem (cf. Koll\'ar-Mori [\ref{Kollar-Mori}, section 2.4]).\\
\end{exa}

Once one goes beyond curves, the so-called adjunction formula plays a very important role, in particular in birational geometry, to reduce problems to lower-dimensional varieties. The formula follows from the machinery of duality.\\ 

\begin{thm}[Adjunction]\label{t-adjunction}
Let $f \colon X\to Y$ be a flat projective morphism of Noetherian schemes with Cohen-Macaulay fibres of pure dimension $r$. Assume that $D$ is an effective Cartier divisor on $X$ and 
$g\colon D\to Y$ is the induced morphism such that the fibres of $g$ are of pure dimension $r-1$. Then, $\omega_g\simeq \omega_f\otimes \x{O}_X(D)\otimes \x{O}_D$.  
\end{thm}
\begin{proof}
We have an exact sequence 
$$
(*) \hspace{1cm} 0\to \x{O}_X(-D)\to \x{O}_X \to \x{O}_D\to 0
$$ 
since $\x{O}_X(-D)$ 
is isomorphic to the ideal sheaf $\x{I}_D$. Since the fibres of $f$ are Cohen-Macaulay, the fibres of $g$ are also Cohen-Macaulay because locally $D$ is defined by an element which is not a zero divisor nor invertible. Fix a closed immersion $e\colon X\to \PP^n_Y$ so that $f=\pi e$ where $\pi\colon \PP^n_Y\to Y$ 
is the projection. From the proof of Theorem \ref{t-dualising-pair} and Theorem \ref{t-duality}, 
we have  $\xExt^{n-q}_{\x{O}_{\PP^n_Y}}(\x{O}_X,\omega_\pi)=0=\xExt^{n-r}_{\x{O}_{\PP^n_Y}}(\x{O}_X(-D),\omega_\pi)$ for every $q\neq r$ and 
$\xExt^{n-q}_{\x{O}_{\PP^n_Y}}(\x{O}_D,\omega_\pi)=0$ for every $q\neq r-1$.  

Thus, we get the short exact sequence 
$$
0 \to \xExt^{n-r}_{\x{O}_{\PP^n_Y}}(\x{O}_X,\omega_\pi) \to 
\xExt^{n-r}_{\x{O}_{\PP^n_Y}}(\x{O}_X(-D),\omega_\pi) \to 
\xExt^{n-r+1}_{\x{O}_{\PP^n_Y}}(\x{O}_D,\omega_\pi) \to 0
$$
from which we get an exact sequence 
$$
 \xExt^{n-r}_{\x{O}_{\PP^n_Y}}(\x{O}_X,\omega_\pi)\otimes \x{O}_D \to 
\xExt^{n-r}_{\x{O}_{\PP^n_Y}}(\x{O}_X(-D),\omega_\pi)\otimes \x{O}_D \to 
\xExt^{n-r+1}_{\x{O}_{\PP^n_Y}}(\x{O}_D,\omega_\pi)\otimes \x{O}_D \to 0
$$
We will show that the first morphism in the sequence is zero. In fact, on any sufficiently small open subscheme $\Spec B\subset \PP^n_Y$, if $X$ is defined by an ideal $I$ in $B$, then 
$D$ is defined by the ideal $J=I+\langle t \rangle$ where we assume that $D$ is defined by the 
element $t$ on $X\cap \Spec B$. In this setting, the exact sequence $(*)$ corresponds tothe exact sequence  $0 \to B/I \to B/I \to B/J \to 0$ where the first map is just multiplication 
by $t$. We then get the exact sequence 
$$
\Ext^{n-r}_{B}(B/I,B)\otimes B/J \to 
\Ext^{n-r}_{B}(B/I,B/I)\otimes B/J \to 
\Ext^{n-r+1}_{B}(B/J,B)\otimes B/J \to 0
$$
which corresponds to the above exact sequence for sheaves. By construction, the first map 
takes $m\otimes b$ to $tm\otimes b=m\otimes tb=m\otimes 0=0$ showing that the first map is indeed 
the zero map. This local analysis implies that the morphism 
$$
\xExt^{n-r}_{\x{O}_{\PP^n_Y}}(\x{O}_X(-D),\omega_\pi)\otimes \x{O}_D \to 
\xExt^{n-r+1}_{\x{O}_{\PP^n_Y}}(\x{O}_D,\omega_\pi)\otimes \x{O}_D=\omega_g
$$
is an isomorphism. 
 
On the other hand, the Yoneda pairing 
$$
\xHom_{\x{O}_{\PP^n_Y}}(\x{O}_{X}(-D),\x{O}_X)\otimes 
\xExt^{n-r}_{\x{O}_{\PP^n_Y}}(\x{O}_{X},\omega_\pi)
\to \xExt^{n-r}_{\x{O}_{\PP^n_Y}}(\x{O}_X(-D),\omega_\pi)
$$ 
gives the natural morphism 
$$
\x{O}_{X}(D)\otimes \omega_f
\to \xExt^{n-r}_{\x{O}_{\PP^n_Y}}(\x{O}_X(-D),\omega_\pi)
$$
which is an isomorphism as can be checked locally. Putting all the above together 
yields the result.\\
\end{proof}

\section{Cohen-Macaulay schemes}

A Noetherian ring $A$ is called Cohen-Macaulay if the following equivalent conditions 
hold:

$\bullet$ for every maximal ideal $P$ of $A$, $\depth_{\mathfrak{m}}A_P=\dim A_P$ 
where $\mathfrak{m}$ is the maximal ideal of $A_P$,

$\bullet$ for every prime ideal $P$ of $A$, $\depth_{\mathfrak{m}}A_P=\dim A_P$ 
where $\mathfrak{m}$ is the maximal ideal of $A_P$,

$\bullet$ for every proper ideal $I$ of $A$, $\hgt I=\depth_IA$ where $\hgt I$ denotes 
the height of $I$.\\

For an ideal $I$ of a ring $A$, some authors use the term codimension of $I$ instead of 
height of $I$. However, here will use the term codimension of $I$ to mean $\dim A-\dim A/I$.

\begin{lem}\label{l-cm-properties}
Let $A$ be a local Cohen-Macaulay ring. Then, 

(i) if $a\in A$ is not a zero-divisor and not an invertible element, then $A/\langle a \rangle$ is 
Cohen-Macaulay of dimension $\dim A-1$,

(ii) any two maximal sequences of prime ideals of $A$ have the same 
length,

(iii) every associated prime of $A$ is a minimal prime ideal.
\end{lem}
\begin{proof} Assume that $d=\dim A$. (i) If $a$ is invertible, then the statement is trivial, so we assume that 
$a$ is not invertible. Since $a$ is not a zero-divisor and since $A$ is a local ring,
$\dim  A/\langle a \rangle=d-1$. On the other hand, there is a maximal 
$A$-regular sequence $a_1=a, a_2,\cdots, a_d$ of elements of the maximal ideal $\mathfrak{m}$ 
of $A$. So, the elements 
$a_2,\cdots, a_d$ give a regular sequence of $A/\langle a \rangle$ implying 
$\depth_{\mathfrak{m}/\langle a \rangle}A/\langle a \rangle\ge d-1$. By some general 
results on depth,
we also have $\depth_{\mathfrak{m}/\langle a \rangle}A/\langle a \rangle\le d-1$ 
hence we have $\depth_{\mathfrak{m}/\langle a \rangle}A/\langle a \rangle= d-1$.

(ii)  We will 
prove that if $P_0\subset P_1\subset \cdots \subset P_l=\mathfrak{m}$ is any maximal sequence of 
prime ideals in $A$, then $l=d$. 
This obviously implies the claim of (ii). Here, the sequence being maximal means that there is 
no prime ideal between $P_i$ and $P_{i+1}$, and that $P_0$ is a minimal prime ideal.

We use induction on $l$. If $l=0$, there is nothing to prove so we may assume that $l>0$. 
Since $A$ is Cohen-Macaulay, $\depth_{P_1}A=\hgt P_1>0$. For the moment assume that 
$\depth_{P_1}A=1$. Then, 
there is some $a_1\in P_1$ which is not a zero-divisor and $ P_1/\langle a_1\rangle\subset \cdots \subset P_l/\langle a_1\rangle$ is a maximal 
sequence of prime ideals of $A/\langle a_1\rangle$. Since the latter ring is Cohen-Macaulay, 
by induction $l-1=\dim A/\langle a_1\rangle=d-1$ which implies the 
result.

To show that $\depth_{P_1}A=1$, by localising at $P_1$, we may assume that 
$P_1=\mathfrak{m}$. Pick an element $b\in P_1$ which is not in $P_0$. 
Since there is no other prime divisor between $P_0$ and $P_1$, $P_1$ is the only prime
ideal over $P_0+\langle b \rangle$, so $P_1^n\subseteq P_0+\langle b \rangle$ 
for some $n>0$. In particular, this means that any element $b'\in P_1$ is 
a zero-divisor of $A/\langle b \rangle$ because $P_0$ is a minimal prime ideal 
and all its elements are zero-divisors. This implies that $\depth_{P_1}A=1$. 

(iii) Let $P$ be an associated prime ideal of $A$. By definition, every element 
of $P$ is a zero-divisor hence $\depth_P A=0$. Thus, $\hgt P=0$ which means that 
$P$ is a minimal prime ideal.\\
\end{proof}

\begin{defn}\label{d-cm-schemes}
A Noetherian scheme $X$ is said to be Cohen-Macaulay if the local ring 
$\x{O}_x$ is Cohen-Macaulay for every $x\in X$.
\end{defn}

\begin{rem}
(1) If $X=\Spec A$ with $A$ Noetherian then obviously $X$ is Cohen-Macaulay 
if and only if $A$ is Cohen-Macaulay.
(2) Any regular ring is Cohen-Macaulay hence regular schemes are Cohen-Macaulay too.
(3) It is well-known that a local ring $A$ is Cohen-Macaulay if and only if 
its completion $\widehat{A}$ is Cohen-Macaulay.
\end{rem}
 
In many respects, Cohen-Macaulay schemes behave like regular schemes. Their 
importance is already clear from the duality theorem, and  
their class is much bigger than that of regular schemes.
Many examples of 
Cohen-Macaulay schemes can be constructed by successive applications of the next theorem.\\

\begin{thm}
Let $X$ be a Cohen-Macaulay scheme and let $D$ be an effective Cartier 
divisor on $X$. Then, $D$ is also a Cohen-Macaulay scheme.
\end{thm}
\begin{proof}
The ideal sheaf $\x{I}_D$ of $D$ is simply the invertible sheaf $\x{O}_X(-D)$.
For any $x\in X$, there is an ideal $I_x$ of $\x{O}_x$ corresponding to 
$\x{I}_D$ which is generated by a single element $t_x$ which is not a 
zero-divisor. If $D=0$ near $x$, then $t_x$ is invertible and there is nothing to prove so 
we assume that $D\neq 0$ near $x$. Thus, $t_x$ is not invertible and so it can be regarded 
as an $\x{O}_x$-regular sequence of length one. In particular, $\x{O}_x/\langle t_x \rangle$ is also 
a Cohen-Macaulay ring which is nothing but the local ring of $D$ at $x$. 
Therefore, $D$ is Cohen-Macaulay.\\
\end{proof}

\begin{exa}[Local complete intersections]
Let $X$ be a regular Noetherian scheme and $Z$ a closed subscheme. We say that 
$Z$ is a local complete intersection at $x\in X$ if the ideal $I_x$ of $Z$ in 
$\x{O}_x$ can be generated by $n:=\dim \x{O}_x-\dim \x{O}_x/I_x$ elements 
$a_1,\dots, a_n$. We show that these elements form an $\x{O}_x$-regular 
sequence which by Lemma 
\ref{l-cm-properties} proves that $\x{O}_x/I_x$  is 
Cohen-Macaulay hence $Z$ is Cohen-Macaulay at $x$. If each $a_i$ is a zero-divisor, then $I_x$ is a 
subset of some associated prime of $\x{O}_x$ and again by Lemma \ref{l-cm-properties} 
such primes are minimal. Therefore, another application of the lemma 
shows that dimension of $\x{O}_x/I_x$ and $\x{O}_x$ are the same 
hence $I_x=0$. So, we may assume that some element, say $a_1$, is not 
a zero-divisor. Now apply induction to $\x{O}_x/\langle a_1 \rangle$.
\end{exa}

In contrast to the above observations, the next few theorems give necessary geometric 
conditions for being Cohen-Macaulay hence yielding many examples of schemes 
which are not Cohen-Macaulay.\\

\begin{thm}\label{t-cm-dimension}
Let $X$ be a Cohen-Macaulay scheme. Then, $X$ is equidimensional at any point 
$x\in X$, i.e., all irreducible components of $X$ passing through $x$ have the 
same dimension.
\end{thm}
\begin{proof}
Pick $x\in X$. First note that the irreducible components of $X$ passing through $x$ 
correspond to the minimal prime ideals of $\x{O}_x$. The theorem then states that if  
$P$ and $Q$ are minimal prime ideals, then $\dim \x{O}_x/P=\dim \x{O}_x/Q$. This 
immediately follows from Lemma \ref{l-cm-properties}.
\end{proof}

Thus the theorem implies that the scheme which consists of a curve and a surface intersecting at a point is not Cohen-Macaulay.\\

\begin{thm}\label{t-cm-embedded-point}
A Cohen-Macaulay scheme has no embedded points.
\end{thm}
\begin{proof}
Let $X$ be a scheme. An embedded point $x\in X$ is by definition 
a point such that the maximal ideal $\mathfrak{m}$ of $\x{O}_x$ 
is an associated prime but it is not a minimal prime ideal. 
Assume that $x$ is an embedded point. If $X$ is Cohen-Macaulay, then 
by Lemma \ref{l-cm-properties}, 
$\mathfrak{m}$ is a minimal prime ideal, 
a contradiction.\\
\end{proof}

\begin{exa}
Let $k$ be a field and let $X$ be the closed subscheme of $\A^2_k=\Spec k[t_1,t_2]$ defined 
by the ideal $\langle t_1^2,t_1t_2 \rangle$. Then, the ideal $\langle t_1,t_2 \rangle$ 
corresponds to a point (the origin of the plane) which is an embedded point of $X$ hence 
$X$ is not Cohen-Macaulay.
\end{exa}

The following theorem is known as the Hartshorne connectedness theorem. It simply says 
that if $X$ is a connected Cohen-Macaulay scheme and if one removes a closed subset of codimension 
at least two, then the remaining scheme would still be connected.\\

\begin{thm}
Let $X$ be a Cohen-Macaulay scheme and $x\in X$. Assume that $X_1$ and $X_2$ are 
closed subschemes of $X$ such that $X=X_1\cup X_2$ and such that near $x$, $X_1\cap X_2\neq X_j$ 
for $j=1,2$. 
Then, near $x$, $X_1\cap X_2$ has codimension $\le 1$ with equality if $X_1$ and $X_2$ have 
no common components near $x$.
\end{thm}
\begin{proof}
Let $I_1$ and $I_2$ be the ideals of $\x{O}_x$ corresponding 
to $X_1$ and $X_2$. Then, the theorem states that 
$\dim \frac{\x{O}_x}{I_1+I_2}\ge \dim \x{O}_x-1$. Since the length of any two maximal 
sequences of prime ideals of $\x{O}_x$ are equal by Lemma 
\ref{l-cm-properties}, the statement is equivalent to saying that 
$\hgt (I_1+I_2)\le 1$. Note that since $\x{O}_x$ is Cohen-Macaulay, 
$\hgt (I_1+I_2)=\depth_{I_1+I_2}\x{O}_x$. 

Assume that $\depth_{I_1+I_2}\x{O}_x>1$.
 The Mayer-Vietoris exact sequence gives the exact sequence 
$$
0\to H^0_{I_1+I_2}(\x{O}_x) \to H^0_{I_1}(\x{O}_x)\oplus H^0_{I_2}(\x{O}_x)
\to H^0_{I_1\cap I_2}(\x{O}_x)\to H^1_{I_1+I_2}(\x{O}_x)
$$
By Theorem \ref{t-lc-depth}, $H^0_{I_1+I_2}(\x{O}_x)=0$ and $H^1_{I_1+I_2}(\x{O}_x)=0$.  
 On the other hand, since 
$X=X_1\cup X_2$, $I_1\cap I_2$ is nilpotent which implies that 
$H^0_{I_1\cap I_2}(\x{O}_x)=\x{O}_x$. Therefore, we get the isomorphism 
$H^0_{I_1}(\x{O}_x)\oplus H^0_{I_2}(\x{O}_x)
\to \x{O}_x$ which implies that $H^0_{I_1}(\x{O}_x)+ H^0_{I_2}(\x{O}_x)
=\x{O}_x$. Thus, $H^0_{I_j}(\x{O}_x)=\x{O}_x$ for $j=1$ or $j=2$. 
But this is possible only if $I_j$ is nilpotent, that is, $X_j=X$ near $x$ 
which is a contradiction.

Finally, if $X_1$ and $X_2$ have no common components near $x$, then 
$\hgt (I_1+I_2)=1$ otherwise $I_1+I_2$ is contained in some minimal prime 
ideal which corresponds to a common component, a contradiction.\\
\end{proof}

\begin{exa}
Let $k$ be a field and let $X_1$ and $X_2$ be the closed subschemes of 
$\A^4_k=\Spec k[t_1,\dots,t_4]$ defined by the ideals $\langle t_1, t_2\rangle$ 
and $\langle t_3, t_4\rangle$ respectively. If $x$ is the point corresponding to 
the prime ideal $\langle t_1, \dots,t_4\rangle$, then $X_1\cap X_2=\{x\}$. Therefore, 
the scheme $X:=X_1\cup X_2$ is not Cohen-Macaulay at $x$ since $X_1\cap X_2$ has 
codimension two near $x$.\\
\end{exa}

\begin{rem}[Macaulayfication]
Let $f\colon X'\to X$ be a proper birational morphism of Noetherian schemes. If $X'$ 
is regular, then we call $f$ a resolution of singularities of $X$. In some cases 
it is known that such a resolution exists, for example, if $X$ is a variety over 
an algebraically closed field $k$ of characteristic zero. But if $k$ has positive characteristic, 
then the existence of resolutions is not known. 
Now in the above setting instead of assuming $X'$ to be regular, assume that it is 
Cohen-Macaulay. We then call $f$ a Macaulayfication of $X$. The existence of a 
Macaulayfication of $X$ is known if $X$ is a separated scheme of finite type 
over a Noetherian ring $A$ which has a dualising complex (eg, when $A$ is a regular ring, or when it is a finitely generated $k$-algebra for some field $k$)[\ref{Kawasaki}].\\
\end{rem}

\begin{exa}[Rational singularities]
Let $X$ be a normal variety over the complex numbers. We say that $X$ has rational singularities 
if for any resolution of singularities $f\colon X'\to X$ we have $R^qf_*\x{O}_{X'}=0$ if $q>0$. 
In general, rational singularities are Cohen-Macaulay and the proof is quite easy when 
$X$ is projective. In fact, if $X$ is projective one can use the duality 
theorem as follows. Let $\x{M}$ be a locally free sheaf of finite rank on $X$. 
It is enough to prove that $H^p(X,\x{M}(-l))=0$ if $p<r=\dim X$ and $l\gg 0$. There is a spectral 
sequence, called the Leray spectral sequence, which compares cohomology on $X$ and 
$X'$. It is given by the fact that the composition of the two functors $H^0(X,-)$ 
and $f_*$ is the same as $H^0(X',-)$. The spectral sequence is as 
$$
E^{p,q}_2=H^p(X, R^qf_*\x{F})\implies H^{p+q}(X',\x{F})
$$ 
where $\x{F}$ is any sheaf on $X'$. We take $\x{F}=f^*(\x{M}(-l))$. By the projection 
formula (see Theorem \ref{t-ext-locally-free} and remarks after it), $R^qf_*f^*(\x{M}(-l))=R^qf_*\x{O}_{X'}\otimes \x{M}(-l)$, and since $X$ is 
normal $f_*\x{O}_{X'}=\x{O}_X$. By assumptions,  
$R^qf_*\x{O}_{X'}=0$ if $q>0$, so 
$E^{p,q}_2=0$ unless $q=0$. The spectral sequence then gives isomorphisms 
$$
H^p(X, \x{M}(-l))\simeq H^{p}(X',f^*(\x{M}(-l)))
$$

On the other hand, by duality on $X'$ we have 
$$
\dim_\C H^{p}(X',f^*(\x{M}(-l)))=\dim_\C H^{r-p}(X',\omega_{X'}\otimes f^*(\x{M}(-l))^\vee)
$$
By a generalisation of the Kodaira vanishing theorem (see Example \ref{exa-Kodaira-vanishing}, and [\ref{KMM}, Theorem 1-2-3]), 
$R^qf_*(\omega_{X'}\otimes f^*(\x{M}(-l))^\vee)=0$ if $q>0$ and another 
application of the Leray spectral sequence with $\x{F}=\omega_{X'}\otimes f^*(\x{M}(-l))^\vee$, 
and the Serre vanishing show that
$$
H^{r-p}(X',\omega_{X'}\otimes f^*(\x{M}(-l))^\vee)\simeq H^{r-p}(X,f_*\omega_{X'}\otimes \x{M}^\vee(l))=0
$$
for any $r-p>0$ and $l\gg 0$. Therefore, $H^p(X, \x{M}(-l))=0$ for any $p<r$ and $l\gg 0$, so $X$ is Cohen-Macaulay.\\
\end{exa}

\begin{rem}[Serre condition $S_n$]
Let $A$ be a local Noetherian ring, and $n$ a natural number or $0$. We say that $A$ has 
property $S_n$ if for every prime ideal $P$ of $A$, we have $\depth_{\mathfrak m}A_P\ge \min\{n,\hgt P\}$ where $\mathfrak{m}$ is the maximal ideal of $A_P$. If $n=\dim A$, then 
$A$ satisfies $S_n$ if and only if $A$ is Cohen-Macaulay. On the other hand, any local 
Noetherian ring has property $S_0$, and it has property $S_1$ if and only if every 
associated prime of $A$ is a minimal prime ideal. Moreover, a theorem of Serre 
states that $A$ is normal if and only if it has property $S_2$ and for any prime 
ideal $P$ of $\hgt P\le 1$, $A_P$ is regular.\\
\end{rem}

Finally, we mention a theorem on the fibres of a flat morphism.

\begin{thm}
Let $f\colon X\to Y$ be a flat morphism of Noetherian schemes. Then, $X$ is  
Cohen-Macaulay if and only if the fibres of $f$ and the points in the image of $f$ 
are Cohen-Macaulay.  
\end{thm}
\begin{proof}
A point $y$ in the image of $f$ being Cohen-Macaulay means that $\x{O}_y$ 
is Cohen-Macaulay. Since the statement is local both on $X$ and $Y$, we may assume that $Y=\Spec A$ 
and $X=\Spec B$. Let $x\in X$ and $y=f(x)$.  Since $f$ is flat, $\x{O}_x$ is flat over $\x{O}_y$. The fibre of $f$ over $y$ is simply 
$\Spec B\otimes \frac{\x{O}_y}{\mathfrak{m}_y}$ where 
$\mathfrak{m}_y$ is the maximal ideal of $\x{O}_y$.
Moreover, the fibre of $f$ over $y$ is isomorphic to the 
fibre of $f'\colon X'=\Spec B\otimes \x{O}_y \to \Spec \x{O}_y$ over $y$ where $y$ can be regarded 
as the closed point of $\Spec \x{O}_y$. In particular, the fibre of $f'$ is a closed 
subscheme of $X'$. Now if $x'$ is the point on $X'$ 
corresponding to $x$, then the local ring of $x'$ on $X'$ is just $\x{O}_x$ and 
its local ring on the fibre of $f'$ is 
$\x{O}_x\otimes \frac{\x{O}_y}{\mathfrak{m}_y}\simeq \frac{\x{O}_x}{\mathfrak{m}_y\x{O}_x}$.

For any $x\in X$, a commutative algebra result [\ref{CM-rings}, Theorem 2.1.7] implies that $\x{O}_x$ is Cohen-Macaulay if and only if both $\x{O}_y$ and 
$\frac{\x{O}_x}{\mathfrak{m}_y\x{O}_x}$ are Cohen-Macaulay. In other words, $X$ is Cohen-Macaulay 
at $x$ if and only if $Y$ is Cohen-Macaulay at $y$ and the fibre of $f$ is Cohen-Macaulay 
at the point corresponding to $x$.
\end{proof}

\section*{Exercises}
\begin{enumerate}
\item Let $f\colon X\to Y$ be a flat projective morphism of  Noetherian schemes 
with $Y$ affine. Let $(f^!,t_f)$ be a dualising pair for $f$ and assume that $f^!$ is exact. 
Show that there is $l_0$ such that $H^p(X,f^!\x{G}(l))=0$ for any $p>0$, any $l\ge l_0$, and 
any coherent sheaf $\x{G}$ on $Y$.\\

\item Under the assumptions of Theorem \ref{t-duality}, show that the conditions (i),(ii),(iii) 
are equivalent to the following: for every sufficiently large $l$,  the sheaf $R^{q}f_*\x{O}_X(-l)$ 
commutes with base change for every $q$ and it vanishes for $q\neq r$.\\

\item In the setting of Theorem \ref{t-adjunction}, show that 
$\omega_g\simeq  \xExt^{1}_{\x{O}_{X}}(\x{O}_D,\omega_f)$.\\

\item Let $A$ be a Noetherian ring. Show that $A$ is Cohen-Macaulay if and only if 
$\A^n_A$ is Cohen-Macaulay.\\

\item Prove or disprove the following: a scheme $X$ is Cohen-Macaulay at $x\in X$ 
if and only if it is Cohen-Macaulay in a neighborhood of $x$.\\

\item Let $A$ be a Noetherian local ring. Show that $A$ has property $S_n$ if and only 
if for any prime ideal $P$ with $\depth_{\mathfrak m}A_P\le n-1$, $A_P$ is Cohen-Macaulay 
where $\mathfrak m$ is the maximal ideal $A_P$.
\end{enumerate}


\chapter{\tt Properties of morphisms of schemes}\label{ch-flat}

In this chapter, we discuss several topics which mostly have to do with 
families of schemes and the behavior of certain invariants in families.
The notion of a flat family seems to be a reasonable rigorous definition of 
the intuitive idea of a "nice" family of schemes. Base change on the other hand 
has to do with the way cohomology on the total space of a family is related 
to the cohomology on the fibres. Grothendieck's solution of the 
base change problem is an elegant example of the success of his philosophy.
These topics beside being very natural 
to study appear both in the relative duality theory and in the construction 
of Hilbert and Quotient schemes. 
A thorough treatment of flatness and base change can be found 
in Grothendieck's EGA IV part 2 and part 3 [\ref{EGA}] and EGA III part 2 [\ref{EGA}] respectively.

\section{Flat sheaves}
 Let $A$ be a ring and $M$ an $A$-module. Remember that $M$ is said to be 
flat over $A$ if the functor $M\otimes_A -\colon \mathfrak{M}(A)\to \mathfrak{M}(A)$. 
It is well-known that $M$ is flat over $A$ if and only if 
for any (finitely generated) ideal $I$ of $A$, the natural map $M\otimes_A I\to M\otimes_A A$ 
is injective. It is also well-known that flatness is a local condition, that is, 
$M$ is flat over $A$ if and only if $M_P$ is flar over $A_P$ for any prime 
ideal $P$ of $A$ if and only if $M_P$ is flat over $A_P$ for any maximal ideal 
$P$ of $A$. It is well-known that any free $A$-module $M$ is projective, and 
any projective $A$-module $M$ is flat, and that projectivity and flatness are equivalent 
if $M$ is finitely generated and $A$ is Noetherian. Moreover, if $A$‌ is a 
local Noetherian ring, then the 
three notions of freeness, projectivity, and flatness coincide for finitely 
generated modules. 

If we denote the left derived functors of the right exact functor 
$M\otimes_A-\colon \mathfrak{M}(A)\to \mathfrak{M}(A)$ by $\Tor^A_p(M,-)$ 
then it is easy to see that $M$ is flat over $A$ if and only if 
$\Tor^A_p(M,N)=0$ for any $A$-module $N$ if and only if 
$\Tor^A_1(M,N)=0$ for any $A$-module $N$.

A flat $A$-module $M$ is called faithfully flat if exactness of  
$$
0\to {N'}\otimes_A M\to {N}\otimes_A M \to {N''}\otimes_A M\to 0
$$ 
implies exactness of $0\to {N'}\to {N} \to {N''}\to 0$ where the latter is a 
sequence of $A$-modules.\\

\begin{defn}
Let $f\colon X\to Y$ be a morphism of schemes and let $\x{F}$ be an $\x{O}_X$-module.
We say that $\x{F}$ is flat at $x\in X$ over $Y$ if  $\x{F}_x$ is a flat 
$\x{O}_{f(x)}$-module. If $\x{F}$ is flat at every $x\in X$ over $Y$, we simply 
say that $\x{F}$ is flat over $Y$. If $\x{O}_X$ is flat over $Y$, we say that $f$ is flat.
 But if $f$ is the identity, we say that $\x{F}$ is flat 
over $X$.
\end{defn}

Note that in the definition, $\x{O}_x$ is an algebra over $\x{O}_{f(x)}$ 
and so $\x{F}_x$ being a module over $\x{O}_x$ is also naturally a module 
over $\x{O}_{f(x)}$. Moreover, note that flatness is a local condition both 
on $Y$ and $X$.\\

\begin{thm}\label{t-flat-affine}
Let $f\colon X=\Spec B \to Y=\Spec A$ be a morphism of schemes and $\x{F}=\widetilde{M}$ 
for some $B$-module $M$. Then, $\x{F}$ is flat over $Y$ if and only if $M$ is a 
flat $A$-module.
\end{thm}
\begin{proof}
The statement follows from the following string of equivalences:
 $\x{F}$ is flat over $Y$ if and only if for any $x\in X$,  
$M_x$ is a flat $A_{f(x)}$-module if and only if for any $x\in X$ and for any ideal $I$ 
of $A_{f(x)}$ the map $M_x\otimes_{A_{f(x)}} I\to M_x\otimes_{A_{f(x)}} A_{f(x)}$ is injective 
if and only if for any ideal $J$ of $A$ and any  $x\in X$ the map $M_x\otimes_{A_{f(x)}} J_{f(x)}\to M_x\otimes_{A_{f(x)}} A_{f(x)}$ is injective if and only if 
for any ideal $J$ of $A$ and any  $x\in X$ the map $(M\otimes_A J)\otimes_B B_x \to (M\otimes_A A)\otimes_B B_x$ is injective 
if and only if 
for any ideal $J$ of $A$ the map $M\otimes_A J\to M\otimes_A A$ is injective if and only if 
$M$ is flat over $A$.\\ 
\end{proof}

\begin{thm}
Let $f\colon X\to Y$ be a finite morphism of Noetherian schemes, 
and $\x{F}$ a coherent $\x{O}_X$-module. Then, 
$\x{F}$‌ is flat over $Y$ if and only if $f_*\x{F}$ is locally free.
In particular, if $X$ is a Noetherian scheme and $\x{F}$ a coherent $\x{O}_X$-module, then 
$\x{F}$‌ is flat over $X$ if and only if $\x{F}$ is locally free.
\end{thm}
\begin{proof}
Since the statement is local we may assume that 
$X=\Spec B$, $Y=\Spec A$, and $\x{F}=\widetilde{M}$ for some $B$-module $M$.
By Theorem \ref{t-flat-affine}, $\x{F}$ is flat over $Y$ if and only if 
$M$ is a flat $A$-module. Since $f$ is finite, $B$ is a finitely generated $A$-module 
hence $M$ is a finitely generated $A$-module too.
On the other hand, 
$M$ is a flat $A$-module if and only if $M_y$ is flat over $\x{O}_y$ for every 
$y\in Y$ which is equivalent to 
the freeness of $M_y$ over $\x{O}_y$ for every $y\in Y$ because $\x{O}_y$ is local 
Noetherian and $M_y$ is finitely generated.\\
\end{proof}

\begin{exa}
The previous theorem shows that any open immersion is a flat morphism 
but closed immersions are rarely flat.\\
\end{exa}

\begin{constr}\label{c-base-change} 
Let $f\colon X\to Y$ and $g\colon S\to Y$ be morphisms of schemes which 
induce the projection $X_S:=S\times_Y X \to S$. So, we get a commutative 
diagram 
$$
\xymatrix{
X_S \ar[r]^{g_S}\ar[d]^{f_S} & X\ar[d]^f\\
S \ar[r]^g & Y
}
$$
which we refer to as a base change - or we just say that we take a base change 
$g\colon S\to Y$. If $\x{F}$ is an $\x{O}_X$-module, we define $\x{F}_S:=g_S^*\x{F}$.
If $\x{G}$ is an $\x{O}_S$-module, then we use the notation $\x{F}\otimes_Y \x{G}$ 
to mean $g_S^*\x{F}\otimes_{\x{O}_{X_S}}f_S^*\x{G}$. 
If  $Y=\Spec A$, $S=\Spec C$, and 
 $\x{G}=\widetilde{N}$, then we further abuse notation 
by writing $\x{F}\otimes_Y M$ to mean $\x{F}\otimes_Y \x{G}$.
If in addition $X=\Spec B$ and $\x{F}=\widetilde{M}$, then it is easy to see that $\x{F}\otimes_Y \x{G}=\widetilde{M\otimes_AN}$.
\end{constr}
\vspace{0.5cm}

\begin{thm}\label{t-flat-properties}
Let $f\colon X\to Y$ be a morphism of schemes, and $\x{F}$ a quasi-coherent $\x{O}_X$-module.
Then,

$(i)$  $\x{F}$ is flat over $Y$ if and only if for any base change $g\colon S\to Y$ 
and any exact sequence $0\to \x{G}'\to \x{G} \to \x{G}''\to 0$ of quasi-coherent $\x{O}_S$-modules, 
the induced sequence $0\to \x{F}\otimes_Y \x{G}'\to \x{F}\otimes_Y\x{G} \to \x{F}\otimes_Y\x{G}''\to 0$ is exact,

$(ii)$ statement $(i)$ holds if we only consider $S=Y$,

$(iii)$ statement $(i)$ holds if we only consider schemes $S=\Spec \x{O}_y$ 
where $y\in Y$,

$(iv)$ flatness is preserved under base, that is, if $\x{F}$ is flat over $Y$, and if $g\colon S\to Y$ is a base change, then $\x{F}_S$ is flat over $S$,

$(v)$  let $g\colon S\to Y$ be a base change, $h\colon S\to Z$ a morphism, 
and $\x{G}$ a quasi-coherent $\x{O}_S$-module which is flat over $Z$. If 
$\x{F}$ is flat over $Y$, then $\x{F}\otimes_Y\x{G}$ is flat over $Z$.
\end{thm}
\begin{proof}
Since all the statements are local on $S$, $X$, and $Y$ we may assume that 
$X=\Spec B$, $Y=\Spec A$, $S=\Spec C$, $\x{F}=\widetilde{M}$,
 $\x{G}'=\widetilde{N'}$, $\x{G}=\widetilde{N}$, and $\x{G}''=\widetilde{N''}$. 

(i) If $M$ is flat over $A$ and if $0\to {N'}\to {N} \to {N''}\to 0$ is an 
exact sequence of $C$-modules, then obviously $0\to {N'}\otimes_A M\to {N}\otimes_A M \to {N''}\otimes_A M\to 0$ is exact. Conversely, if $0\to {N'}\otimes_A M\to {N}\otimes_A M \to {N''}\otimes_A M\to 0$ is exact for any exact sequence $0\to {N'}\to {N} \to {N''}\to 0$ of $C$-modules, 
then the same statement holds for any exact sequence of $A$-modules which means the flatness of 
$M$ over $A$. 

(ii) The only if part follows from (i). The if part is a simple consequence of 
Theorem \ref{t-flat-affine}.

(iii) The only if part follows from (i). For the converse, if $0\to {N'}\to {N} \to {N''}\to 0$ is an exact sequence of $A$-modules, then for any $y\in Y$, 
$0\to {N'}_y\to {N}_y \to {N''}_y\to 0$ is an exact sequence of $\x{O}_y$-modules and the exactness 
of $0\to {N'}_y\otimes_{A} M\to {N}_y\otimes_{A} M \to {N''}_y\otimes_{A} M\to 0$ is the same as 
the exactness of $0\to ({N'}\otimes_A M)_y\to ({N}\otimes_A M)_y \to ({N''}\otimes_A M)_y\to 0$. 
This exactness for every $y$ implies the exactness of 
$0\to {N'}\otimes_A M\to {N}\otimes_A M \to {N''}\otimes_A M\to 0$ which shows that 
$M$ is flat over $A$.

(iv) This follows from (i) since any base change $T\to S$ induces a base change 
$T\to Y$ and for any quasi-coherent $\x{O}_T$-module $\x{L}$ we have 
$\x{F}_S\otimes_S\x{L}\simeq \x{F}\otimes_Y\x{L}$.

(v) Let $V\to Z$ be a base change, and let $0\to \x{L}'\to \x{L}\to \x{L}''\to 0$
be an exact sequence of quasi-coherent $\x{O}_V$-modules. Put $T=V\times_ZS$. 
Since $\x{G}$ is flat over $Z$, by (i), the sequence 
$0\to \x{G}\otimes_Z\x{L}'\to \x{G}\otimes_Z\x{L}\to \x{G}\otimes_Z\x{L}''\to 0$ 
is exact and since $\x{F}_S$ is flat over $S$, by (iv), the sequence 
$$
0\to \x{F}_S\otimes_S(\x{G}\otimes_Z\x{L}')\to \x{F}_S\otimes_S(\x{G}\otimes_Z\x{L})\to \x{F}_S\otimes_S(\x{G}\otimes_Z\x{L}'')\to 0
$$ 
is also exact. The result then follows from the isomorphism
$$
\x{F}_S\otimes_S(\x{G}\otimes_Z\x{L})\simeq (\x{F}\otimes_Y\x{G})\otimes_Z\x{L}
$$ 
and similar isomorphisms for $\x{L}'$ and $\x{L}''$.\\
\end{proof}

\begin{rem}
Note that in the theorem part (v), if we take $g$ to be the identity and $\x{G}=\x{O}_Y$, then the theorem says that $\x{F}$ flat over $Y$ and $h$ flat implies that $\x{F}$ is flat over $Z$. 
On the other hand, if we assume $h=g$, and $\x{F}$ and $\x{G}$ flat over 
$Y$, then $\x{F}\otimes_Y\x{G}$ is flat over $Y$.
\end{rem}

\begin{thm}\label{t-flat-exact-preserving}
Let $f\colon X\to Y$ be a morphism of schemes and $0\to \x{F}'\to \x{F} \to \x{F}''\to 0$ an exact sequence of quasi-coherent $\x{O}_X$-modules where $\x{F}''$ is flat over $Y$. Then, for any 
base change $g\colon S\to Y$ and any 
quasi-coherent $\x{O}_S$-module $\x{G}$ the sequence 
$$
0\to \x{F}'\otimes_Y \x{G}\to \x{F}\otimes_Y\x{G} \to \x{F}''\otimes_Y\x{G}\to 0
$$ 
is exact.
\end{thm}
\begin{proof}
Since the statement is local we may assume that $X=\Spec B$, $Y=\Spec A$, $S=\Spec C$, $\x{G}=\widetilde{N}$,
$\x{F}'=\widetilde{M'}$, $\x{F}=\widetilde{M}$, and $\x{F}''=\widetilde{M''}$.
The exactness of $0\to {M'}\to {M}\to {M''}\to 0$ gives a long exact sequence 
$$
0=\Tor^A_1(M'',N)\to {M'}\otimes_A N\to {M}\otimes_A N\to {M''}\otimes_A N\to 0
$$  
which implies the claim.\\
\end{proof}

Let $f\colon X\to Y$ be a morphism of schemes and $\x{G}$ and $\x{E}$ be   
$\x{O}_Y$-modules. For any open subset $U\subseteq Y$, any morphism $\x{G}|_U\to \x{E}|_U$ 
naturally induces a morphism $f^*\x{G}|_{f^{-1}U}\to f^*\x{E}|_{f^{-1}U}$, that is, 
an element of $f_*\xHom_{\x{O}_X}(f^*\x{G},f^*\x{E})(U)$. So, we get a natural 
morphism 
$$
\xHom_{\x{O}_Y}(\x{G},\x{E}) \to f_*\xHom_{\x{O}_X}(f^*\x{G},f^*\x{E})
$$
and since $f^*$ is left adjoint to $f_*$, we get a bifunctorial morphism 
$$
\alpha\colon f^*\xHom_{\x{O}_Y}(\x{G},\x{E}) \to \xHom_{\x{O}_X}(f^*\x{G},f^*\x{E})
$$\\

\begin{thm}
If $Y$ is Noetherian, $f$ is flat, $\x{G}$ is coherent, and $\x{E}$‌ is quasi-coherent, then 
$\alpha$ is an isomorphism.
\end{thm}
\begin{proof}
The statement is local so we may assume that $X=\Spec B$, $Y=\Spec A$, $\x{G}=\widetilde{N}$, 
and $\x{E}=\widetilde{L}$. The morphism $\alpha$ corresponds to the map 
$$
\Hom_A(N,L)\otimes_A B\to \Hom_B(N\otimes_A B,L\otimes_A B)
$$
It is easy to see that this is an isomorphism if $N=A$ or more generally if $N=A^r$ 
for some $r$. Since $A$ is Noetherian and $N$ is finitely generated over $A$, 
there is an exact sequence $A^s\to A^r\to N\to 0$ which leads to a commutative 
diagram
$$
\xymatrix{
0 \ar[r] & \Hom_A(N,L)\otimes_A B \ar[r]\ar[d] & \Hom_A(A^r,L)\otimes_A B\ar[r]\ar[d] & 
\Hom_A(A^s,L)\otimes_A B\ar[d]\\
0 \ar[r] & \Hom_B(N\otimes_A B,L\otimes_A‌ B) \ar[r] & 
\Hom_B(A^r\otimes_A B,L\otimes_A B)\ar[r] & 
\Hom_B(A^s\otimes_A B,L\otimes_A B)
}
$$
with exact rows which implies the theorem since the right and the middle vertical maps are isomorphisms.\\
\end{proof}


\section{Flat base change}

\begin{constr}
Under the notation of Construction \ref{c-base-change}, for any sheaf $\x{M}$ on $X_S$ we have two spectral sequences given by the compositions 
$h:=fg_S=gf_S$ as follows: 
$$
E^{p,q}_2=R^pf_*R^q{g_S}_*\x{M}\implies R^{p+q}h_*\x{M}
$$
$$
E^{m,n}_2=R^mg_*R^n{f_S}_*\x{M}\implies R^{m+n}h_*\x{M}
$$

The first spectral sequence gives natural morphisms 
$$
E^{p,0}_2=R^pf_*({g_S}_*\x{M})\to E^{p,0}_{\infty}\to R^{p}h_*\x{M}
$$ 
and the second spectral sequence gives morphisms 
$$
R^{p}h_*\x{M}\to E^{0,p}_{\infty}\to E^{0,p}_2=g_*R^p{f_S}_*\x{M}
$$
Combining the two gives a natural morphism 
$R^pf_*({g_S}_*\x{M})\to g_*R^p{f_S}_*\x{M}$. By taking $\x{M}=\x{F}_S$ 
we get a natural morphism $R^pf_*({g_S}_*\x{F}_S)\to g_*R^p{f_S}_*\x{F}_S$. 
On the other hand, the natural morphism $\x{F}\to {g_S}_*\x{F}_S$ induces 
a morphism $R^pf_*\x{F}\to g_*R^p{f_S}_*\x{F}_S$ and since $g^*$ is a 
left adjoint of $g_*$ the latter 
morphism corresponds to a morphism 
$$
\phi\colon g^*R^pf_*\x{F}\to R^p{f_S}_*\x{F}_S
$$ 
which is called the base change morphism of cohomology of the base change $g$. 
In particular, if $X$, $Y$, and $S$ are Noetherian schemes with $Y=\Spec A$ and $S=\Spec C$ 
affine, and if $\x{F}$ is quasi-coherent, then the base change morphism corresponds 
to the map
$$
\psi\colon H^p(X,\x{F})\otimes_A C \to H^p(X_S,\x{F}_S)
$$
\end{constr}
\vspace{0.5cm}

\begin{thm}[Flat base change]\label{t-flat-base-change}
If $X$, $Y$, and $S$ are Noetherian schemes, $g$ is flat, $f$ is separated and of finite type, 
and $\x{F}$ is quasi-coherent, then the base change morphism of cohomology $\phi$ is an isomorphism.
\end{thm}
\begin{proof} 
The problem is local on both $Y$ and $S$ so it is enough to prove that 
$\psi$ is an isomorphism. Let $\mathcal{U}=(U_i)_{i\in I}$ be an open 
covering of $X$ by finitely many open affine subschemes. The open subschemes 
$V_i=S\times_YU_i$ give an open affine covering $\mathcal{V}$ of $X_S$. Since $X$ and $X_S$ are  
Noetherian and separated, and since $\x{F}$ and $\x{F}_S$ are quasi-coherent, 
the cohomology  objects of the \v{C}ech complex 
$$
{C}^\bullet:\hspace{1cm} 0 \to {C}^0(\mathcal{U},\x{F}) \to {C}^{1}(\mathcal{U},\x{F}) \to \cdots
$$
are the cohomology groups $H^p(X,\x{F})$, and the cohomology objects of the \v{C}ech complex 
$$
{C}^\bullet_S:\hspace{1cm} 0 \to {C}^0(\mathcal{V},\x{F}_S) \to {C}^{1}(\mathcal{V},\x{F}_S) \to \cdots
$$
are the cohomology groups $H^p(X_S,\x{F}_S)$.

On the other hand, for each $p$, ${C}^p(\mathcal{V},\x{F}_S)\simeq {C}^p(\mathcal{U},\x{F})\otimes_AC$. Moreover, since $g$ is flat, $C$ is flat over $A$. Therefore, the cohomology objects of $C^\bullet_S$ are the cohomology objects 
of $C^\bullet$ tensored with $C$, that is, $\check{H}^p(\mathcal{V},\x{F}_S)\simeq \check{H}^p(\mathcal{U},\x{F})\otimes_AC$ which implies that $\psi$ is an isomorphism.\\
\end{proof}

\begin{rem}
One important situation in which one can apply Theorem \ref{t-flat-base-change} 
is when $g$ is the natural morphism $S=\Spec \x{O}_y\to Y$ for some $y\in Y$.  
In the base change section, we will study the base change morphism of cohomology induced by 
base changes of the form $g\colon \Spec k(y)\to Y$. In general, this kind 
of morphism is not flat so the theorem does not apply. However, if 
$y$ is the generic point of an integral component of $Y$, then $g$ is flat 
because in this case $k(y)=\x{O}_y$.
\end{rem}

\section{The $T^p_{\x{F}}$ functors and the Grothendieck complex} 

\begin{defn}\label{d-flat-T}
Let $f\colon X\to Y=\Spec A$ be a projective morphism of Noetherian schemes, 
and let $\x{F}$ be a coherent sheaf on $X$. We define 
the functors $T^p_{\x{F}}(-)\colon \mathfrak{M}(A)\to \mathfrak{M}(A)$ associated with this setting as 
$$
T^p_{\x{F}}(M):=H^p(X,\x{F}\otimes_YM)
$$
Note that the functorial homomorphism $M\simeq \Hom_A(A,M)\to \Hom_A(T^p_{\x{F}}(A),T^p_{\x{F}}(M))$ 
corresponds to a functorial homomorphism 
$\beta\colon T^p_{\x{F}}(A)\otimes_A M\to T^p_{\x{F}}(M)$. 
\end{defn}

Note that if $\x{F}$ is flat over $Y$, any exact sequence $0\to M'\to M \to M'' \to 0$ of $A$-modules gives an 
exact sequence $0\to \x{F}\otimes_YM'\to \x{F}\otimes_YM \to \x{F}\otimes_YM'' \to 0$ 
and by taking the associated long exact sequence one can see that the functors 
$T^p_{\x{F}}$ form a $\delta$-functor and they are always exact in the middle.

Before further studying the above functors, we replace the geometric situation 
with an algebraic one.\\

\begin{thm}\label{t-flat-complex}
Under the assumptions of Definition \ref{d-flat-T} with $\x{F}$ flat over $Y$, there is a finite length complex 
$$
L^\bullet: \hspace{1cm} 0\to L^0 \to L^1 \to \cdots
$$ 
of  finitely generated flat $A$-modules, 
called the Grothendieck complex, 
such that for any $p$ and any $A$-module $M$, $T^p_{\x{F}}(M)\simeq h^p(L^\bullet\otimes_AM)$ 
where $h^p(L^\bullet\otimes_AM)$ is the $p$-th cohomology object of the complex $L^\bullet \otimes_A M$. Moreover, the isomorphism is functorial in $M$.
\end{thm}
\begin{proof}
 Let $\mathcal{U}=(U_i)_{i\in I}$ be an open 
covering of $X$ by finitely many open affine subschemes. 
Consider the \v{C}ech complex 
$$
{C}^\bullet(\mathcal{U},\x{F}):\hspace{1cm} 0 \to {C}^0(\mathcal{U},\x{F}) \to {C}^{1}(\mathcal{U},\x{F}) \to \cdots
$$
Then, for any $A$-module $M$, we have 
$$
{C}^\bullet(\mathcal{U},\x{F}\otimes_YM)\simeq {C}^\bullet(\mathcal{U},\x{F})\otimes_AM
$$
where the left hand side is the \v{C}ech complex of $\x{F}\otimes_YM$. In particular, 
this means that $H^p(X,\x{F}\otimes_YM)\simeq h^p({C}^\bullet(\mathcal{U},\x{F})\otimes_AM)$. 
Even though the modules ${C}^p(\mathcal{U},\x{F})$ are flat over $A$ but they are not necessarily 
finitely generated. However, the cohomology objects of ${C}^\bullet(\mathcal{U},\x{F})$ 
are finitely generated over $A$.

We will construct $L^\bullet$ out of ${C}^\bullet(\mathcal{U},\x{F})$ as follows. 
For $p\gg 0$ put $L^p=0$. We use descending induction, so assume that we have constructed 
$L^\bullet$ up to some $p>0$ so that we have a commutative diagram 
$$
\xymatrix{
 & L^p \ar[r]^{e^p}\ar[d]^{c^p} & L^{p+1} \ar[d]^{c^{p+1}} \ar[r] &\cdots\\
{C}^{p-1}(\mathcal{U},\x{F}) \ar[r]^{d^{p-1}}& {C}^p(\mathcal{U},\x{F}) \ar[r]^{d^p}& {C}^{p+1}(\mathcal{U},\x{F}) \ar[r]& \cdots
}
$$
such that  the $L^i$ are finitely generated flat $A$-modules, and for any $i>p$,  $h^i(L^\bullet)\simeq h^i({C}^\bullet(\mathcal{U},\x{F}))$ and the natural surjective map 
$\ker d^p\to h^p({C}^\bullet(\mathcal{U},\x{F}))$ induces a surjective map 
$a^p\colon \ker e^p\to h^p({C}^\bullet(\mathcal{U},\x{F}))$. If $p=1$, we let 
$L^0=\ker a^1\oplus H^0(X,\x{F})$ and the maps $e^0$ and $c^0$ are given by the 
obvious first and second projections. We will show below that $L^0$ is flat. 
Meanwhile, if $p>1$ we proceed as follows.
Let $L'\to L^p$ be a homomorphism from a finitely generated free $A$-module which surjects onto $\ker a^p$. The induced 
map $L'\to {C}^{p}(\mathcal{U},\x{F})$ factors through a map $L'\to {C}^{p-1}(\mathcal{U},\x{F})$ because $c^p(\ker a^p)\subseteq \im d^{p-1}$ and because $L'$ is a projective $A$-module. 
On the other hand, let $L''$ be a free finitely generated $A$-module with a map 
$L''\to \ker d^{p-1}$ which induces a surjection $L''\to h^{p-1}({C}^\bullet(\mathcal{U},\x{F}))$. 
By taking the zero map $L''\to L^p$ and putting $L^{p-1}:=L'\oplus L''$ we 
have succeeded in completing the induction step. To summarise, we have a morphism 
$L^\bullet\to {C}^\bullet(\mathcal{U},\x{F})$ inducing isomorphisms on the cohomology 
objects and $L^p$ is a finitely generated flat $A$-module for any $p>0$.

Now we show that $L^0$ is a flat $A$-module. Let $D^\bullet$ be the complex in which 
$D^p=L^p\oplus C^{p-1}(\mathcal{U},\x{F})$ and the map $b^p\colon D^p\to D^{p+1}$ 
is given by $b^p(s,t)=(e^p(s),d^{p-1}(t)+(-1)^pc^p(s))$. It is easy to see that 
in fact $D^\bullet$ is an exact sequence. Since $D^p$ is flat for any $p>0$, and 
since $D^0=L^0$, we deduce that $L^0$ is also flat. 

Now, if $M$ is an $A$-module, then we have a morphism 
$L^\bullet\otimes_AM\to {C}^\bullet(\mathcal{U},\x{F})\otimes_AM$ which induces isomorphisms 
on the cohomology objects as follows. First, $M$ can be regarded as the direct limit 
of its finitely generated submodules and this reduces the problem to the finitely 
generated case because tensor product and taking cohomology both commute with direct 
limits. Second, if $M$ is a free finitely generated $A$-module, then obviously  
$L^\bullet\otimes_AM\to {C}^\bullet(\mathcal{U},\x{F})\otimes_AM$ induces isomorphisms 
on the cohomology objects. 
Finally, for the general finitely generated $M$ we take an exact sequence $0\to K\to F\to M\to 0$ of $A$-modules with $F$ free 
and finitely generated. Since every object in the two complexes $L^\bullet$ and 
${C}^\bullet(\mathcal{U},\x{F})$ are flat we get a commutative diagram of 
complexes
$$
\xymatrix{
0\ar[r]& L^\bullet\otimes_AK \ar[r]\ar[d]& 
L^\bullet\otimes_AF\ar[r]\ar[d]& 
L^\bullet\otimes_AM \ar[d]\ar[r]& 0\\ 
0\ar[r]& {C}^\bullet(\mathcal{U},\x{F})\otimes_AK \ar[r] &
{C}^\bullet(\mathcal{U},\x{F})\otimes_A\ar[r] F& 
{C}^\bullet(\mathcal{U},\x{F})\otimes_AM \ar[r] &0
}
$$
with exact rows whose long exact sequences allow descending induction on $p$ 
to prove the isomorphisms for $M$.\\
\end{proof}

\section{Exactness properties of $T^p_{\x{F}}$}

In this section, we assume the setting and notation of Definition \ref{d-flat-T} 
and also assume that $\x{F}$ is flat over $Y$. 
We assume that $L^\bullet$ is the Grothendieck complex of Theorem 
\ref{t-flat-complex}.
We would like to know when the functors $T^p_{\x F}(-)$ are left exact, right exact, 
and exact.\footnote{
In this section, the theorem on the formal functions is used for which we refer to Hartshorne [\ref{Hartshorne}, III, \S 11].}

For 
any complex $N^\bullet$, we introduce the notation 
$W^p(N^\bullet):={\rm{coker}}~ N^{p-1}\to N^{p}$ which then induces exact sequences 
$$
N^{p-1}\to N^p\to W^p(N^\bullet) \to 0
$$
$$
0\to h^p(N^\bullet)\to W^p(N^\bullet) \to N^{p+1}
$$ 

\begin{thm}\label{t-flat-left-exact}
The following are equivalent:

$(i)$ the functor $T^p_{\x{F}}$ is left exact,

$(ii)$ the module $W^p(L^\bullet)$ is flat over $A$,

$(iii)$ there is a unique finitely generated $A$-module 
$Q$ such that we have a functorial isomorphism
$T^p_{\x{F}}(M)\simeq \Hom_A(Q,M)$ for every $A$-module $M$.
\end{thm}
\begin{proof}

(i) $\Leftrightarrow$ (ii): $T^p_{\x{F}}$ is left exact if and only if for any 
exact sequence $0\to M'\to M$ of $A$-modules the induced 
map $T^p_{\x{F}}(M')\to T^p_{\x{F}}(M)$ is injective. We have a 
commutative diagram 
$$
\xymatrix{
0\ar[r]& T^p_{\x{F}}(M')\simeq h^p(L^\bullet\otimes_AM')\ar[r] \ar[d] & 
W^p(L^\bullet\otimes_AM')\ar[r]\ar[d] & 
L^{p+1}\otimes_AM' \ar[d] \\
0\ar[r]& T^p_{\x{F}}(M)\simeq h^p(L^\bullet\otimes_AM)\ar[r]& 
W^p(L^\bullet\otimes_AM)\ar[r] & 
L^{p+1}\otimes_AM \\
}
$$ 
with exact rows and satisfying $W^p(L^\bullet\otimes_AM')\simeq W^p(L^\bullet)\otimes_AM'$ 
and $W^p(L^\bullet\otimes_AM)\simeq W^p(L^\bullet)\otimes_AM$. Since $L^{p+1}$ 
is flat the right vertical map is always injective hence the middle vertical map is injective  if and only if the left one is injective. That is, $T^p_{\x{F}}(-)$ is left exact if and only if 
$W^p(L^\bullet)\otimes_A-$ is left exact. The latter is equivalent to the flatness 
(hence projectivity) of $W^p(L^\bullet)$.
  
(ii) $\implies$ (iii): Since $W^p(L^\bullet)$ and $L^{p+1}$ are both flat and finitely 
generated they satisfy 
$$
\Hom_A(\Hom_A(W^p(L^\bullet),A),M)\simeq W^p(L^\bullet)\otimes_AM
$$
and 
$$
\Hom_A(\Hom_A(L^{p+1},A),M)\simeq L^{p+1}\otimes_AM
$$ 
for any $A$-module $M$ (these isomorphisms are easier to recognize on the 
level of the associated sheaves). If we put 
$Q:={\rm{coker}}~ \Hom_A(L^{p+1},A)\to \Hom_A(W^p(L^\bullet),A)$ 
then since $\Hom_A(-,M)$ is left exact, we get an exact sequence 
$$
0 \to \Hom_A(Q,M)\to W^p(L^\bullet)\otimes_AM \to L^{p+1}\otimes_AM
$$
for any $A$-module $M$. Comparing this with the above sequences 
proves the claim. 

For the uniqueness of $Q$: if $Q'$ is any other finitely generated $A$-module 
satisfying $T^p_{\x F}(M)\simeq \Hom_A(Q',M)$ for any $A$-module $M$, then having 
the isomorphisms 
$$
\Hom_A(N,\Hom_A(Q,M))\simeq \Hom_A(N\otimes_A Q,M)
$$ 
and 
$$
\Hom_A(N,\Hom_A(Q',M))\simeq \Hom_A(N\otimes_A Q',M)
$$ 
for any $A$-modules $M,N$ 
means that the two functors $-\otimes_AQ$ and $-\otimes_AQ'$ are both 
left adjoints of $T^p_{\x F}(M)$ hence they must be isomorphic. This is 
possible only if $Q\simeq Q'$.

(iii) $\implies$ (i): Since $\Hom_A(Q,-)$ is left exact, the statement is clear.   
\end{proof}

\begin{thm}\label{t-flat-right-exact}
The following statements are equivalent:

$(i)$ the functor $T^p_{\x{F}}$ is right exact,

$(ii)$ the map $\beta\colon T^p_{\x{F}}(A)\otimes_A M\to T^p_{\x{F}}(M)$ 
is surjective for any $A$-module $M$,

$(iii)$ the map $\beta\colon T^p_{\x{F}}(A)\otimes_A M\to T^p_{\x{F}}(M)$ 
is an isomorphism for any $A$-module $M$.
\end{thm}
\begin{proof}
Since the two functors $T^p_{\x F}$ and $\otimes$ both commute with 
direct limits, it is enough to only consider finitely generated $A$-modules.
Obviously, (iii) $\implies$ (i) and (iii) $\implies$ (ii).

For each finitely generated $A$-module $M$ we can take an exact sequence 
$A^s\to A^r\to M\to 0$ which induces a commutative 
diagram 
$$
\xymatrix{
 T^p_{\x{F}}(A)\otimes_A A^s\ar[r] \ar[d] & 
T^p_{\x{F}}(A)\otimes_A A^r\ar[r]\ar[d] & 
T^p_{\x{F}}(A)\otimes_A M \ar[d]^\beta\ar[r] & 0\\
 T^p_{\x{F}}(A^s)\ar[r]& 
T^p_{\x{F}}(A^r)\ar[r]^\alpha & 
T^p_{\x{F}}(M) &\\
}
$$ 
where the upper row is exact and the left and middle vertical maps are isomorphisms.
If $T^p_{\x F}$ is right exact, then a simple diagram chase shows that $\beta$ is an isomorphism 
hence (i) $\implies$ (iii). On the other hand, if $\beta$ is always surjective then
for any surjective homomorphism $M\to M''$ we get a surjective map 
$T^p_{\x{F}}(A)\otimes_A M\to T^p_{\x{F}}(A)\otimes_A M''$ which immediately 
implies that the corresponding map $T^p_{\x{F}}(M)\to T^p_{\x{F}}(M'')$ is also 
surjective hence $T^p_{\x{F}}$ is right exact, that is,
(ii) $\implies$ (i).\\
\end{proof}

\begin{cor}\label{t-flat-T-exact}
$T^p_{\x{F}}$ is exact if and only if $T^p_{\x{F}}$ is right exact and $T^p_{\x{F}}(A)$ 
is a finitely generated flat $A$-module.
\end{cor}
\begin{proof}
 $T^p_{\x{F}}$ is exact if and only if it is right exact and left exact if and only if  $T^p_{\x{F}}$ is right exact and $T^p_{\x{F}}(A)\otimes_A -$ is left exact if and only if
 $T^p_{\x{F}}$ is right exact and $T^p_{\x{F}}(A)$ is a flat $A$-module. Finite generation always 
holds under the assumptions.\\  
\end{proof}

\begin{rem}\label{r-T0-Tr}
Note that $T^0_{\x{F}}$ is always left exact. This in particular, means that 
there is some $A$-module $Q$ such that $H^0(X,\x{F}\otimes_YM)\simeq \Hom_A(Q,M)$ 
for any $A$-module $M$. 
On the other hand, if $r=\max\{\dim X_y \mid y\in Y\}$, then $T^r_{\x{F}}$ 
is always right exact because $T^p_{\x{F}}=0$ for any $p>r$. We will see that this implies that $R^rf_*\x{F}$ commutes 
with base change.
\end{rem}

\begin{thm}\label{t-base-change-right-exact}
Assume that $A$ is a local ring, $y\in Y$ is the closed point, $\mathfrak{m}$ is the 
maximal ideal of $A$, and $k(y)$ is the residue field of $y$. 
Then, the following are equivalent:

$(i)$  the functor $T^p_{\x{F}}$ is right exact,

$(ii)$ $T^p_{\x{F}}(A)\otimes_A k(y)\to T^p_{\x{F}}(k(y))$ is surjective,

$(iii)$ $T^p_{\x{F}}(A/\mathfrak{m}^l)\to T^p_{\x{F}}(k(y))$ is surjective 
for every $l>1$.
\end{thm}
\begin{proof}
(i) $\implies$ (iii): We have the commutative diagram
$$
\xymatrix{
T^p_{\x{F}}(A)\otimes_A A/\mathfrak{m}^l\ar[r]\ar[d] & T^p_{\x{F}}(A/\mathfrak{m}^l)\ar[d]\\
T^p_{\x{F}}(A)\otimes_A k(y) \ar[r] & T^p_{\x{F}}(k(y))
}
$$ 
in which the left vertical map is surjective. Applying Theorem \ref{t-flat-right-exact} 
implies the surjectivity of the lower horizontal map which then implies the claim.

(iii) $\implies$ (ii): If $T^p_{\x{F}}(A/\mathfrak{m}^l)\to T^p_{\x{F}}(k(y))$ is surjective 
for every $l>1$, then by taking inverse limits we get a surjective map 
$$
\varprojlim_l T^p_{\x{F}}(A/\mathfrak{m}^l)\to \varprojlim_l T^p_{\x{F}}(k(y))=T^p_{\x{F}}(k(y))
$$
On the other hand, by the theorem on the formal functions, the left hand side is just the completion of $T^p_{\x{F}}(A)$. So the map corresponds to the map 
$T^p_{\x{F}}(A)\to T^p_{\x{F}}(k(y))$. 
Since completion is a faithfully exact functor, the latter map must be surjective. Therefore, 
$T^p_{\x{F}}(A)\otimes_A k(y)\to T^p_{\x{F}}(k(y))$ is also surjective.

(ii) $\implies$ (i): 
 Assume that 
$M$ is an $A$-module of finite length (i.e. $M$ is both Noetherian and Artinian). 
If $M$ has length one, then $M\simeq k(y)$ so in this case the map 
$T^p_{\x{F}}(A)\otimes_{A} M\to T^p_{\x F}(M)$ is surjective. If the length 
of $M$ is larger than one, then one can find an exact sequence 
$0\to M'\to M\to M\to 0$ of $A$-modules such that $M',M''$ have smaller length 
than $M$. Using this sequence and applying induction on the length one easily 
deduces that the map $T^p_{\x{F}}(A)\otimes_{A} M\to T^p_{\x F}(M)$ is again 
surjective.

Now let $M$ be a finitely generated $A$-module. Then, for any $l>0$, the module 
$M/\mathfrak{m}^lM$ has finite length so the map 
$T^p_{\x{F}}(A)\otimes_{A} M/\mathfrak{m}^lM\to T^p_{\x F}(M/\mathfrak{m}^lM)$ is surjective.
By taking inverse limits we get a surjective map 
$$
\varprojlim_l (T^p_{\x{F}}(A)\otimes_{A} M/\mathfrak{m}^lM)\to \varprojlim_l T^p_{\x F}(M/\mathfrak{m}^lM)
$$
The left hand side is just the completion of the module $T^p_{\x{F}}(A)\otimes_{A} M$ 
while the right hand side is the completion of $T^p_{\x{F}}(M)$. Since completion 
is a faithfully exact functor, the map  $T^p_{\x{F}}(A)\otimes_{A} M\to T^p_{\x F}(M)$ 
must be surjective. 

Finally, as usual one proves the same surjectivity for any $A$-module by reducing it 
to the case of finitely generated modules. Now Theorem \ref{t-flat-right-exact} 
implies the right exactness of $T^p_{\x{F}}$.
\end{proof}


\section{Base change and semi-continuity}

Let $f\colon X\to Y$ be a projective morphism of Noetherian schemes, 
and let $\x{F}$ be an $\x{O}_X$-module. We would like to know when the 
base change map on cohomology 
$$
R^pf_*\x{F}\otimes_Y k(y)\to H^p(X_y,\x{F}_y)
$$ 
for the base change $\Spec k(y)\to Y$ is an isomorphism.
If the map is an isomorphism for every $y\in Y$, we then say that $R^pf_*\x{F}$ 
commutes with base change. 

We apply the functors $T^p_{\x{F}}$ to solve the base change problem for coherent 
flat sheaves. The following theorem relates the two theories.\\

\begin{thm}\label{t-T-base-change}
Let $f\colon X\to Y$ be a projective morphism of Noetherian schemes with $Y=\Spec A$, 
and let $\x{F}$ be a coherent sheaf on $X$. Then, we have 
$T^p_{\x{F}}(k(y))\simeq H^p(X_y,\x{F}_y)$ for any $y\in Y$ where $\x{F}_y$ 
is the pullback of $\x{F}$ to $X_y$.
\end{thm}
\begin{proof}
Pick $y\in Y$. The support of the sheaf $\widetilde{k(y)}$ 
is inside the subscheme $S$ which is the closure of $y$ with the 
induced integral subscheme structure, and $\x{F}\otimes_Yk(y)$ is an 
$\x{O}_{X_S}$-module. Thus,
$$
 H^p(X,\x{F}\otimes_Y k(y))=H^p(X_S,(\x{F}\otimes_Yk(y))_S)
$$
so 
by replacing $Y$ with $S$ we could assume that $Y$ is integral and that
$y$ is its generic point. In that case, $\x{O}_y=k(y)$ so 
by the flat base change theorem 
(\ref{t-flat-base-change}), the base change map on cohomology 
$$
 H^p(X,\x{F})\otimes_A k(y) \to H^p(X_y,\x{F}_y)
$$
for the base change $\Spec k(y)\to Y$ is an isomorphism. On the other hand, 
since  $\x{O}_y=k(y)$, the functor $-\otimes_Ak(y)$ is exact and by looking at 
the \v{C}ech complex of $\x{F}$ (eg, see the proof of Theorem \ref{t-flat-complex}) 
one can easily see that 
$$
T^p_{\x{F}}(k(y))=H^p(X,\x{F}\otimes_Y k(y))\simeq H^p(X,\x{F})\otimes_A k(y)
$$
\end{proof}

\begin{thm}[Semi-continuity]\label{t-flat-semi-continuity}
Let $f\colon X\to Y$ be a projective morphism of Noetherian schemes, 
and let $\x{F}$ be a coherent sheaf on $X$ which is flat over $Y$. 
Then, $\dim_{k(y)}H^p(X_y,\x{F}_y)$ viewed as a function in $y\in Y$ is upper semi-continuous. 

Moreover, for a fixed $z\in Y$, if the function $\dim_{k(y)}H^p(X_y,\x{F}_y)$ 
is not constant near $z$, then either the function $\dim_{k(y)}H^{p-1}(X_y,\x{F}_y)$ 
or the function $\dim_{k(y)}H^{p+1}(X_y,\x{F}_y)$ is not constant near $z$.
\end{thm}
\begin{proof}
The problem is local on $Y$ so we may assume that $Y=\Spec A$ and that $L^\bullet$ 
is the Grothendieck complex.
By Theorem \ref{t-T-base-change}, 
$$
\dim_{k(y)}H^p(X_y,\x{F}_y)=\dim_{k(y)}T^p_{\x{F}}(k(y))=\dim_{k(y)}h^p(L^\bullet\otimes_A k(y))
$$ 
On the other hand, we have an exact sequence 
$$
(*) \hspace{0.5cm} 0\to T^p_{\x{F}}(A)\to W^p(L^\bullet)\to L^{p+1}\to W^{p+1}(L^\bullet)\to 0
$$ 
which induces an exact sequence 
$$
(**) \hspace{0.5cm} 0\to T^p_{\x{F}}(k(y))\to  W^p(L^\bullet)\otimes_Ak(y)\to L^{p+1}\otimes_Ak(y)\to W^{p+1}(L^\bullet)\otimes_A k(y)\to 0
$$
which reduces the problem to proving that the two functions 
$\dim_{k(y)}W^p(L^\bullet)\otimes_Ak(y)$ and $\dim_{k(y)}W^{p+1}(L^\bullet)\otimes_A k(y)$ 
are upper semi-continuous because $\dim_{k(y)}L^{p+1}\otimes_Ak(y)$ is independent of $y$. 
Now use Exercise \ref{exe-flat-semi-continuous}.

The second statement: if the function $\dim_{k(y)}H^p(X_y,\x{F}_y)$ 
is not constant near $z$, then either the function 
$\dim_{k(y)}W^p(L^\bullet)\otimes_Ak(y)$ or the function 
$\dim_{k(y)}W^{p+1}(L^\bullet)\otimes_A k(y)$ is not constant near $z$.

 The result then follows from the corresponding exact sequence $(**)$ for 
$T^p_{\x{F}}(k(y))$ or $T^{p+1}_{\x{F}}(k(y))$.\\
\end{proof}

\begin{thm}[Base change]\label{t-base-change}
Let $f\colon X\to Y$ be a projective morphism of Noetherian schemes, 
and let $\x{F}$ be a coherent sheaf on $X$ which is flat over $Y$, 
and fix $y\in Y$.
Assume that the base change map on cohomology 
$R^pf_*\x{F}\otimes_Y k(y)\to H^p(X_y,\x{F}_y)$ is surjective. Then 
it is an isomorphism and the same holds for any point in a neighborhood of $y$. 
Moreover, the following are equivalent: 

$(i)$ $R^{p-1}f_*\x{F}\otimes_Y k(y)\to H^{p-1}(X_y,\x{F}_y)$ is also surjective,

$(ii)$ $R^pf_*\x{F}$ is locally free near $y$.
\end{thm}
\begin{proof}
We may assume that $Y=\Spec A$ and that $L^\bullet$ is the Grothendieck complex.
Let $Y'=\Spec A'$ where $A':=\x{O}_y$, and let $T^p_{\x{F}'}$ be the functor associated to 
the sheaf $\x{F}':=\x{F}_{Y'}$ on $X':=X_{Y'}=X\times_Y Y'$ and the morphism $f'\colon X'\to Y'$. 
Note that the projection $\mu \colon X'\to X$ is an affine morphism so 
$H^p(X',\x{G})\simeq H^p(X,\mu_*\x{G})$ for any quasi-coherent sheaf $\x{G}$ on $X'$.
In particular, if $\x{G}=\x{F}'\otimes_{Y'}M$ for an $A'$-module $M$, then 
$\mu_*\x{G}\simeq \x{F}\otimes_YM$. Therefore, $T^p_{\x{F}'}$ 
is simply the restriction of $T^p_{\x{F}}$ to the category of $A'$-modules hence $L'^\bullet:=L^\bullet \otimes_AA'$ is the Grothendieck complex of $\x{F}'$ and $f'$.
If $y'$ is the closed point of $Y'$, then 
$R^pf_*'\x{F}'\otimes_{Y'} k(y')\to H^p(X_{y'}',\x{F}_{y'}')$ is surjective (note that the residue fields of $y$ and $y'$ are the same, and the fibres over $y$ and $y'$ are also the same).

By Theorem \ref{t-base-change-right-exact}, 
$T^p_{\x{F}'}$ is right exact. This implies that the functor $T^{p+1}_{\x F'}$ is left exact hence 
the module $W^{p+1}(L'^\bullet)=W^{p+1}(L^\bullet)\otimes_AA'$ is free over $A'$
hence the sheaf associated to $W^{p+1}(L^\bullet)$ is free near $y$. So, 
after replacing $Y$ with some open affine neighborhood of $y$ we can assume that 
the module $W^{p+1}(L^\bullet)$ is flat over $A$ and that 
the functor $T^{p+1}_{\x F}$ is left exact which implies that $T^{p}_{\x F}$ is right exact. 
In particular, the map $R^pf_*\x{F}\otimes_Y k(z)\to H^p(X_z,\x{F}_z)$ is an isomorphism 
for any $z\in Y$ by Theorem \ref{t-flat-right-exact}.

For the equivalence of $(i)$ and $(ii)$: if  
$R^{p-1}f_*\x{F}\otimes_Y k(y)\to H^{p-1}(X_y,\x{F}_y)$ is also surjective, 
then the above arguments show that we may assume that $T^{p-1}_{\x F}$  is also right exact hence $T^{p}_{\x F}$ is left exact and so exact. By Corollary \ref{t-flat-T-exact}, 
$T^{p}_{\x F}(A)$ is a finitely generated flat $A$-module which implies that 
$R^pf_*\x{F}$ is locally free (near $y$; remember that we have replaced $Y$ with smaller
neighborhoods).

Conversely, if $R^pf_*\x{F}$ is locally free near $y$ then
after shrinking $Y$ we may assume that $T^{p}_{\x F}(A)$ is a finitely generated flat $A$-module 
hence $T^{p}_{\x F}$ is exact so
$T^{p-1}_{\x F}$ is right exact, in particular, 
$R^{p-1}f_*\x{F}\otimes_Y k(y)\to H^{p-1}(X_y,\x{F}_y)$ is surjective.\\
\end{proof}

\begin{cor}\label{t-flat-base-change-locally-free}
Let $f\colon X\to Y$ be a projective morphism of Noetherian schemes, 
and let $\x{F}$ be a coherent sheaf on $X$ which is flat over $Y$. 
Then, for any $m$, $R^pf_*\x{F}$ commutes with base change for every $p\ge m$ if and only 
if $R^pf_*\x{F}$ is locally free for every $p>m$.
\end{cor}
\begin{proof}
Fix $m$. We use decreasing induction on $m$ and assume that the theorem is true for $m+1$. 
If $R^pf_*\x{F}$ commutes with base change for every $p\ge m$, then by induction 
$R^pf_*\x{F}$ is locally free for every $p>m+1$. Moreover, by the base change theorem, 
the fact that $R^mf_*\x{F}$ commutes with base change is equivalent to the 
local freeness of $R^{m+1}f_*\x{F}$. 

Conversely, if $R^pf_*\x{F}$ is locally free for every $p>m$, then 
by induction, $R^pf_*\x{F}$ commutes with base change for every $p>m$. 
Another application of the base change theorem shows that the local freeness 
of $R^{m+1}f_*\x{F}$ implies that $R^mf_*\x{F}$ commutes with base change.\\ 
\end{proof}

\begin{cor}\label{c-flat-locally-free}
Let $f\colon X\to Y$ be a projective morphism of Noetherian schemes, 
and let $\x{F}$ be a coherent sheaf on $X$ which is flat over $Y=\Spec A$. 
Consider the following properties: 

$(i)$ $T^p_{\x F}$ is exact,

$(ii)$  $R^pf_*\x{F}$ is locally free and it commutes with base change, 

$(iii)$ $\dim_{k(y)}H^p(X_y,\x{F}_y)$ viewed as a function in $y\in Y$ is locally constant,\\

Then, $(i) \Leftrightarrow (ii) \implies (iii)$, and if $Y$ is reduced, then 
$(i) \Leftrightarrow (ii) \Leftrightarrow (iii)$. 
\end{cor}
\begin{proof}
(i) $\Leftrightarrow$ (ii): In view of Corollary \ref{t-flat-T-exact} and 
Theorem \ref{t-base-change} (and its proof), $T^p_{\x F}$ is exact if and only if $T^p_{\x F}$ is 
right exact and $T^p_{\x{F}}(A)$ is a finitely generated flat $A$-module if and only if 
$R^pf_*\x{F}$ commutes with base change and it is locally free.

(i) $\implies$ (iii): The exactness of $T^p_{\x F}$ implies that the modules 
$W^p(L^\bullet)$ and $W^{p+1}(L^\bullet)$ are flat over $A$ hence the functions 
$\dim_{k(y)}W^p(L^\bullet)\otimes_Ak(y)$ and $\dim_{k(y)}W^{p+1}(L^\bullet)\otimes_A k(y)$ 
are locally constant. Now the exact sequence  $(**)$ of the proof of Theorem \ref{t-flat-semi-continuity} shows that the function $\dim_{k(y)}H^p(X_y,\x{F}_y)$ is locally constant.

Now in addition assume that $Y$ is integral.
We show that (iii) $\implies$ (i): The exact sequence $(**)$ of the proof of Theorem \ref{t-flat-semi-continuity} shows that if the function $\dim_{k(y)}H^p(X_y,\x{F}_y)$ is locally constant, then the two functions 
$\dim_{k(y)}W^p(L^\bullet)\otimes_Ak(y)$ and $\dim_{k(y)}W^{p+1}(L^\bullet)\otimes_A k(y)$ 
are locally constant. The latter is equivalent to the flatness of $W^p(L^\bullet)$ and $W^{p+1}(L^\bullet)$ by Exercise \ref{exe-flat-locally-constant-flat-module}. 
By Theorem \ref{t-flat-left-exact}, this is in turn equivalent to the left exactness of  
$T^p_{\x F}$ and $T^{p+1}_{\x F}$ which is the same as $T^p_{\x F}$ being exact.
\end{proof}

\section{Invariance of Euler characteristic and Hilbert polynomial}

\begin{thm}\label{t-flat-Euler}
Let $f\colon X\to Y$ be a projective morphism of Noetherian schemes, 
and let $\x{F}$ be a coherent sheaf on $X$ which is flat over $Y$. Then, the  
Euler characteristic $\mathcal{X}(X_y,\x{F}_y)$ viewed as a function in $y$ is 
locally constant on $Y$.
\end{thm}
\begin{proof}
The problem is local on $Y$ so we may assume that $Y=\Spec A$ and that 
$L^\bullet$ is the Grothendieck complex.
By definition, 
$$
\mathcal{X}(X_y,\x{F}_y)=\sum_{p}(-1)^p\dim_{k(y)}H^p(X_y,\x{F}_y)
$$
By Theorems \ref{t-flat-complex} and \ref{t-T-base-change}, 
$$
\dim_{k(y)}H^p(X_y,\x{F}_y)=\dim_{k(y)}T^p_{\x{F}}(k(y))=\dim_{k(y)}h^p(L^\bullet\otimes_Ak(y))
$$

Since all the modules $L^p$ in $L^\bullet$ are flat and finitely generated, for any $z\in Y$ 
the values $\dim_{k(y)}L^\bullet\otimes_Ak(y)$ are the same in a neighborhood 
of $z$. So, we may assume that these values are independent of $y$. 
The result follows immediately from some simple linear algebra, that is, the fact 
that
$$
\mathcal{X}(L^\bullet\otimes_Ak(y)):=\sum_{p}(-1)^p\dim_{k(y)}L^\bullet\otimes_Ak(y)=\sum_{p}(-1)^p\dim_{k(y)}h^p(L^\bullet\otimes_Ak(y))
$$
\end{proof}

\begin{thm}\label{t-flat-twist}
Let $f\colon X\to Y$ be a projective morphism of Noetherian schemes, 
and let $\x{F}$ be a coherent sheaf on $X$. Then, $\x{F}$ is flat over $Y$ if and only if $f_*\x{F}(l)$ is locally free for every $l\gg 0$.
\end{thm}
\begin{proof}
We may assume that $Y=\Spec A$ and by embedding $X$ into a projective space over $Y$ we may assume that $X=\PP^n_Y=\Proj A[t_0,\dots,t_n]$.
Assume that $\x{F}$ is flat.  For $l\gg 0$, $R^pf_*\x{F}(l)=0$ if $p>0$. 
By Corollary \ref{t-flat-base-change-locally-free},  for any $l\gg 0$, each $R^pf_*\x{F}(l)$ commutes with base change which in particular 
implies that $f_*\x{F}(l)$ is locally free by the base change theorem (\ref{t-base-change})
(actually this local freeness simply follows from the construction of the Grothendieck 
complex of $\x{F}(l)$; $T^0_{\x{F}(l)}(A)$ is a direct summand of the flat $A$-module 
$L^0$).

Conversely, assume that $f_*\x{F}(l)$ is locally free for every $l\gg 0$. 
 Then, $\x{F}=\widetilde{M}$ 
where  $M:=\oplus_{l\gg 0} H^0(X,\x{F}(l))$ 
and by assumptions each $ H^0(X,\x{F}(l))$ is a flat $A$-module for $l\gg 0$ (here 
the $\widetilde{M}$ is the sheaf associated to the graded module $M$ on $\Proj A[t_0,\dots,t_n]$). 
In particular, $M$ is a flat $A$-module and so $M_{t_i}$ 
is also flat over $A$. On the other hand, the localised module 
$M_{(t_i)}$ determines the sheaf $\x{F}|_{D_+(t_i)}$ and it is 
enough to prove that each $M_{(t_i)}$ is a flat $A$-module. 
The module $M_{t_i}$ has a natural graded structure in which 
the piece of degree $d$ is generated by the elements $m/t_i^r$ where  
$m$ is homogeneous of degree $r+d$. In particular, the piece of degree zero is just $M_{(t_i)}$. 
Thus, the latter module is a direct summand of $M_{t_i}$ hence it 
is flat over $A$.\\
\end{proof}

For a coherent sheaf $\x{F}$ on a scheme $X$ projective over a field $k$, 
one defines the Hilbert polynomial $\Phi$ of $\x{F}$ on $X$ to be 
the unique polynomial in $\Q[t]$ satisfying $\Phi(l)=\mathcal{X}(X,\x{F}(l))$ 
for any $l\in \Z$. Note that $\Phi$ also depends on the choice of 
a very ample sheaf $\x{O}_X(1)$.\\

\begin{thm}\label{t-flat-Hilbert-polynomial}
Let $f\colon X\to Y$ be a projective morphism of Noetherian schemes, 
and let $\x{F}$ be a coherent sheaf on $X$. If $\x{F}$ is flat over $Y$, then the 
Hilbert polynomial $\Phi_y$ of $\x{F}_y$ on $X_y$ viewed as a function in $y$ 
is locally constant on $Y$. Moreover, if $Y$ is integral the converse also holds.
\end{thm}
\begin{proof}
We may assume that $Y$ is connected. For each $l$ the function 
$\mathcal{X}(X_y,\x{F}(l)_y)$ is locally constant on $Y$ and since $Y$ 
is connected it is just constant. Therefore the Hilbert polynomial $\Phi_y$ 
is independent of $y$.

Now assume that $Y$ is integral and that the 
Hilbert polynomial $\Phi_y$ as a function in $y$ is locally constant on $Y$.
Since $Y$ is connected the polynomial is independent of $y$.
We can assume that $Y=\Spec A$, and 
by embedding $X$ into a projective space over $Y$ we may assume that 
$X=\PP^n_Y=\Proj A[t_0,\dots,t_n]$. For $l\gg 0$, $R^pf_*\x{F}(l)=0$ if $p>0$. 
By Theorem \ref{t-flat-twist}, it is enough to show that 
$f_*\x{F}(l)$ is locally free for $l\gg 0$. 

Pick $y\in Y$. Applying the flat base change theorem (\ref{t-flat-base-change}), we can replace 
$A$ by $\x{O}_y$ so we can assume that $A$ is a local ring and that $y$ is the 
closed point of $Y$.

We have an exact sequence $A^s\to A\to k(y)\to 0$ which for $l\gg 0$ gives rise to a commutative 
diagram 
$$
\xymatrix{
H^p(X,\x{F}(l))\otimes_A A^s \ar[r] \ar[d]&
H^p(X,\x{F}(l))\otimes_A A\ar[r]\ar[d] &
H^p(X,\x{F}(l))\otimes_Ak(y) \ar[d] \ar[r]&
0\\
H^p(X,\x{F}(l)\otimes_YA^s) \ar[r]&
H^p(X,\x{F}(l)\otimes_YA) \ar[r] &
H^p(X,\x{F}(l)\otimes_Yk(y))\ar[r]&
0
}
$$
in which the two vertical maps on the left are isomorphisms hence the right one is also an 
isomorphism. Therefore, by Theorem \ref{t-T-base-change},  
$$
\Phi_y(l)=\dim_{k(y)} H^0(X_y,\x{F}(l)_y)=\dim_{k(y)}H^0(X,\x{F}(l))\otimes_Ak(y)
$$ 
for $l\gg 0$. 
On the other hand, if $\eta$ is the generic point of $Y$, then the flat 
base change theorem implies that 
$$
\Phi_\eta(l)=\dim_{k(\eta)} H^0(X_{\eta},\x{F}(l)_{\eta})=\dim_{k(\eta)}H^0(X,\x{F}(l))\otimes_Ak(\eta)
$$ 
for $l\gg 0$. Since $\Phi_y=\Phi_\eta$, 
$\dim_{k(y)}H^0(X,\x{F}(l))\otimes_Ak(y)=\dim_{k(\eta)}H^0(X,\x{F}(l))\otimes_Ak(\eta)$.
Now use Exercise \ref{exe-module-local-free} to deduce that $H^0(X,\x{F}(l))$ 
is a free $A$-module for $l\gg 0$.\\
\end{proof}

\section{Generic flatness, flat morphisms, flat families of schemes}

\begin{thm}[Generic flatness]\label{t-flat-generic-flatness}
Let $f\colon X\to Y$ be a morphism of finite type between Noetherian schemes with $Y$
reduced, and $\x{F}$ a coherent sheaf on $X$. Then, $\x{F}$ is generically flat 
over $Y$ in the sense that there is a non-empty open subset $U\subseteq Y$ 
such that $\x{F}|_{f^{-1}U}$ is flat over $U$.
\end{thm}
\begin{proof}
We may assume that $Y=\Spec A$, $X=\Spec B$, and also that $Y$ is integral. 
Moreover, since $B$ is a finitely generated $A$-algebra, $f$ factors 
as the composition of a closed immersion $X\to \A^n_Y$ and the natural map $\A^n_Y\to Y$ 
for some $n$. By replacing $X$ with $\A^n_Y$ we may assume that $X=\A^n_Y$.
In addition, as $\A^n_Y$ can be embedded as an open subset of $\PP^n_Y$ 
we can replace $X$ and assume that $X=\PP^n_Y=\Proj A[t_0,\dots,t_n]$ (we can extend 
$\x{F}$ to a coherent sheaf on $\PP^n_Y$ for example as in Hartshorne [\ref{Hartshorne}, II, Exercise 5.15]). 
If $n=0$, $X=Y$ and we are done since $\x{F}$ would be a coherent sheaf 
on $Y$ and it is generically flat, i.e. locally free. So, assume that $n>0$.

If $Y$ has finitely many points, then we may replace $Y$ with its generic 
point in which case $A$ would be a field and flatness is automatic in this case.
So, we may assume that $Y$ is infinite which in particular means that $A$ 
has infinitely many elements.

Let  $h\in A[t_0,\dots,t_n]$ be a homogeneous polynomial
of degree one and let $H$ be the closed subscheme of $\PP^n_Y$ defined by $h$.
We then have the natural exact sequence 
$0\to \x{I}_H\to \x{O}_X\to \x{O}_H\to 0$. If $h$ is sufficiently general (this makes 
sense since $A$ is infinite), 
the sequence
$$
0\to \x{F}\otimes_{\x{O}_X}\x{I}_H\to \x{F}\otimes_{\x{O}_X}\x{O}_X\to \x{F}\otimes_{\x{O}_X}\x{O}_H\to 0
$$  
is exact (see Exercise \ref{exe-flatness-tensor-exact}). 
Moreover, perhaps after shrinking $Y$ we can assume that the coefficients 
of $h$ are invertible elements of $A$ and so after a linear change of variables 
we may assume that $h=t_1$. We then get the exact sequence 
$$
0\to \x{F}(l)\to \x{F}(l+1)\to \x{F}(l+1)\otimes\x{O}_H\to 0
$$ 
If $l\gg 0$, we get the induced exact sequence 
$$
0\to f_*\x{F}(l)\to f_*\x{F}(l+1)\to f_*(\x{F}(l+1)\otimes\x{O}_H)\to 0
$$ 
Since $H\simeq \PP^{n-1}_Y$, by applying induction on $n$ we may assume that 
$\x{F}\otimes\x{O}_H$ is flat over $Y$ hence by Theorem \ref{t-flat-twist}, we can assume that 
$f_*(\x{F}(l+1)\otimes\x{O}_H)$ is locally free for every $l\gg 0$. 
If we shrink $Y$ we can assume that $f_*\x{F}(l_0)$ is locally free 
for some sufficiently large $l_0$. Now the above exact sequence implies that 
$f_*\x{F}(l_0+1)$ is also locally free. Inductively, we deduce that 
$f_*\x{F}(l)$ is locally for any $l\gg 0$. Now another application of Theorem \ref{t-flat-twist}
implies that $\x{F}$ is flat over $Y$.\\  
\end{proof}

\begin{rem}
There are direct and more elementary proofs of the previous theorem 
(cf. [\ref{F}, Theorem 5.12]). The theorem is not true without the 
reduced assumption. For example, if $X=\Spec k$ and $Y=\Spec k[t]/\langle t^2 \rangle$ 
for some field $k$, then the induced morphism $f\colon X\to Y$ is not generically 
flat. In fact, for $f$ to be generically flat the sheaf $f_*\x{O}_X$ should 
be locally free which is not the case. 
\end{rem}

A morphism $f\colon X\to Y$ is called open if the image of any open subset $U\subseteq X$ 
is open in $Y$. It is said to be universally open if for any base change 
$S\to Y$, the induced morphism $X_S\to S$ is open.

\begin{thm}\label{t-flat-open}
Let $f\colon X\to Y$ be a flat morphism of finite type between Noetherian schemes. 
Then, $f$ is universally open.
\end{thm}
\begin{proof}
Since flatness is preserved under base change it is enough to prove that 
$f$ is open. We may assume that $Y$ is affine, say $\Spec A$.
Assume that $y=f(x)$. Then, the map $\alpha\colon \x{O}_y\to \x{O}_x$ is flat. 
Moreover, since $\alpha(\mathfrak{m}_y)\subseteq \mathfrak{m}_x$ we have 
$\mathfrak{m}_y\x{O}_x\neq \x{O}_x$ hence $\alpha$ is faithfully 
flat and the corresponding morphism $\Spec \x{O}_x\to \Spec \x{O}_y$ is surjective 
by Exercise \ref{exe-faithfully-flat}. In particular, this means that 
if we consider $y$ as a prime ideal of $A$, then for any prime ideal $y'$ of $A$, 
$y'\subset y$ implies that $y'\in f(X)$.

If $f$ is surjective, then there is nothing to prove. 
So, assume that $z\in Y$ is not in $f(X)$. If $z'\in \overline{z}$ (the closure of the 
point $z$), then $z'\notin f(X)$ otherwise $z\in f(X)$ by the above arguments. 

On the other hand, it is well-known that $f(X)$ is a constructible subset of $Y$, that is, 
it is the disjoint union of finitely many locally closed subsets of $Y$ 
(cf. [\ref{Hartshorne}, II, Exercise 3.19]). Moreover, the set 
$X\setminus f(X)$ is also constructible hence the union of finitely many locally 
closed subsets $W_i\cap Z_i$ where $W_i$ is an open subset and $Z_i$ is a 
closed subset of $Y$. If $z$ is the generic point of a component of 
$W_i\cap Z_i$, then $\overline{z}\subseteq X\setminus f(X)$. Therefore, 
$X\setminus f(X)$ is a closed subset hence $f(X)$ is open.\\
\end{proof}

\begin{exa}
Theorem \ref{t-flat-open} does not hold without the finite type condition. 
For example let $Y$ be an integral scheme and let $\eta$ be the generic point.
Then, the morphism $f\colon X=\Spec k(\eta)\to Y$ is flat but not of finite type 
and $f(X)$ is not open in most cases.\\
\end{exa}

\begin{thm}\label{t-flat-associated-point}
Let $f\colon X\to Y$ be a morphism of schemes.
If $f$ is flat, then any associated point of $X$ is mapped to an associated point of $Y$. Moreover, 
if $Y$ is regular of dimension one, then the converse also holds.
\end{thm}
\begin{proof}
Assume that $f$ is flat and that $x\in X$ is an associated point which means that 
the maximal ideal $\mathfrak{m}_x$ of $\x{O}_x$ is an associated prime, that is, 
its elements are all zero-divisors. Since $f$ is flat, the natural map 
$\x{O}_y\to \x{O}_x$ is flat where $y=f(x)$. If  
 $a\in \mathfrak{m}_y\subset \x{O}_y$ is not a zero-divisor, then 
the multiplication homomorphism $\x{O}_y\xrightarrow{\cdot a} \x{O}_y$ is injective 
and since $\x{O}_x$ is flat over $\x{O}_y$, the induced map 
$\x{O}_y\otimes_{\x{O}_y} \x{O}_x\xrightarrow{\cdot a} \x{O}_y\otimes_{\x{O}_y}\x{O}_x$ 
is also injective which is equivalent to saying that the image of $a$ in 
$\x{O}_x$ is not a zero divisor which is not possible. Therefore, 
$\mathfrak{m}_y$ is an associated prime hence $y$ is an associated point.  

Now assume that $Y$ is regular of dimension one. So, for any point $y\in Y$, the local 
ring $\x{O}_y$ is a discrete valuation ring, in particular, it is a principal ideal domain. 
If $f(x)=y$, then the natural map $\x{O}_y\to \x{O}_x$  being flat is equivalent to 
saying that $\x{O}_x$ has no torsion elements over $\x{O}_y$. Assume that 
$a\in \x{O}_y$ and $b\in \x{O}_x$ such that $ab=0$. If $b\neq 0$, then $a$ is a zero-divisor hence 
it should belong to an associated prime $P$ of $\x{O}_x$. The ideal $P$ corresponds 
to an associated point of $X$ which by assumptions is mapped to an associated 
point of $Y$. The only associated point of $Y$ is its generic point. Thus, 
the inverse image of $P$ under $\x{O}_y\to \x{O}_x$ is the zero ideal. In particular, 
this means that $a=0$.\\     
\end{proof}

\begin{rem}\label{r-flat-reduced-rings}
Let $A$ be a Noetherian ring with trivial nil-radical, i.e. $\Spec A$ reduced. 
Assume that $P$ is an associated prime of $A$. 
So, $Pa=0$ for some $0\neq a\in A$. But then if $P$ is not a minimal prime ideal,
$a$ must belong to every minimal prime ideal of $A$ hence $a\in {\rm{nil}}(A)=0$, a contradiction. 
Therefore, the only 
associated primes of $A$ are the minimal prime ideals.\\
\end{rem}

\begin{cor}
Let $f\colon X\to Y$ be a morphism of schemes where $X$ is reduced and
$Y$ is regular of dimension one. Then, $f$ is flat if and only if 
each irreducible component of $X$ dominates $Y$.
\end{cor}
\begin{proof}
This follows from Theorem \ref{t-flat-associated-point} because $X$ being reduced, its associated points are the generic points of its irreducible components.\\
\end{proof}

A morphism $f\colon X\to Y$ of schemes can be viewed as a family of schemes, that is 
the fibres of $f$, parameterised by the points of $Y$. In general, the fibres do not 
necessarily share strong common properties. However, when $f$ is flat, one can think of 
$f$ as a family of schemes in which the fibre changes "nicely" from point to point.   

According to Theorem \ref{t-flat-Hilbert-polynomial}, the Hilbert polynomial 
(of the structure sheaf) of the fibres of a flat projective morphism of Noetherian schemes is locally 
constant. In particular, the dimension, the arithmetic genus, the degree, and 
any other invariant of the fibres which is determined by the Hilbert polynomial would be locally 
constant.

Dimension of the fibres of a flat family behaves well even in the local case, i.e. 
when $f$ is not projective. If $X$ is a 
scheme and $x\in X$, we define $\dim_xX=\dim \x{O}_x$.\\

\begin{thm}
Let $f\colon X\to Y$ be a flat morphism of Noetherian schemes of finite dimension.
Then, for any $x\in X$ we have $\dim_xX_y=\dim_xX-\dim_yY$ where $y=f(x)$.
\end{thm}
\begin{proof}
The problem is local so may assume that $X=\Spec B$ and $Y=\Spec A$ 
and that $f$ corresponds to a homomorphism $\alpha\colon A\to B$. 
By replacing $A$ with $\x{O}_y$ we can assume that $A$ is local and 
$y$ is the closed point, and by replacing $B$ with $\x{O}_x$ we can also 
assume that $B$ is local too and $x$ is the closed point. Note that the local 
ring of $x$ on the fibre $X_y$ is then nothing but $B/\mathfrak{m}_yB$.

Now if $\dim A=0$, then the maximal ideal $\mathfrak{m}_y$ of $A$ is 
nilpotent hence $\mathfrak{m}_yB$ is also a nilpotent ideal. Therefore, 
$$
\dim_xX_y=\dim B/\mathfrak{m}_yB=\dim B=\dim_xX=\dim_xX-\dim_yY
$$ 
in this case. 
But if $\dim A>0$, we first replace $A$ by $A/{\rm{nil}}(A)$ which does not 
change the dimensions we are concerned with. Then, by Remark \ref{r-flat-reduced-rings}, 
there is some 
element $a\in A$ which is not a zero-divisor. 
The multiplication homomorphism $A\xrightarrow{\cdot a} A$ is injective 
and since $B$ is flat over $A$, the induced map $A\otimes_A B\xrightarrow{\cdot a} A\otimes_AB$ 
is also injective which is equivalent to saying that $\alpha(a)$ is not 
a zero divisor in $B$. Therefore, $\dim A/\langle a\rangle=\dim A-1$ and 
$\dim B/\langle \alpha(a)\rangle B=\dim B-1$ (cf. [\ref{Atiyah-Macdonald}, Corollary 11.18]) and by replacing $A$ with $A/\langle a\rangle$
and $B$ with $B/\langle \alpha(a)\rangle B$ we can use induction 
(the fibre does not change).\\
\end{proof}

\section{Stratification by Hilbert polynomials}

Let $f\colon X\to Y$ be a projective morphism of Noetherian schemes, and 
$\x{F}$ a coherent sheaf on $X$. We also choose an invertible sheaf $\x{O}_X(1)$ that is 
very ample over $Y$ and we use this to calculate Hilbert polynomials.
If $\x{F}$ is flat over $Y$, then 
we know from Theorem \ref{t-flat-Hilbert-polynomial} that the 
Hilbert polynomial of $\x{F}$ on the fibres, as a function on $Y$, is locally constant.
When $\x{F}$ is not flat over $Y$, we use the generic flatness  theorem (\ref{t-flat-generic-flatness}) 
to get a finite stratification of $Y$ so that on the strata we have constant 
Hilbert polynomials. More precisely, the set 
$$
\mathcal{H}=\{\Phi_y \mid \Phi_y ~\mbox{is the Hilbert poynomial of $\x{F}_y$ on 
$X_y$}\}=\{\Phi_1,\dots,\Phi_r\}
$$ 
is finite and $Y_\Phi=\{y\in Y \mid \Phi_y=\Phi\}\subseteq Y$ is a disjoint 
union of finitely many locally closed subschemes 
of $Y$ for each $\Phi\in \mathcal{H}$; this is proved below. 
Moreover, $Y_\Phi$ is disjoint from $Y_\Psi$ if $\Phi\neq \Psi$. 
 By reordering, we may assume that 
$\Phi_i(l)<\Phi_j(l)$ for any $i<j$ and $l\gg 0$. 

We are interested in the $Y_\Phi$ not only as subsets but also 
with a certain scheme structure which is specified in the next theorem.\\ 

\begin{thm}\label{t-flat-Hilb-stratification}
Under the above assumptions, we have the following:

$(i)$  for each $i$, the union $\bigcup_{i\le j} Y_{\Phi_j}$ is a closed subset of $Y$,

$(ii)$ each  $Y_\Phi$ carries a unique (locally closed) subscheme structure 
of $Y$ such that if $Y'$ is the scheme formed by the disjoint union of all the $Y_\Phi$, then 

$(iii)$ $\x{F}_{Y'}$ is flat over $Y'$, and

$(iv)$ if $S\to Y$ is any morphism from a Noetherian scheme such that $\x{F}_S$ 
is flat over $S$, then $S\to Y$ factors through the natural morphism $Y'\to Y$.  
\end{thm}
\begin{proof}
 
It is enough to prove the theorem locally on $Y$ 
and so we may assume that $Y$ is affine, say $Y=\Spec A$.
The uniqueness of $Y'$ allows one to glue local data to 
prove the theorem when $Y$ is not affine.

\emph{Step 1.} Let $V$ be any reduced locally closed subscheme of $Y$ 
which comes with a natural morphism $V\to Y$. Then 
by Theorem \ref{t-flat-generic-flatness} and 
Theorem \ref{t-flat-Hilbert-polynomial}, $\x{F}_{V}$ is flat over a  
non-empty maximal open subset $U$ of $V$ such that $\Phi_y$ does not depend 
on $y\in U$.  Replace $V$ by $V\setminus U$ and continue the 
process. By construction, Exercise \ref{exe-noflat-base-change}, 
Theorem \ref{t-base-change}, and the vanishing $R^pf_*\x{F}(l)=0$ when $p>0$ 
and $l\gg 0$, show that there is some $l_0$ such that 
$H^p(X_y,\x{F}(l)_y)=0$ when $p>0$, and 
$$
\dim_{k(y)}f_*\x{F}(l)\otimes k(y)=\dim_{k(y)}H^0(X_y,\x{F}(l)_y)=\Phi_y(l)
$$
for any $y\in V$ and any $l\ge l_0$.

In particular, the process decomposes $V$ into a disjoint union of locally 
closed subschemes of $Y$ (set-theoretically).
Taking $V=Y_{\rm red}$ proves that each $Y_\Phi$ is a disjoint 
union of finitely many locally closed subschemes.\\

\emph{Step 2.} We prove (i). If $i=1$, the statement is trivial. So we 
assume that $i>1$. If $y\in Y\setminus \bigcup_{i\le j} Y_{\Phi_j}$, 
then $y\in Y_{\Phi_k}$ for some $k<i$. Then, by Step 1, for $l\ge l_0$ we have   
$$
\dim_{k(y)}f_*\x{F}(l)\otimes k(y)=\dim_{k(y)}H^0(X_y,\x{F}(l)_y)=\Phi_k(l)
$$
On the other hand, by Exercise \ref{exe-flat-semi-continuous}, 
there is a neighborhood $U$ of $y$ such that 
$$
\dim_{k(u)}f_*\x{F}(l)\otimes k(u)\le \dim_{k(y)}f_*\x{F}(l)\otimes k(y)=\Phi_k(l)
$$ 
for any $u\in U$. Since $\Phi_k(l)<\Phi_i(l)<\cdots<\Phi_r(l)$, no point of $\bigcup_{i\le j} Y_{\Phi_j}$ can be in $U$ because if $z\in Y_{\Phi_j}$ with $i\le j$ then 
$$
\dim_{k(z)}f_*\x{F}(l)\otimes k(z)=\dim_{k(\eta)}H^0(X_z,\x{F}(l)_z)=\Phi_j(l)
$$ 
 Therefore,  
$U\cap (\bigcup_{i\le j}Y_{\Phi_j})=\emptyset$. Thus, $y$ cannot be in the closure of $\bigcup_{i\le j} Y_{\Phi_j}$. 

The above implies that $Y_{\Phi_i}$ is an open subset of $\bigcup_{i\le j} Y_{\Phi_j}$, and that  $Y_{\Phi_r}$ is a closed subset of $Y$.\\

\emph{Step 3.}  
We do induction on $r$. We may assume that the theorem holds 
for $W:=Y\setminus Y_{\Phi_r}$ in place of $Y$. In particular, this determines the 
scheme structures on $Y_{\Phi_i}$ if $i<r$ and also provides $W'$ as in the theorem.
We will construct $Y'$ as a disjoint of $W'$ and $Y_{\Phi_r}$ with a scheme structure 
to be defined. In Step 2, it was shown that $Y_{\Phi_r}$ is a closed subset of $Y$. 
Let $I$ be the radical ideal such that $V(I)=Y_{\Phi_r}$

Let $l\ge l_0$, and let $M=f_*\x{F}(l)(Y)$. Pick $y\in Y_{\Phi_r}$ and 
let $p=\dim_{k(y)}M\otimes k(y)=\Phi_r(l)$. Then, there is some 
$b\in A$ with $y\in D(b)=Y\setminus V(b)$ such that we have an exact sequence   
$$
0\to K\to A^p_b\to M_b\to 0
$$ 
where the subscript $b$ stands for localisation at $b$. Each element of $K$ can be  
written as $(a_1,\dots,a_p)$ for certain $a_i\in A_b$. Let $J$ be the ideal in $A_b$ 
generated by all the coordinates $a_i$ of the elements of $K$. 

Pick a ring homomorphism $A_b\to C$ with kernel $J'$. Then,  
$ A^p_b\otimes_{A_b}C\to M_b\otimes_{A_b}C$ is an isomorphism if and only if $J\subseteq J'$:  
the diagram 
$$
\xymatrix{
 && 0\ar[d]  & & \\
 &&  A^p_b\otimes J'\ar[d]^\beta & & \\
0 \ar[r] & K\otimes A_b \ar[d]\ar[r]^\alpha & A^p_b\otimes A_b \ar[d]\ar[r] & M_b\otimes A_b\ar[d]\ar[r] & 0\\ 
 & K\otimes C\ar[r] & A^p_b\otimes C\ar[r]^\gamma & M_b\otimes C\ar[r] & 0\\
}
$$
shows that if $\gamma$ is an isomorphism, then $\im(\alpha)\subseteq \im(\beta)$ which implies 
that $J\subseteq J'$. Conversely, if $J\subseteq J'$, then $A_b\to C$ factors 
through $A_b\to A_b/J$ hence $\gamma$ is an isomorphism as it is an isomorphism 
when $C=A_b/J$.

We can rephrase the property of $J$ in the last paragraph as: if $g\colon S\to D(b)$ is any morphism from a Noetherian scheme such that $g^*f_*\x{F}(l)$ is locally free of rank $p$, then $g$ factors through $\Spec A_b/J\to Y$. 
So, the ideals $J$ for various $y$ glue together to 
give an ideal ${I}_l$ such that if $Z_l$ is the closed subscheme corresponding to ${I}_l$, then the pullback of $f_*\x{F}(l)$ on $Z_l$ is locally free of rank $p$ and 
if $g\colon S\to Y$ is any other morphism from a Noetherian scheme such that $g^*f_*\x{F}(l)$ 
is locally free of rank $p$, then $g$ factors through $Z_l\to Y$ (note that near 
the points of $Y_{\Phi_r}$ we determined $I_l$ precisely; away from $Y_{\Phi_r}$ the ideal sheaf 
$\widetilde{I_l}$ is just the structure sheaf).

In fact, the underlying set of $V(I_l)$ is just $Y_{\Phi_r}$: $V(I_l)\subseteq Y_{\Phi_r}$ because 
by construction the pullback of $f_*\x{F}(l)$ on $Z_l$ is locally free of rank $\Phi_r(l)$. 
On the other hand, $Y_{\Phi_r}\subseteq V(I_l)$ because the sheaf $f_*\x{F}(l)$ pulled back on 
$\Spec A/I$ becomes locally free of rank $p$ by Exercise \ref{exe-flat-locally-constant-flat-module}.\\

\emph{Step 4.} For each $l\ge l_0$ we constructed an ideal ${I}_l$, in 
Step 3. Let ${I}'=\sum_l {I}_l$. The Noetherian property implies that only finitely many 
of the ${I}_l$ generate ${I}'$. Let $Z'$ be the closed subscheme defined by  
${I}'$ and let $h_1\colon Z' \to Y$ be the induced morphism. Note that $h_1$ factors through $Z_l\to Y$ for any $l\ge l_0$. Therefore, for such $l$, $h_1^*f_*\x{F}(l)$ is a 
locally free sheaf of rank $\Phi_r(l)$ which implies that $\x{F}_{Z'}$ is flat over $Z'$.

On the other hand, let $W$ and $W'$ be as in Step 3 and $h_2\colon W'\to Y$ be the given morphism. Define $Y'$ to be the scheme which is the disjoint of $W'$ and $Z'$. Set-theoretically, $Y'$ is the disjoint union of all the $Y_{\Phi_i}$. 
By assumptions, $\x{F}_{W'}$ is flat over $W'$ hence $\x{F}_{Y'}$ is flat 
over $Y'$. This proves (iii).

Now let $g\colon S\to Y$ be a morphism from a Noetherian scheme such that $\x{F}_S$ 
is flat over $S$. Then, for each $l\gg 0$, by Theorem \ref{t-flat-twist} and Exercise \ref{exe-noflat-base-change}, $g^*f_*\x{F}(l)$ is locally free. So, by the flat  
base change theorem (\ref{t-flat-base-change}) and Step 1,  
we can write $S$ as the disjoint union of $S_1$ and $S_2$ 
such that for any $l\gg 0$, $S_1$ is the union of those connected 
components of $S$ on which 
$g^*f_*\x{F}(l)$ is locally free of rank $\Phi_r(l)$. In particular, $S_1$ is mapped into 
$Y_{\Phi_r}$ and $S_2$ is mapped into $W$. By assumptions and constructions, $S_1\to Y$ factors through $Z'$ and $S_2\to Y$ factors through $W'\to Y$ hence 
$g$ factors through $Y'\to Y$. This proves (iv).\\   
\end{proof}

\section{Further results and related problems}

In studying the base change problem over 
a regular base scheme it is natural to frequently reduce arguments 
to the case of regular schemes of dimension one. The following theorem 
describes the obstruction to base change in that situation.\\ 

\begin{thm}\label{t-base-change-curves}
Let $f\colon X\to Y$ be a projective morphism of Noetherian schemes 
where $Y=\Spec A$ is regular of dimension one at every point, 
and let $\x{F}$ be a coherent sheaf on $X$ which is flat over $Y$. 
Then, for any $A$-module $M$ there is an exact sequence 
$$
0\to T^p_{\x{F}}(A)\otimes_A M\xrightarrow{\beta} T^p_{\x{F}}(M)\to \Tor^A_1(T^{p+1}_{\x{F}}(A),M)\to 0
$$
which is functorial in $M$.
\end{thm}
\begin{proof}
Let $L^\bullet$ be the Grothendieck complex. We let $B^p=\im L^{p-1}\to L^p$ 
and $Z^p=\ker L^p\to L^{p+1}$. We then have the following exact sequences
$$
0\to T^p_{\x{F}}(A)\to W^p(L^\bullet)\to B^{p+1}\to 0
$$
$$
0\to B^{p+1}\to Z^{p+1}\to T^{p+1}_{\x{F}}(A)\to 0
$$
$$
0\to Z^{p+1} \to L^{p+1} \to B^{p+2} \to 0
$$

On the other hand, by assumptions, $A_y$ is a discrete valuation ring for every $y\in Y$. This 
in particular means that a module over $A_y$ is flat if and only if it is 
torsion free. Thus, submodules of a flat $A_y$-module are also flat.  
Since $L^p_y$ is a flat $A_y$-module for each $p$, $B^p_y$ 
and $Z^p_y$ are also flat $A_y$-modules which in turn implies that 
$B^p$ and $Z^p$ are flat $A$-modules. Therefore, for any $A$-module $M$ 
the above exact sequences induce the following exact sequences
$$
0=\Tor^A_1(B^{p+1},M) \to T^p_{\x{F}}(A)\otimes_A M\to W^p(L^\bullet)\otimes_A M\xrightarrow{u} B^{p+1}\otimes_AM\to 0
$$
$$
0=\Tor^A_1(Z^{p+1},M) \to \Tor^A_1(T^{p+1}_{\x{F}}(A),M) \to B^{p+1}\otimes_A M\xrightarrow{v} Z^{p+1}\otimes_A M
$$
$$
0=\Tor^A_1(B^{p+2},M) \to Z^{p+1}\otimes_A M \xrightarrow{w} L^{p+1}\otimes_A M 
$$ 

As it was observed in the proof of Theorem \ref{t-flat-left-exact}, we have an 
exact sequence 
$$
0\to T^p_{\x{F}}(M)\to W^p(L^\bullet)\otimes_A M\xrightarrow{t} L^{p+1}\otimes_AM
$$
The map $t$ is the composite of the maps $u,v,w$, and since $w$ is injective, 
$\ker t=\ker (vu)$. Now the claim follows from the exact sequence 
$$
0\to \ker u \to \ker (vu)\to \ker v \to 0
$$
\end{proof}

\begin{rem}[Invariance of plurigenera]
(1) Let $f\colon X\to Y$ be a smooth projective morphism of smooth varieties over $\C$ 
with fibres of dimension $r$. 
Pick a natural number $m\ge 0$ and put $\x{F}=\omega_f^{\otimes m}$. 
Then, it is an important result of birational geometry, proved by 
Siu, that $\dim_{k(y)}H^0(X_y,\x{F}_y)$ is 
independent of $y$ and we call this the invariance of $m$-genus. Note that we can assume 
that $Y$ is a smooth affine curve. By Corollary \ref{c-flat-locally-free}, the problem is equivalent to saying that $f_*\x{F}$ 
is locally free and that it commutes with base change which in turn is equivalent to 
the functor $T^0_{\x{F}}$ being exact. Since $T^0_{\x{F}}$ is in any case left exact, the latter  statement  is equivalent to the left exactness of $T^1_{\x{F}}$.

(2) Under the assumptions of (1), suppose that $R^pf_*\x{F}=0$ for any $p>0$. Then, 
by Theorem \ref{t-base-change}, $f_*\x{F}$ 
is locally free and it commutes with base change, so the invariance of $m$-genus holds  
in this case. However, in general the higher direct images do not vanish. 
But at least when the generic fibre of $f$ is a variety of general type
one can use the minimal model program to achieve the vanishing. 

(3) Let $\x{G}:=\omega_f^{\otimes 1-m}=\x{F}^\vee\otimes \omega_f$ and assume that 
$R^{r-1}f_*\x{G}$ commutes with base change which is equivalent to local freeness of 
$R^rf_*\x{G}$ by Theorem \ref{t-base-change}. Then, by the relative duality 
theorem (\ref{t-duality}), for any finitely generated $A$-module $M$ we have 
$$
T^1_{\x{F}}(M)\simeq H^1(X,\x{F}\otimes_Y M)\simeq
\Ext^1_{\x{O}_X}(\x{O}_X,\x{F}\otimes_Y M)\simeq 
$$
$$
\Ext^1_{\x{O}_X}(\x{G},\omega_f\otimes_Y {M})
\simeq \Hom_{\x{O}_Y}(R^{r-1}f_*\x{G}, \widetilde{M})
$$
Since $R^{r-1}f_*\x{G}$ is fixed, this means that  $T^1_{\x{F}}$ is left exact 
hence the invariance of $m$-genus follows. 

(4) For any $y\in Y$, by Theorem \ref{t-base-change-curves}, there is an exact sequence 
$$
0\to T^0_{\x{F}}(A)\otimes_A k(y)\xrightarrow{\beta} T^0_{\x{F}}(k(y))\to \Tor^A_1(T^{1}_{\x{F}}(A),k(y))\to 0
$$
So the invariance of $m$-genus is equivalent to the vanishing 
$$
\Tor^A_1(T^{1}_{\x{F}}(A),k(y))=\Tor^A_1(H^1(X,\x{F}),k(y))=0
$$ 
for any $y\in Y$. It is enough to 
take closed points $y$ in which case $k(y)=\C$.
\end{rem}

\section*{Exercises}
\begin{enumerate}
\item Let $f_1\colon X_1\to Y_1$ and $f_2\colon X_2\to Y_2$ be two flat morphisms where 
$Y_1,Y_2$ are schemes over another scheme $Z$. Show that 
the induced morphism $ X_1\times_ZX_2\to Y_1\times_Z Y_2$ is a flat morphism.\\

\item Let $f\colon X\to Y$ be a flat morphism. Show that if $X$‌ is reduced, then $Y$‌ 
is reduced too. Show that if $X$‌ is normal, then $Y$‌ is normal too.\\

\item\label{exe-faithfully-flat} Let $f\colon X\to Y$ be a morphism of schemes, and $\x{F}$ a quasi-coherent $\x{O}_X$-module.
We say that $\x{F}$ is faithfully flat over $Y$ if for any base change $g\colon S\to Y$, a sequence $0\to \x{G}'\to \x{G} \to \x{G}''\to 0$ of quasi-coherent $\x{O}_S$-modules is exact  
if and only if the induced sequence $0\to \x{F}\otimes_Y \x{G}'\to \x{F}\otimes_Y\x{G} \to \x{F}\otimes_Y\x{G}''\to 0$ is exact. If $\x{F}=\x{O}_X$ is faithfully flat over $Y$, 
we simply say that $f$ is faithfully flat. Show that 

(i) $\x{F}$ is faithfully flat over $Y$ if and only if it is flat over $Y$ and for every $y\in Y$, $\x{F}_y\neq 0$ on $X_y$;

(ii) if $X=\Spec B$, $Y=\Spec A$, and $\x{F}=\widetilde{M}$, then $\x{F}$ is faithfully 
flat over $Y$ if and only if $M$ is a faithfully flat $A$-module if and only if $M$ 
is a flat $A$-module and for any maximal ideal $P$ of $A$, $PM\neq M$;

(iii) if $\x{F}$ is faithfully flat over $Y$, then $f(\Supp \x{F})=Y$. In particular, 
$f$ is surjective.

(iv) $f$ is faithfully flat if and only if $f$ is flat and surjective.\\

\item Under the assumptions of Definition \ref{d-flat-T} with $\x{F}$ flat over $Y$, 
generalise the result of 
Remark \ref{r-T0-Tr} as follows. Let $\x{G}$ be a coherent $\x{O}_X$-module.
Show that there is a finitely generated $A$-module $Q$ such that for any $A$-module $M$ 
we have 
$$
\Hom_{\x{O}_X}(\x{G},\x{F}\otimes_YM)\simeq \Hom_A(Q,M)
$$

\item\label{exe-flat-semi-continuous} Let $Y=\Spec A$ where $A$ is a Noetherian ring and $M$ a finitely generated $A$-module. Show that $\dim_{k(y)}M\otimes_A k(y)$ is an upper 
semi-continuous function in $y\in Y$.\\

\item\label{exe-flat-locally-constant-flat-module} Let $Y=\Spec A$ be a 
reduced Noetherian scheme and $M$ a finitely generated $A$-module. 
Show that if the function $\dim_{k(y)}M\otimes_A k(y)$ is constant, then 
$M$ is flat over $A$.\\

\item\label{exe-module-local-free} Let $A$ be an integral local Noetherian ring 
with residue field $k$ and fraction field $K$, 
and $M$ a finitely generated $A$-module. 
Show that if $\dim_{k}M\otimes_A k=\dim_{K} M\otimes_A K$, then 
$M$ is a free $A$-module.\\

\item\label{exe-noflat-base-change} Let $f\colon X\to Y$ be a projective morphism of Noetherian schemes,  
$\x{F}$ a coherent sheaf on $X$, and $g\colon S\to Y$ a morphism. 
Show that the base change morphism on cohomology $g^*R^pf_*\x{F}(l)\to R^p{f_S}_*\x{F}(l)_S$ is 
an isomorphism for any $l\gg 0$. (The point is that $\x{F}$ need not be flat over $Y$ 
and $g$ need not be a flat morphism.)\\

\item Let $A$ be a Noetherian ring and let $T\colon \mathfrak{M}(A)\to \mathfrak{M}(A)$ be 
a covariant additive left exact functor which commutes with direct sums. Show that there is a 
unique $A$-module $Q$ such that there is a functorial isomorphism 
$T(M)\simeq \Hom_A(Q,M)$ for any $A$-module $M$. (This gives the $Q$ in Theorem \ref{t-flat-left-exact} but it does not give the fact that $Q$ is finitely generated.)\\

\item Let $A$ be a Noetherian ring and let $T\colon \mathfrak{M}(A)\to \mathfrak{M}(A)$ be 
a covariant additive right exact functor which commutes with direct sums. Show that there is a unique 
$A$-module $N$ such that there is a functorial isomorphism $T(M)\simeq N\otimes_AM$ 
for any $A$-module $M$.\\

\item\label{exe-flatness-tensor-exact} Let $X=\Proj A[t_0,\dots,t_n]$,  
$h\in A[t_0,\dots,t_n]$ a homogeneous polynomial, and $H$ the closed subscheme of $X$ 
defined by $h$. Show that if $\x{F}$ is a coherent sheaf on $X$ such that 
$H$ does not contain the (finitely many) associated points of $\x{F}$, then  
the sequence
$$
0\to \x{F}\otimes_{\x{O}_X}\x{I}_H\to \x{F}\otimes_{\x{O}_X}\x{O}_X\to \x{F}\otimes_{\x{O}_X}\x{O}_H\to 0
$$  
is exact.\\  
\end{enumerate}


\chapter{\tt Hilbert and Quotient schemes}

Constructing moduli or parameterising spaces of objects in algebraic geometry 
is an important part of the classification process. One tries to put an algebraic 
structure on the parametrising object making it into a scheme, a sheaf, etc, so 
that the methods of algebraic geometry can be applied to study these spaces.
The importance of constructing moduli spaces is not only to parameterise 
objects but even more importantly the moduli space gives information about 
families of those objects.

Before we get into the theory of Hilbert and Quotient schemes, 
we discuss a simple example, that of parametrising projective curves over an algebraically 
closed field $k$. The set of all such curves is obviously huge. The idea is to 
divide it into smaller groups each containing curves which are similar in some way. 
One approach tries to use the invariant genus. Smooth curves with the same genus have many similar properties. For example, all smooth projective curves of genus $g$ over $\C$ are  
topologically the same, that is, they are homeomorphic to the compact Riemann surface with $g$ holes. It is well-known that for each $g\in \N$, 
there is a smooth quasi-projective variety $M_g$ over $k$ 
whose $k$-rational points (i.e. closed points in this case) correspond to 
smooth projective curves over $k$ of genus $g$, in a one-to-one way.
However, these moduli spaces are not perfect in the sense that they do not have 
a universal family so they are called coarse moduli spaces. 

Lets restrict our attention to projective curves in 
$\PP^2_k=\Proj k[t_0,t_1,t_2]$. Then, we can use the invariant degree to put curves 
into separate groups. In fact, one can easily parameterise 
all projective curves of a given degree in $\PP^2_k$ by simply considering 
the coefficients of the defining equation as parameters. For example, any 
projective curve in $\PP^2_k$ of degree two is defined by a non-zero
homogeneous polynomial 
$$
h=a_{0,0}t_0^2+a_{0,1}t_0t_1+a_{0,2}t_0t_2+a_{1,1}t_1^2+a_{1,2}t_1t_2+a_{2,2}t_2^2
$$
which can be considered as the point $(a_{0,0}:a_{0,1}:a_{0,2}:a_{1,1}:a_{1,2}:a_{2,2})$  
in $\PP^5_k$. One might primarily be interested only in those points of 
$\PP^5_k$ which correspond to smooth curves. However, the set of such points is 
 an open subset $U$ of $\PP^5_k$. A crucial feature of moduli theory is to 
try to compactify the moduli space so that one can apply the methods of study 
of projective (and more generally proper) schemes and morphisms. In this specific example 
the compactified space is just $\PP^5_k$. The points in $\PP^5_k\setminus U$ 
correspond to singular curves of degree two which are still acceptable since 
they are not too singular. This in particular suggests that one should try to 
parameterise more general schemes rather than just the smooth ones.

The equation $h$ defines a closed subscheme $\mathfrak{U}\subseteq \PP^2_k\times_k\PP^5_k$ 
such that a closed point $(x,\alpha)\in \mathfrak{U}$ if and only if $\alpha(x)=0$. 
In particular, the closed fibres of the projection $u\colon \mathfrak{U}\to \PP^5_k$ are 
simply the projective curves in $\PP^2_k$ of degree two. The scheme $\mathfrak{U}$ 
together with the projection $u$ is called the universal family of the moduli.
The universal family satisfies the following important property: let $S$ be a 
 Noetherian scheme over $k$ which parameterises a family of projective 
curves of degree two, that is, there is a closed subscheme $Z$ of $\PP^2_S$ 
such that the projection $Z\to S$ is flat and the fibre $Z_s$ over $s\in S$ is a 
projective curve of degree two in $\PP^2_{k(s)}$. Then, there is a unique 
morphism $S\to \PP^5_k$ over $k$ such that $Z\simeq S\times_{\PP^5_k}\mathfrak{U}$, 
that is, the family $Z\to S$ is the pullback of the universal family via the 
morphism $S\to \PP^5_k$.    

In the above example the invariant degree was enough to classify curves in 
$\PP^2_k$. However, if one wants to parameterise projective subschemes
of $\PP^n_k$, then the degree is not enough. Instead one tries to fix the 
Hilbert polynomial which in particular fixes the degree. 
We know from the last chapter that the Hilbert 
polynomial in a flat family of projective schemes is locally constant. 

The purpose of this chapter is to present Grothendieck's systematic 
treatment of constructing moduli schemes of families of schemes and sheaves. 
The original material are mostly in Grothendieck's FGA [\ref{FGA}] however 
our presentation follows [\ref{F}] which takes into account improvements by Mumford.\\


\section{Moduli spaces of subschemes of a scheme, and quotients of a sheaf}

We would like to construct the moduli space of the closed subschemes of a given 
scheme $X$ projective over a field $k$ using Hilbert polynomials.
For families, we are interested in flat families with a given Hilbert polynomial. 
Instead of just parameterising the closed subschemes of $X$, one better try to construct 
the moduli space of closed subschemes of the fibres of 
any $X_S\to S$ once and for all. For example if $S=\Spec L$ 
where $k\subseteq L$ is a field extension then we parameterise closed 
subschemes of $X$ as well as $X_S$. The advantage of this approach is that 
one frequently constructs moduli spaces for schemes over $\Spec \Z$ and 
get the corresponding moduli space for other base schemes for free.

Here we consider a more general form of Hilbert polynomials than those considered 
in the last chapter. Let $X$ be a projective scheme over a field $k$, $\x{L}$ an 
invertible sheaf on $X$ and $\x{F}$ a coherent sheaf on $X$. The Hilbert polynomial 
of $\x{F}$ on $X$ with respect to $\x{L}$ is the unique polynomial $\Phi\in \Q[t]$ 
which satisfies $\Phi(l)=\mathcal{X}(X,\x{F}\otimes \x{L}^{\otimes l})$ for every 
$l\in \Z$ (see Exercise \ref{exe-Hilb-polynomial}). If $\x{F}=\x{O}_X$, then we call $\Phi$ 
the Hilbert polynomial of $X$ with respect to $\x{L}$.\\

\begin{defn}[Hilb functor]
Let $f\colon X\to Y$ be a projective morphism of Noetherian schemes, $\x{L}$ 
an invertible sheaf on $X$, and $\Phi\in \Q[t]$ a polynomial. Let 
$\rm{N}\mathfrak{S}ch/Y$ denote the category of Noetherian schemes over $Y$.
We define the Hilbert functor 
$$
\cHilb_{X/Y}^{\Phi,\x{L}}\colon \rm{N}\mathfrak{S}ch/Y\to \mathfrak{S}et
$$
to the category of sets by 
$$
\cHilb_{X/Y}^{\Phi,\x{L}}(S)=\{\mbox{closed subschemes $Z$ of $X_S$ flat over $S$ 
such that for each $s\in S$, $\Phi_s=\Phi$}\}
$$  
where $\Phi_s$ is the Hilbert polynomial of the fibre $Z_s$ with respect to $\x{L}_s$, 
the pullback of $\x{L}$ to $Z_s$. One may think of $Z\to S$ as a flat family of 
schemes parameterised by the points of $S$ such that the fibres look like closed 
subschemes of the fibres of $f$ (the fact that we actually have a contravariant functor 
is proved easily, see Exercise \ref{exe-Hilb-funtoriality}).
\end{defn}

The Hilbert polynomial is locally constant on the base for a flat sheaf. 
This is proved in Theorem \ref{t-flat-Hilbert-polynomial} for $\x{L}$ a very ample sheaf but 
the proof works for any invertible sheaf.

To give a closed subscheme of a Noetherian scheme $X$ is the same as giving a coherent ideal 
sheaf which in turn is the same as giving a surjective morphism 
$\x{O}_X\to \x{G}$ of coherent sheaves. One may simply call $\x{G}$ a 
coherent quotient of $\x{O}_X$. More generally, if $\x{F}\to \x{G}$ is a surjective  
morphism of coherent sheaves, we call $\x{G}$ a coherent quotient of $\x{F}$.
Two coherent quotients $\x{G}$ and $\x{G}'$ of $\x{F}_S$ are 
considered the same or equivalent if there is an isomorphism $\psi\colon \x{G}\to \x{G}'$ 
such that $\phi'=\psi\phi$ where $\phi\colon \x{F}\to \x{G}$ and 
$\phi'\colon \x{F}\to \x{G}'$ are the given surjective morphisms.
We are led to the following generalisation of the Hilbert functor.\\ 

\begin{defn}[Quot functor]
Let $f\colon X\to Y$ be a projective morphism of Noetherian schemes, 
$\x{L}$ an invertible sheaf on $X$, $\x{F}$ a coherent sheaf on $X$, and 
$\Phi\in \Q[t]$ a polynomial. 
We define the quotient functor 
$$
\cQuot_{\x{F}/X/Y}^{\Phi,\x{L}}\colon \rm{N}\mathfrak{S}ch/Y\to \mathfrak{S}et
$$
by 
$$
\cQuot_{\x{F}/X/Y}^{\Phi,\x{L}}(S)=\{\mbox{coherent quotients $\x{G}$ of $\x{F}_S$ flat over $S$ 
with $\Phi_s=\Phi$}\}
$$  
where $\Phi_s$ is the Hilbert polynomial of the sheaf $\x{G}_s$ on the fibre $X_s$ of 
$X_S\to S$ over $s\in S$ with respect to $\x{L}_s$ the pullback of $\x{L}$.\\
\end{defn}

\begin{defn}[Representability]
We say that the functor $\cQuot_{\x{F}/X/Y}^{\Phi,\x{L}}$ is representable if we have 
the following:

$(i)$  a scheme $\Quot_{\x{F}/X/Y}^{\Phi,\x{L}}$ in $\rm{N}\mathfrak{S}ch/Y$ called the quotient scheme, 

$(ii)$ a coherent sheaf 
$\mathfrak{G}\in \cQuot_{\x{F}/X/Y}^{\Phi,\x{L}}(\Quot_{\x{F}/X/Y}^{\Phi,\x{L}})$ on $X\times_Y\Quot_{\x{F}/X/Y}^{\Phi,\x{L}}$ called the universal family, 

$(iii)$  the functorial map  
$$
\Hom_{\mathfrak{S}ch/Y}(S,\Quot_{\x{F}/X/Y}^{\Phi,\x{L}}) \to \cQuot_{\x{F}/X/Y}^{\Phi,\x{L}}(S) 
$$ 
which sends a morphism $S\to \Quot_{\x{F}/X/Y}^{\Phi,\x{L}}$ to the pullback of $\mathfrak{G}$ 
via the induced morphism $X_S\to X\times_Y\Quot_{\x{F}/X/Y}^{\Phi,\x{L}}$, is an 
isomorphism.\\
 
If $\cQuot_{\x{F}/X/Y}^{\Phi,\x{L}}$ is representable as above, we simply say that 
the functor is represented by $\Quot_{\x{F}/X/Y}^{\Phi,\x{L}}$. In case $\x{F}=\x{O}_X$, instead of $\Quot_{\x{O}_X/X/Y}^{\Phi,\x{L}}$ we use the notation $\Hilb_{X/Y}^{\Phi,\x{L}}$ and call it the Hilbert scheme 
of $X$ over $Y$.\\
\end{defn}

We come to the main theorem of this chapter:

\begin{thm}\label{t-Hilb-quot-functor}
Let $f\colon X\to Y$ be a projective morphism of Noetherian schemes, 
$\x{L}=\x{O}_X(1)$ a very ample invertible sheaf over $Y$, 
and $\Phi\in \Q[t]$ a polynomial. Let  $\x{F}$ be a coherent sheaf on $X$ 
which is the quotient of a sheaf $\x{E}=\x{O}_X(m)^r$ for some 
$r\in \N$ and $m\in \Z$. 
Then, the quotient functor $\cQuot_{\x{F}/X/Y}^{\Phi,\x{L}}$ is represented by a 
scheme $\Quot_{\x{F}/X/Y}^{\Phi,\x{L}}$ which is projective over $Y$.\\
\end{thm}

The theorem is proved in Section \ref{sec-Hilb-proofs}.
Note that the condition of $\x{F}$ being a quotient of a sheaf $\x{E}=\x{O}_X(m)^r$ 
is automatically satisfied for example if $Y$ is affine. Another occasion where 
this property is trivially satisfied is when $\x{F}=\x{O}_X$ hence we get the 
following result.

\begin{cor}
Let $f\colon X\to Y$ be a projective morphism of Noetherian schemes, 
$\x{L}=\x{O}_X(1)$ a very ample invertible sheaf over $Y$, 
and $\Phi\in \Q[t]$ a polynomial. 
Then, the Hilbert functor $\cHilb_{X/Y}^{\Phi,\x{L}}$ is represented by a 
scheme $\Hilb_{X/Y}^{\Phi,\x{L}}$ which is projective over $Y$.
\end{cor}
\begin{proof}
Apply Theorem \ref{t-Hilb-quot-functor}.\\
\end{proof}


\section{Examples}

In this section, we discuss several examples of Hilbert and Quotient schemes.

\begin{exa}
Let $f\colon X\to Y$ be a projective morphism of Noetherian schemes, $\x{L}=\x{O}_X(1)$, 
and let $\Phi=1$. The closed subschemes of the fibres of $f$ with 
Hilbert polynomial $\Phi$ are exactly the closed points. It is then easy to guess 
that  $\Hilb_{X/Y}^{\Phi,\x{L}}=X$ and the universal family is the diagonal  
$\Delta\subset X\times_YX$. In fact, if $S$ is any Noetherian scheme over $Y$ 
and $Z\subseteq X_S$ 
a closed subscheme flat over $S$ with fibres having Hilbert polynomial 
$\Phi$, then $Z\to S$ is an isomorphism since it is a finite morphism of degree one as 
the fibres are single points. So, we get a unique morphism $S\to X$ which 
induces a morphism $S\to \Delta$. Moreover, $Z=S$ is just the product 
$S\times_X\Delta$ because $\Delta \to X$ is an isomorphism.\\
\end{exa}

\begin{exa}
Let $X$ be a smooth projective curve of genus $g$ over $Y=\Spec k$ 
where $k$ is an algebraically closed field, 
and $\x{L}=\x{O}_X(1)$. 

(1) Let $\Phi(t)=(\deg \x{L})t+1-g$. Then the only closed subscheme of $X$ 
with Hilbert polynomial $\Phi$ is $X$ itself. In this case, $\Hilb_{X/Y}^{\Phi,\x{L}}=\Spec k$ 
and the universal family is just the given morphism $X\to Y$.  
If $S$ is a Noetherian scheme over $k$, and $Z\subseteq X_S$ 
a closed subscheme flat over $S$ with fibres having Hilbert polynomial 
$\Phi$, then $Z=X_S$, that is, there is only one family parameterised by $S$ 
which corresponds to the given morphism $S\to Y$.

(2) Let $\Phi(t)=m$ for some $m>0$. Then, the only closed subschemes of $X$ 
with Hilbert polynomial $\Phi$ are the zero-dimensional closed subschemes 
$Z$ with $\dim_kH^0(Z,\x{O}_Z)=m$. Such a $Z$ can be identified with 
an effective Cartier divisor $D$ of degree 
$m$. Each such divisor can be written as a sum $\sum d_iD_i$ where $D_i$ is a 
single point on $X$ and $\sum d_i=m$. Any such divisor can be considered as a 
point on the product of $X$ with itself $m$ times, that is, 
$X^m$. However, since the order of the $D_i$ does not make any difference 
we should take the quotient of $X^m$ by the equivalence relation $\sim$ 
which is defined on the closed points as follows: $(x_1,\dots,x_m)\sim (x_1',\dots,x_m')$
if $(x_1',\dots,x_m')$ can be obtained from $(x_1,\dots,x_m)$ by a permutation 
of the $x_i$. The quotient $X^m/\sim$ is called the $m$-th symmetric product 
of $X$ and it is denoted by $\Sym^mX$. The fact that the symmetric product is an 
algebraic variety follows from a general fact that quotients of quasi-projective varieties by 
finite groups are again quasi-projective varieties. We leave it to the reader 
to verify that $\Hilb_{X/Y}^{\Phi,\x{L}}=\Sym^m X$.\\ 
\end{exa}

\begin{exa}
Let $X=\PP^n_\Z=\Proj \Z[t_0,\dots,t_n]$, $Y=\Spec \Z$, $f\colon X\to Y$ the natural 
morphism, and $\x{L}=\x{O}_X(1)$. Pick a natural number $d$ and let 
$\Phi(t)=\binom{n+t}{n}-\binom{n-d+t}{n}$. Note that strictly speaking a function 
like $\binom{2+t}{2}$ takes in only values $t\ge 0$, however, we rather have the  
function $(t+2)(t+1)/2$ in mind which agrees with $\binom{2+t}{2}$ for $t\ge 0$.
We would like to compute $\Hilb_{X/Y}^{\Phi,\x{L}}$. We have chosen the polynomial 
$\Phi$ so that if $Z$ is a hypersurface of degree $d$ in the projective space $\PP^n_k$ over some 
field $k$, then $\Phi$ is the Hilbert polynomial of $Z$. Conversely, if $\Phi$ 
is the Hilbert polynomial of a closed subscheme $Z$ of $\PP^n_k$, then 
$Z$ is a hypersurface of degree $d$ (see Exercise \ref{exe-Hilb-hypersurface}).
As in the example in the 
introduction to this chapter such $Z$ can be parameterised by looking at its 
defining equation, a homogeneous polynomial of degree $d$ in $n+1$ variables. 
Such polynomials correspond to the points of the projective space 
$\PP^N_k$ where $N=\binom{n+d}{n}-1$. This suggests that 
$\Hilb_{X/Y}^{\Phi,\x{L}}\simeq \PP^N_\Z$ and we are going to verify it.   

Let $Z\subseteq X_S$ be a flat family in $\cHilb_{X/Y}^{\Phi,\x{L}}(S)$ for some 
Noetherian scheme $S$ over $Y$. Consider the natural exact sequence 
$0\to \x{I}_Z(d)\to \x{O}_{X_S}(d)\to \x{O}_Z(d) \to 0$ in which 
$\x{O}_{X_S}(d)$ and $\x{O}_Z(d)$ are flat over $S$ hence $\x{I}_Z(d)$ is also 
flat over $S$. By Theorem \ref{t-flat-exact-preserving},
for any point $s\in S$ the sequence 
$$
0\to \x{I}_Z(d)_s\to \x{O}_{X_S}(d)_s\to \x{O}_Z(d)_s \to 0
$$
on the fibre of $X_S\to S$ over $s$, say $X_s$, is exact. In particular, this means that 
$(\x{I}_Z)_s$ is the ideal sheaf of $Z_s$ in $X_s$ hence $(\x{I}_Z(d))_s=\x{O}_{X_s}$ 
because $\x{I}_{Z_s}=\x{O}_{X_s}(-d)$ as $Z_s$ is a hypersurface of degree $d$ in 
$X_s=\PP^n_{k(s)}$.

On the other hand, by looking at the cohomology on the fibres we have the vanishings 
 $H^p(X_S,\x{I}_Z(d)\otimes_Sk(s))=0$, $H^p(X_S,\x{O}_{X_S}(d)\otimes_Sk(s))=0$, 
and $H^p(X_S, \x{O}_Z(d)\otimes_Sk(s))=0$ for $p>0$. Thus, by the base change theorem 
(\ref{t-base-change}), $R^1{f_S}_*\x{I}_Z(d)=0$, and we get an exact sequence 
$$
0\to \x{M}_1:={f_S}_*\x{I}_Z(d)\to \x{M}_2:={f_S}_*\x{O}_{X_S}(d)\to \x{M}_3:={f_S}_*\x{O}_Z(d) \to 0
$$
of locally free sheaves. The rank of these sheaves can be determined by looking 
at the fibres, in particular, ${f_S}_*\x{I}_Z(d)$ is of rank one. Moreover, 
the morphism $f_S^*{f_S}_*\x{I}_Z(d)\to \x{I}_Z(d)$ is an isomorphism since the 
induced morphism $f_S^*{f_S}_*\x{I}_Z(d)_s\to \x{I}_Z(d)_s$ on $X_s$ is an isomorphism 
so one can apply Exercise \ref{exe-Hilb-fibre-isom}. 

Let $g\colon S\to Y$ be the given morphism and let $\x{E}:=\x{O}_Y^{N+1}$. So, $\PP(\x{E})=\PP^N_\Z=\Proj \Z[w_0,\dots,w_N]$.
Moreover, 
$$
\x{M}_2^\vee\simeq \x{M}_2\simeq \Sym^d\x{O}_S^{n+1}\simeq \x{O}_S^{N+1}\simeq g^*\x{E}
$$
Therefore, the exact sequence 
$$
0\to \x{M}_3^\vee \to \x{M}_2^\vee \to \x{M}_1^\vee \to 0
$$
uniquely induces a morphism $h\colon S\to \PP^N_\Z$ such that $h^*\x{O}_{\PP^N_\Z}(1)\simeq \x{M}_1^\vee$.

Let $\Psi_0,\dots,\Psi_N$ be all the monomials of 
degree $d$ in the variables $t_0,\dots,t_n$. The expression 
$\Theta=\sum_iw_i\Psi_i$ defines a closed subscheme $\mathfrak{U}$ of $\PP^n_\Z\times_\Z\PP^N_\Z$ 
such that the fibre of the projection $u\colon \mathfrak{U}\to \PP^N_\Z$ over a point 
$p$ is the hypersurface of degree $d$ in $\PP^n_{k(p)}$ defined by the polynomial 
$\sum_iw_i(p)\Psi_i$ where $w_i(p)$ is the image of $w_i$ in the residue field 
$k(p)$. In particular, the Hilbert polynomial of the fibres is constant hence 
$\mathfrak{U}\to \PP^N_\Z$ is flat by Theorem \ref{t-flat-Hilbert-polynomial}.
We will prove that $\mathfrak{U}$ is the universal family.

Let $e\colon \PP^n_\Z\times_\Z\PP^N_\Z \to \PP^N_\Z$ be the second projection 
and $e'$ the first projection. Then 
the above arguments show that 
$$
0\to \x{N}_1:=e_*\x{I}_{\mathfrak{U}}(d)\to \x{N}_2:=e_*\x{O}_{\PP^n_\Z\times_\Z\PP^N_\Z}(d)\to \x{N}_3:=e_*\x{O}_{\mathfrak{U}}(d) \to 0
$$
is an exact sequence of locally free sheaves with $\x{N}_1$ of rank one and 
$e^*e_*\x{I}_{\mathfrak{U}}(d) \to \x{I}_{\mathfrak{U}}(d)$ an isomorphism, 
thus $ \x{I}_{\mathfrak{U}}(d)$ is an invertible sheaf. 
Moreover,  the Picard group of 
$\PP^n_\Z\times_\Z\PP^N_\Z$ is generated, as a free abelian group, by $e^*\x{O}_{\PP^N_\Z}(1)$
and $e'^*\x{O}_{\PP^n_\Z}(1)$, and since $\x{I}_{\mathfrak{U}}$ has degree $-1$ 
over $\PP^n_\Z$ and has degree $-d$ over $\PP^N_\Z$, we deduce that 
$$
\x{I}_{\mathfrak{U}}\simeq e^*\x{O}_{\PP^N_\Z}(-1)\otimes e'^*\x{O}_{\PP^n_\Z}(-d)
$$
In view of the isomorphism  $e^*e_*\x{I}_{\mathfrak{U}}(d) \to \x{I}_{\mathfrak{U}}(d)$ 
we get $e_*\x{I}_{\mathfrak{U}}(d)\simeq \x{O}_{\PP^N_\Z}(-1)$.
Therefore, $\x{M}_1\simeq g^*\x{N}_1$, and if $c$ 
is the induced morphism from $X_S$ to $\PP^n_\Z\times_\Z\PP^N_\Z$, then 
$c^*\x{I}_\mathfrak{U}(d)\simeq \x{I}_Z(d)$. This implies that 
$c^*\x{O}_\mathfrak{U}\simeq \x{O}_Z$ and we are done since 
$X_S$ is just the product of $S$ and $\PP^n_\Z\times_\Z\PP^N_\Z$ 
over $\PP^N_\Z$ (uniqueness of $h$ is left as an exercise).\\ 
\end{exa}

\begin{exa}
Let $X$ be a Noetherian scheme, $Y=X$, $f\colon X\to Y$ the identity 
morphism, $\x{L}$ any invertible sheaf, $\x{F}$ a locally free sheaf of finite 
rank, and $\Phi=1$. We will show that the functor  $\cQuot_{\x{F}/X/Y}^{\Phi,\x{L}}$
is represented by  $\Quot_{\x{F}/X/Y}^{\Phi,\x{L}}=\PP(\x{F})$ 
with the natural morphism $\pi\colon \PP^(\x{F})\to Y$ and the universal 
family being the quotient $\pi^*\x{F}\to \mathfrak{G}:=\x{O}_{\PP(\x{F})}(1)$.
Let $g\colon S\to Y$ be a morphism from a Noetherian scheme. Then, $X_S\to S$ 
is the identity morphism and any coherent quotient $\x{G}$ of $g^*\x{F}$ 
which is flat over $S$ is locally free. Moreover, since $\Phi=1$, the 
quotient having Hilbert polynomial $\Phi$ means that $\x{G}$ has rank 
one, that is, it is an invertible sheaf. Therefore, any such quotient 
determines a morphism $h\colon S\to \PP(\x{F})$. In fact, the natural 
quotient $\pi^*\x{F}\to \mathfrak{G}$ is pulled back via $h$ to the 
quotient $g^*\x{F}\to \x{G}$. Uniqueness of $h$ follows from the 
properties of $\PP(\x{F})$. Note that the fibres of $f$ are the 
$\Spec k(y)$ for points $y\in Y$ and the pullback of $\x{F}$ to 
$\Spec k(y)$ is a vector space $\x{F}_y$. The quotients 
of $\x{F}_y$ of dimension one correspond to the linear subspaces of 
$\x{F}_y$ of codimension one which are parameterised by the 
points of $\PP(\x{F}_y)$, that is, the fibre of $\pi$ 
over $y$.\\
\end{exa}

\begin{exa}
Let $f\colon X\to Y$ and $f'\colon X'\to Y$ be two projective morphisms of 
Noetherian schemes with $f$ flat. First assume that $Y=\Spec k$ for a field 
$k$. We are interested in parameterising morphisms $X\to X'$ over $Y$. 
Note that any such morphism uniquely determines a closed immersion 
$X\to X\times_YX'$ hence a closed subscheme of $X\times_YX'$ which is 
called the graph of $X\to X'$. 
A family of "such morphisms" parameterised by a scheme $S$ over $Y$ can be 
defined simply as a morphism $X_S\to X_S'$ over $S$. Over each point $s$ of $S$, 
the morphism $X_S\to X_S'$ gives a morphism of the fibres over $s$ which 
just looks like a morphism $X\to X'$ over $Y$. 

Now we want to parameterise morphisms in the general case, that is, 
when $Y$ is any Noetherian scheme. Pick $\Phi\in \Q[t]$, and let $\x{L}$ be 
an invertible sheaf on $X\times_Y X'$ which is very ample over $Y$. 
A morphism $a\colon X_S\to X_S'$ over $S$ determines a closed immersion 
$\Gamma_a\colon X_S\to X_S\times_SX_S'\simeq (X\times_YX')_S$, called the graph of $a$,  
which is flat over $S$. We can consider $\Gamma_a$ as a closed subscheme of 
$(X\times_YX')_S$. We then define a functor 
$$
\xHom_{Y}^{\Phi}(X,X')\colon \rm{N}\mathfrak{S}ch/Y\to \mathfrak{S}et
$$
by 
$$
\xHom_{Y}^{\Phi}(X,X')(S)=\{\mbox{morphisms $a\colon X_S\to X_S'$ over $S$ such that 
$\Gamma_a\in\cHilb_{X\times_YX'/Y}^{\Phi,\x{L}}(S)$ }\}
$$
Each element $a$ of $\xHom_{Y}^{\Phi}(X,X')(S)$, by definition, determines 
$\Gamma_a\in\cHilb_{X\times_YX'/Y}^{\Phi,\x{L}}(S)$ which in turn induces unique morphisms 
$h\colon S\to \Hilb_{X\times_YX'/Y}^{\Phi,\x{L}}$ and $X_S\to \mathfrak{U}$ where $u\colon \mathfrak{U}\to 
\Hilb_{X\times_YX'/Y}^{\Phi,\x{L}}$ is the universal family. However, the points of 
$S$ are mapped to those points of $\Hilb_{X\times_YX'/Y}^{\Phi,\x{L}}$  over which the   
fibre of $u$ and the projection 
$$
Z:=X\times_Y\Hilb_{X\times_YX'/Y}^{\Phi,\x{L}} \to \Hilb_{X\times_YX'/Y}^{\Phi,\x{L}}
$$ 
are isomorphic because 
$$X_S\simeq \Gamma_a\simeq S\times_{\Hilb_{X\times_YX'/Y}^{\Phi,\x{L}}} \mathfrak{U}
~~~~\mbox{and}~~~~
X_S\simeq S\times_{\Hilb_{X\times_YX'/Y}^{\Phi,\x{L}}}Z
$$ 
The set of such points turns out to be an (possibly empty) open subscheme of $\Hilb_{X\times_YX'/Y}^{\Phi,\x{L}}$  denoted by $\Hom_{Y}^{\Phi}(X,X')$ which represents the functor 
$\xHom_{Y}^{\Phi}(X,X')$ (cf. [\ref{F}, 5.23]).\\
\end{exa}


\section{The Grassmannian}

In this section, we discuss the Grassmannian which serves both as a good classical example of a parameterising space and as a technical tool in the proof of Theorem \ref{t-Hilb-quot-functor}.

\begin{exa}\label{exa-Hilb-Grassmannian}
Let $k$ be an algebraically closed field, $V$ an $n$-dimensional $k$-vector space, and $d$ a non-negative integer. 
The Grassmannian $\Grass(V,d)$ is the space which parameterises the $(n-d)$-dimensional $k$-vector subspaces of $V$ (caution: there are several different notations for Grassmannian). If $d=n-1$, then $\Grass(V,d)$ parameterises the lines passing through the origin so it is nothing but the classical projective space of $V$ which is isomorphic to $\PP^{n-1}_k$. Similarly, if $d=1$, $\Grass(V,d)$ is 
the classical projective space of the dual of $V$ which is again isomorphic to $\PP^{n-1}_k$. 
Obviously, $\Grass(V,0)$ and $\Grass(V,n)$ are single points.

Let $W$ be a $d$-dimensional $k$-vector space. Any subvector space $V'$ of $V$ 
of dimension $n-d$ is the kernel of some surjective $k$-linear map $\phi\colon V\to W$ 
which is in turn determined by a $d\times n$ matrix $M_\phi$ over $k$ of rank $d$.
However, $\phi$ and $M_\phi$ are not uniquely determined by $V'$. They are uniquely 
determined up to a certain equivalence relation. 
One can think of $\phi$ as a quotient of $V$ of dimension $d$ and say that two quotients 
$\phi\colon V\to W$ and $\psi\colon V\to W$ are equivalent if there is a  $k$-isomorphism
$\alpha\colon W\to W$ satisfying $\psi=\alpha\phi$. Similarly, $M_\phi$ is uniquely 
determined up to the following equivalence:  $M_\phi$ 
is equivalent to any matrix of the form $NM_\phi$ where $N$ is an invertible $d\times d$ 
matrix. 

The set of all $d\times n$ matrices over $k$ is parameterised by the affine space 
$\A^{dn}_k$, and the subset corresponding to matrices of rank $d$ corresponds to 
an open subset $U\subset \A^{dn}_k$. So, $\Grass(V,d)$ is just the quotient of 
$U$ by the group ${\rm{GL}}(k,d)$. We can describe $\Grass(V,d)$ locally as follows. 
Pick a point in $\Grass(V,d)$ which is represented by a matrix $M$ with entries 
$m_{i,j}$ and columns $\gamma_1,\dots,\gamma_n$. 
If $I\subseteq \{1,\dots,n\}$ is a subset of size $d$, by $M_I$ we mean the 
$I$-th submatrix of $M$, that is, the $d\times d$ matrix with columns $\gamma_i$, $i\in I$.
Assuming that $d>0$, there is an $I$ such that $\det M_I\neq 0$.  
Moreover, perhaps after choosing another representative instead 
of $M$, we may assume that $M_I$ is the $d\times d$ identity matrix. The other entries 
of $M$ which are not in $M_I$ are then uniquely determined.
So, all the points $M$ 
of $\Grass(V,d)$ with $\det M_I\neq 0$ simply correspond to the points of 
the affine space $\A^{dn-d^2}_k$. Let $G_I$ be the set of such points.

If $I,J\subseteq \{1,\dots,n\}$ are subsets of size $d$, then we let  
$G_{I,J}$ to be the set of those points $M\in G_I$ with $\det M_J\neq 0$ 
(note that by definition we already have $\det M_I\neq 0$). We define a 
morphism $f_{I,J}\colon G_{I,J}\to G_{J,I}$ be sending $M$ to $M_J^{-1}M$.
If $I,J,K\subseteq \{1,\dots,n\}$ are subsets of size $d$, then
$f_{I,K}=f_{J,K}f_{I,J}$ because if $M\in G_{I,J}\cap G_{I,K}$, then   
$$
f_{J,K}f_{I,J}(M)=f_{J,K}(M_J^{-1}M)=(M_J^{-1}M)_K^{-1}M_J^{-1}M=
$$
$$
(M_J^{-1}M_K)^{-1}M_J^{-1}M=M_K^{-1}M_JM_J^{-1}M=M_K^{-1}M=f_{I,K}(M)
$$
This, in particular, means that $f_{J,I}f_{I,J}=f_{I,I}=\rm{id}$ hence 
the $f_{I,J}$ are isomorphisms. Therefore, we can glue the $G_I$ 
via the $f_{I,J}$ and put a scheme structure on $\Grass(V,d)$.

Another way of describing  $\Grass(V,d)$ is via the so-called Pl\"uker coordinates.
For each linear subspace $V'$ of dimension $n-d$, choose a basis $v_1,\dots,v_{n-d}$ 
to which we can associate the point  $v_1\wedge \dots \wedge v_{n-d}$ in 
the projective space $\PP(\wedge^{n-d}V)$. It turns out that the point only 
depends on $V'$ and not on the basis chosen. Moreover, the image of $\Grass(V,d)$ 
under this association is a closed subset of $\PP(\wedge^{n-d}V)$ which can be 
explicitly described by some quadratic equations. This in particular means that 
$\Grass(V,d)$ is projective. The scheme structure coincides with the one obtained 
above. 

Another important fact about the Grassmannian is that it carries a certain 
universal locally free sheaf with some strong properties as mentioned in the next theorem.\\
\end{exa}

\begin{thm}\label{t-Hilb-Grassmannian}
Let $Y$ be a Noetherian scheme and $\x{E}$ a locally free sheaf of rank $n$. 
Then there exist a unique (up to isomorphism) scheme $\Grass(\x{E},d)$ with a 
closed embedding into $\PP(\wedge^d \x{E})$ and the induced  
morphism $\pi\colon \Grass(\x{E},d)\to Y$, and a rank $d$ locally free 
quotient $\pi^*\x{E}\to \mathfrak{G}$ satisfying the following universal 
property: for any morphism $g\colon S\to Y$, and any  
 rank $d$ locally free quotient $g^*\x{E}\to \x{G}$, there is a unique morphism 
$h\colon S\to \Grass(\x{E},d)$ over $Y$ such that $g^*\x{E}\to \x{G}$ coincides 
with the pullback of $\pi^*\x{E}\to \mathfrak{G}$ via $h$. 
\end{thm}
\begin{proof}
See Grothendieck's EGA I (new edition) [\ref{EGA-new}, 9.7.4, 9.7.5, and 9.8.4].\\
\end{proof}

The Grassmannian $\Grass(\x{F},d)$ is constructed by mimicking Example  \ref{exa-Hilb-Grassmannian} 
in a way that the construction is functorial in $\x{F}$.  
In EGA, the theorem is stated essentially as in the next corollary. However, the 
above formulation is more natural in the sense that it generalises the corresponding 
result for the projective space of a locally free sheaf 
(cf. Hartshorne [\ref{Hartshorne}, II, 7.12]).\\

\begin{cor}
Let $X=Y$ be a Noetherian scheme, $f\colon X\to Y$ the identity morphism, $\x{L}=\x{O}_X$,
$\x{F}$ a locally free sheaf of rank $n$, and $\Phi=d$. Then, the quotient 
functor $\cQuot^{\Phi,\x{L}}_{\x{F}/X/Y}$ is represented by the Grassmannian
$\Grass(\x{F},d)$. 
\end{cor}
\begin{proof}
Let $g\colon S\to Y$ be a morphism from a Noetherian scheme. The sheaf $\x{F}_s$ on the fibre  $X_s=\Spec k(s)$ of $f_S=\rm{id}$ is a $k(s)$-vector space of dimension $n$. If $\x{G}$ is a 
quotient of $\x{F}_S=g^*\x{F}$ in  $\cQuot^{\Phi,\x{L}}_{\x{F}/X/Y}(S)$, then on $X_s$ we get the 
$d$-dimensional quotient $\x{G}_s$ of $\x{F}_s$. Thus, since $\x{G}$ is flat over $S$, it is locally 
free of rank $d$. By Theorem \ref{t-Hilb-Grassmannian}, the surjection $\x{F}_S\to \x{G}$ 
uniquely determines a morphism $h\colon S\to \Grass(\x{F},d)$ such that 
$g^*\x{F}\to \x{G}$ coincides with the pullback of $\pi^*\x{F}\to \mathfrak{G}$.\\
\end{proof}

\section{Proof of main results}\label{sec-Hilb-proofs}

In this section we give the proof of Theorem \ref{t-Hilb-quot-functor}.

\begin{proof}(of Theorem \ref{t-Hilb-quot-functor})
\emph{Step 1.}
We can replace $\x{F}$ by $\x{F}(-m)$. Indeed, let $\Psi$ be the polynomial defined 
by $\Psi(t)=\Phi(t-m)$. Then, there is a natural transformation of functors 
$$
\cQuot^{\Phi,\x{L}}_{\x{F}/X/Y}\to \cQuot^{\Psi,\x{L}}_{\x{F}(-m)/X/Y}
$$
which is defined by sending a quotient $\x{F}_S\to \x{G}$ in 
$\cQuot^{\Phi,\x{L}}_{\x{F}/X/Y}(S)$ to $\x{F}_S(-m)\to \x{G}(-m)$ 
for any Noetherian scheme $S$ over $Y$. The above natural transformation is 
an isomorphism of functors so we can indeed replace $\x{F}$ by $\x{F}(-m)$ 
and so assume that $\x{E}=\x{O}_X^r$.\\

\emph{Step 2.} 
Take a closed immersion $e\colon X\to \PP^n_Y$ such that $\x{O}_{\PP^n_Y}(1)$ pulls back 
to $\x{L}$. Then, there is a natural transformation 
$$
\cQuot^{\Phi,\x{L}}_{\x{F}/X/Y}\to \cQuot^{\Phi,\x{O}_{\PP^n_Y}(1)}_{e_*\x{F}/{\PP^n_Y}/Y}
$$
which is an isomorphism of functors: indeed for any morphism $S\to Y$ from a Noetherian 
scheme and any quotient $\x{F}_S\to \x{G}$ in  $\cQuot^{\Phi,\x{L}}_{\x{F}/X/Y}(S)$ 
the morphism ${e_S}_*\x{F}_S\to {e_S}_*\x{G}$ is a quotient in 
$\cQuot^{\Phi,\x{O}_{\PP^n_Y}(1)}_{e_*\x{F}/{\PP^n_Y}/Y}(S)$ where $e_S$ is the 
induced closed immersion $X_S\to \PP^n_S$. Note that ${e_S}_*\x{F}_S$ is nothing but the 
pullback of $e_*\x{F}$ via the morphism $\PP^n_S\to \PP^n_Y$. Conversely, if 
 ${e_S}_*\x{F}_S\to \x{G}'$ is a quotient in $\cQuot^{\Phi,\x{O}_{\PP^n_Y}(1)}_{e_*\x{F}/{\PP^n_Y}/Y}(S)$, then $\x{G}'$ is a module over $\x{O}_{X_S}$ so there is some $\x{G}$ on $X_S$ 
 with ${e_S}_*\x{G}=\x{G}'$. So, the quotient  ${e_S}_*\x{F}_S\to \x{G}'$ is the direct 
 image of the corresponding quotient  $\x{F}_S\to \x{G}$.

Moreover, the surjection $\x{E}=\x{O}_X^r\to \x{F}$ 
gives the surjection $e_*\x{E}\to e_*\x{F}$ which in turn induces a surjection 
$\x{O}_{\PP^n_Y}^r\to e_*\x{F}$. Therefore, we can from now on assume that 
$X=\PP^n_Y$, $\x{L}=\x{O}_X(1)$, and that $\x{E}=\x{O}_X^r$.\\

\emph{Step 3.} 
The surjective morphism $\x{E}=\x{O}_X^r\to \x{F}$ induces a 
natural transformation of functors 
$$
\cQuot^{\Phi,\x{L}}_{\x{F}/X/Y}\to \cQuot^{\Phi,\x{L}}_{\x{E}/X/Y}
$$
which is defined by sending a quotient $\x{F}_S\to \x{G}$ in 
$\cQuot^{\Phi,\x{L}}_{\x{F}/X/Y}(S)$ to the induced quotient 
$\x{E}_S\to \x{G}$, for any Noetherian scheme $S$ over $Y$.
By Lemma \ref{l-Hilb-quot-functor-embedding} below, if $\cQuot^{\Phi,\x{L}}_{\x{E}/X/Y}$ is represented by $\Quot^{\Phi,\x{L}}_{\x{E}/X/Y}$, 
then $\cQuot^{\Phi,\x{L}}_{\x{F}/X/Y}$ is represented by a closed subscheme 
$\Quot^{\Phi,\x{L}}_{\x{F}/X/Y}$ of $\Quot^{\Phi,\x{L}}_{\x{E}/X/Y}$. Thus, 
we may take $\x{F}=\x{E}=\x{O}_X^r$.\\

\emph{Step 4.}  In this step, among other things, we show that there is 
$l\gg 0$ such that: for any 
morphism $g\colon S\to Y$ from a Noetherian scheme and any  
quotient $\x{F}_S\to \x{G}$ in $\cQuot^{\Phi,\x{L}}_{\x{F}/X/Y}(S)$ with kernel $\x{K}$,  
the induced sequence 
$$
(1) \hspace{1cm} 0\to {f_S}_*\x{K}(l)\to {f_S}_*\x{F}_S(l)\to {f_S}_*\x{G}(l)\to 0
$$
is an exact sequence of locally free sheaves. For any point $s\in S$,  
the sequence 
$$
0\to \x{K}_s\to \x{F}_s\to \x{G}_s\to 0
$$
is exact on the fibre $X_s=\PP^n_{k(s)}$ over $s$ by Theorem \ref{t-flat-exact-preserving}. 
In particular, $\x{K}_s$ is a subsheaf of $\x{O}_{X_s}^r$. 
Now, since the Hilbert polynomial of $\x{K}_s$ is independent of $s$, $S$, and $\x{G}$, 
by Theorem \ref{t-regularity-2}, there is $l\gg 0$ not depending on $S$, 
$s$, and $\x{G}$ such that $\x{K}_s$ is $l$-regular. 

Now, by Theorem \ref{t-regularity-1}, 
we have $H^p(X_s,\x{K}(l)_s)=0$ if $p>0$. So,  
the vanishing  $H^p(X_s,\x{F}_S(l)_s)=0$ for $p>0$ implies that $H^p(X_s,\x{G}(l)_s)=0$ if $p>0$
and that $\x{K}(l)_s$, $\x{F}_S(l)_s$ and $\x{G}(l)_s$ are generated 
by global sections (perhaps after replacing $l$ with $l+n$). Now the base change theorem (\ref{t-base-change}) allows us to 
deduce that $R^p{f_S}_*\x{K}(l)=0$, $R^p{f_S}_*\x{F}_S(l)=0$, $R^p{f_S}_*\x{G}(l)=0$ if $p>0$, 
and that the sequence $(1)$ is indeed an exact sequence of locally free sheaves. 
These properties also imply that in the diagram 
$$ 
\xymatrix{
 0\ar[r] & f_S^*{f_S}_*\x{K}(l)\ar[r]\ar[d]^\alpha & f_S^*{f_S}_*\x{F}_S(l)\ar[r]\ar[d]^\beta & f_S^*{f_S}_*\x{G}(l)\ar[r]\ar[d]^\gamma & 0\\
 0\ar[r] & \x{K}(l)\ar[r] & \x{F}_S(l) \ar[r]& \x{G}(l)\ar[r] & 0
}
$$
the maps $\alpha$, $\beta$, and $\gamma$ are surjective: we verify it for $\alpha$ and 
similar arguments apply to $\beta$ and $\gamma$ (actually it is well-known that 
$\beta$ is surjective which also implies surjectivity of $\gamma$). Here we may assume 
$Y$ is affine. For any $s\in S$, by base change, the map 
$$
{f_S}_*\x{K}(l)\otimes k(s)=H^0(X_S,\x{K}(l))\otimes k(s)\to H^0(X_s,\x{K}(l)_s)
$$
is an isomorphism. Now the pullback of the sheaf ${f_S}_*\x{K}(l)\otimes k(s)$ on $X_s$ 
is 
$$
{f_S}_*\x{K}(l)\otimes k(s)\otimes \x{O}_{X_s}\simeq H^0(X_s,\x{K}(l)_s)\otimes \x{O}_{X_s}
$$
which surjects onto $\x{K}(l)_s$ since $K(l)_s$ is generated by global sections. 
Therefore, $\alpha$ restricted to $X_s$ is 
surjective for every $s$. This implies that the cokernel of $\alpha$ is 
zero hence $\alpha$ is surjective.\\

\emph{Step 5.} In Step 4, for each morphism $g\colon S\to Y$ from a Noetherian scheme 
and each quotient $\x{F}_S\to \x{G}$ in $\cQuot^{\Phi,\x{L}}_{\x{F}/X/Y}(S)$ we constructed 
a quotient ${f_S}_*\x{F}_S(l)\to {f_S}_*\x{G}(l)$ which actually is an element 
of $\cQuot^{\Phi(l),\x{O}_Y}_{f_*\x{F}(l)/Y/Y}(S)$ where $\Phi(l)$ is considered 
as a constant polynomial. We are using the isomorphism 
${f_S}_*\x{F}_S(l)\simeq g^*f_*\x{F}(l)$ which follows from the facts that 
$X=\PP^n_Y$, $\x{F}=\x{O}_X^r$, and that the  
natural map $g^*f_*\x{F}(l)\otimes k(s)\to {f_S}_*\x{F}_S(l)\otimes k(s)$ is an isomorphism 
for any $s\in S$. Thus, we get a natural transformation 
$$
\cQuot^{\Phi,\x{L}}_{\x{F}/X/Y} \to \cQuot^{\Phi(l),\x{O}_Y}_{f_*\x{F}(l)/Y/Y}
$$
which is injective in the sense that for any morphism $g\colon S\to Y$ from a Noetherian scheme 
the map 
$$
\cQuot^{\Phi,\x{L}}_{\x{F}/X/Y}(S) \to \cQuot^{\Phi(l),\x{O}_Y}_{f_*\x{F}(l)/Y/Y}(S)
$$
is injective: in fact if for two quotients $\x{F}_S\to \x{G}$ and $\x{F}_S\to \x{G}'$
in $\cQuot^{\Phi,\x{L}}_{\x{F}/X/Y}(S)$ the quotients ${f_S}_*\x{F}_S(l)\to {f_S}_*\x{G}(l)$
and ${f_S}_*\x{F}_S(l)\to {f_S}_*\x{G}'(l)$ are equivalent, then 
${f_S}_*\x{K}(l)={f_S}_*\x{K}'(l)$ where $\x{K}$ and $\x{K}'$ are the corresponding kernels.
Thus, $f_S^*{f_S}_*\x{K}(l)=f_S^*{f_S}_*\x{K}'(l)$ and the diagram 
in Step 4 shows that both $\x{K}(l)$ and $\x{K}'(l)$ are the image of $f_S^*{f_S}_*\x{K}(l)$ 
under the map $f_S^*{f_S}_*\x{F}(l)\to \x{F}(l)$. Therefore, $\x{K}(l)=\x{K}'(l)$ 
which implies that $\x{K}=\x{K}'$ and that the two quotients $\x{F}_S\to \x{G}$ and $\x{F}_S\to \x{G}'$ are equivalent hence the same object in $\cQuot^{\Phi,\x{L}}_{\x{F}/X/Y}(S)$.\\

\emph{Step 6.} 
The functor $\cQuot^{\Phi(l),\x{O}_Y}_{f_*\x{F}(l)/Y/Y}$ is represented by the Grassmannian scheme $G:=\Grass(f_*\x{F}(l),\Phi(l))$, by Theorem \ref{t-Hilb-Grassmannian}, which is 
a closed subscheme of $\PP(\wedge^{\Phi(l)}f_*\x{F}(l))$ hence projective over $Y$. 
For ease of notation let $\x{M}=f_*\x{F}(l)$, let $\mathfrak{G}$ be the universal family 
on $G$ which comes with a surjective morphism $\pi^*\x{M}\to \mathfrak{G}$ whose kernel we denote by $\x{L}$ where $\pi$ is the structure morphism $G\to Y$. Pulling back the exact sequence 
$$
0\to \x{L}\to \pi^*\x{M}\to \mathfrak{G}\to 0
$$ 
onto $X_G$ via $f_G$ gives the exact sequence 
$$
0\to f_G^*\x{L}\to f_G^*\pi^*\x{M}\to f_G^*\mathfrak{G}\to 0
$$
Since $\pi^*\x{M}={f_G}_*\x{F}_G(l)$, we have  $f_G^*\pi^*\x{M}=f_G^*{f_G}_*\x{F}_G(l)$. The natural 
morphism  $f_G^*{f_G}_*\x{F}_G(l)\to \x{F}_G(l)$ then induces a morphism $f_G^*\x{L}\to \x{F}_G(l)$ whose 
cokernel we denote by $\x{R}$. So, we have an exact sequence 
$f_G^*\x{L}\to \x{F}_G(l)\to \x{R}\to 0$.

Now let $g\colon S\to Y$ be a morphism from a Noetherian scheme, 
$\x{F}_S\to \x{G}$ an element in $\cQuot^{\Phi,\x{L}}_{\x{F}/X/Y}(S)$, and 
${f_S}_*\x{F}_S(l)\to {f_S}_*\x{G}(l)$ the corresponding element 
in $\cQuot^{\Phi(l),\x{O}_Y}_{f_*\x{F}(l)/Y/Y}(S)$. Then, there is a unique 
morphism $h\colon S\to G$ such that 
$$
0\to {f_S}_*\x{K}(l)\to {f_S}_*\x{F}_S(l)\to {f_S}_*\x{G}(l)\to 0
$$
coincides with
$$
0\to h^*\x{L}\to h^*\pi^*\x{M}\to h^*\mathfrak{G}\to 0
$$ 
We are 
using the facts that $g^*\x{M}={f_S}_*\x{F}_S(l)$ and that $\mathfrak{G}$ 
is locally free. Thus, the sequence 
$$
0\to f_G^*\x{L}\to f_G^*\pi^*\x{M}\to f_G^*\mathfrak{G}\to 0
$$
 pulls back via the 
induced morphism $c\colon X_S\to X_G$ to the sequence 
$$ 
0\to f_S^*{f_S}_*\x{K}(l)\to f_S^*{f_S}_*\x{F}_S(l)\to f_S^*{f_S}_*\x{G}(l)\to 0
$$ 
Therefore, the morphism $c$ pulls back $f_G^*\x{L}\to \x{F}_G(l)\to \x{R}\to 0$ 
to 
$$
f_S^*{f_S}_*\x{K}(l)\to \x{F}_S(l)\to \x{G}(l)\to 0
$$ 
in particular, $\x{R}(-l)$ 
is pulled back to $\x{G}$ which is flat over $S$. 

Now the Hilbert stratification of $G$ for the sheaf $\x{R}(-l)$ as in Theorem 
\ref{t-flat-Hilb-stratification} shows that the morphism $h\colon S\to G$ 
factors through the locally closed subscheme $G_{\Phi}$. On the other hand, 
any morphism $S\to G$ over $Y$ which factors through $G_{\Phi}$ 
will give a quotient of $\x{F}_S$ in $\cQuot^{\Phi,\x{L}}_{\x{F}/X/Y}(S)$ 
which is nothing but the pullback of the quotient $\x{F}_G\to \x{R}(-l)$.
Therefore, the scheme $G_{\Phi}$ together with the universal family 
$\x{F}_G\to \x{R}(-l)$ restricted 
to $X_{G_{\Phi}}$ represent the functor $\cQuot^{\Phi,\x{L}}_{\x{F}/X/Y}$.
 The quotient scheme  $\Quot^{\Phi,\x{L}}_{\x{F}/X/Y}$
is then nothing but $G_{\Phi}$.\\

\emph{Step 7.} It remains to prove that the scheme $G_{\Phi}$ is a closed subset 
of $G$. Since $G$ is projective over $Y$, it is enough to show that $G_{\Phi}$ is proper 
over $Y$. To do that we use the valuative criterion of properness as $G_{\Phi}$ is 
Noetherian and of finite type over $Y$. Let $R$ be a DVR, $K$ its fraction field 
and assume that we have a commutative diagram 
$$
\xymatrix{
T=\Spec K \ar[r]\ar[d]^j & G_{\Phi}\ar[d]\ar[r] & G \ar[ld] \\
S=\Spec R \ar[r] & Y
}
$$
which induces a commutative diagram 
$$
\xymatrix{
X_T\ar[r]\ar[d]^{j'} & X_{G_{\Phi}}\ar[d]\ar[r] & X_G \ar[ld] \\
X_S\ar[r] & X
}
$$
where $j$ and $j'$ are open immersions.
Now the pullback of $\x{F}_G\to \x{R}(-l)$ gives a quotient $\x{F}_T\to \x{G}$ 
in $\cQuot^{\Phi,\x{L}}_{\x{F}/X/Y}(T)$ from which we get a morphism $j'_*\x{F}_T\to j'_*\x{G}$. On the other hand, 
since $\x{F}_T=j'^*\x{F}_S$, there is a natural morphism $\x{F}_S\to j'_*\x{F}_T$ 
which combined with $j'_*\x{F}_T\to j'_*\x{G}$ determines a morphism 
$\phi\colon \x{F}_S\to j'_*\x{G}$ hence a quotient $\x{F}_S\to \x{H}$ where 
$\x{H}$ is the image of $\phi$. 

Since $\x{G}$ is flat over $T$, $j'_*\x{G}$ is flat over $S$. Moreover, since $R$ 
is a DVR, any subsheaf of $j'_*\x{G}$ is also flat over $R$ hence $\x{H}$ is flat 
over $S$ (note that $R$ is a PID so flatness is equivalent to torsion-freeness). Thus, the Hilbert polynomial of $\x{H}$ on the fibres of $X_S\to S$ 
is the same over the two points of $S$. Therefore, the quotient $\x{F}_S\to \x{H}$ 
is an element of $\cQuot^{\Phi,\x{L}}_{\x{F}/X/Y}(S)$. By construction, the pullback 
of $\x{F}_S\to \x{H}$ to $X_T$ coincides with $\x{F}_T\to \x{G}$. Moreover, 
the arguments in Step 6 show that $\x{F}_G\to \x{R}(-l)$ pulls back to 
$\x{F}_S\to \x{H}$.
Therefore, there is a unique morphism $S\to G_{\Phi}$ over $Y$ which restricts to the 
given morphism $T\to G_{\Phi}$. This proves that $G_{\Phi}$ is indeed proper 
hence projective over $Y$.\\
\end{proof}

\begin{lem}\label{l-Hilb-quot-functor-embedding}
Let $f\colon X\to Y$ be a projective morphism of Noetherian schemes, 
$\x{L}=\x{O}_X(1)$ a very ample invertible sheaf over $Y$, 
and $\Phi\in \Q[t]$ a polynomial. Let $\x{F}'\to \x{F}$ be a surjective 
morphism of coherent sheaves on $X$.
If $\cQuot^{\Phi,\x{L}}_{\x{F}'/X/Y}$ is represented 
by a scheme $\Quot^{\Phi,\x{L}}_{\x{F}'/X/Y}$, then $\cQuot^{\Phi,\x{L}}_{\x{F}/X/Y}$ 
is represented by a closed subscheme of $\Quot^{\Phi,\x{L}}_{\x{F}'/X/Y}$. 
\end{lem}
\begin{proof}
For ease of notation we let $Q':=\Quot^{\Phi,\x{L}}_{\x{F}'/X/Y}$. Let 
${\mathfrak{G}'}$ be the universal family on $X_{Q'}$ which comes with a 
surjective morphism $\x{F}'_{Q'}\to {\mathfrak{G}'}$ whose kernel we denote by 
$\x{L}$. The given morphism 
$\x{F}'\to \x{F}$ also gives a surjective morphism $\x{F}'_{{Q'}}\to \x{F}_{{Q'}}$ 
whose kernel we denote by $\x{N}$. Put $\x{R}=\x{F}'_{{Q'}}/(\x{L}+\x{N})$.
So, we have a commutative diagram 
$$
\xymatrix{
0\ar[rd] &                  &    0\ar[d]   &  &\\
 &         \x{L}+\x{N}\ar[rd]         &  \x{L}\ar[d] & &\\
0\ar[r] & \x{N}\ar[r] & \x{F}'_{Q'}\ar[r]\ar[d]\ar[rd] &  \x{F}_{{Q'}}\ar[r]\ar[d] & 0\\
 &               &   {\mathfrak{G}'} \ar[r]\ar[d] & \x{R} \ar[r]\ar[d]\ar[rd] & 0\\ 
   &            &     0     &  0 & 0
}
$$
Let $Q:=Q'_{\Phi}$ be the locally closed subscheme of $Q'$ corresponding to $\Phi$ given by the stratification of $Q'$ as in Theorem \ref{t-flat-Hilb-stratification} for the sheaf $\x{R}$. 
 
Now let $g\colon S\to Y$ be a morphism from a Noetherian scheme, 
$\x{F}_S\to \x{G}$ an element in $\cQuot^{\Phi,\x{L}}_{\x{F}/X/Y}(S)$. 
The surjection $\x{F}'_S\to \x{F}_S$ composed with  $\x{F}_S\to \x{G}$
naturally gives a quotient $\x{F}'_S\to \x{G}$ in $\cQuot^{\Phi,\x{L}}_{\x{F}'/X/Y}(S)$. 
Thus, we have a natural transformation 
$$
\cQuot^{\Phi,\x{L}}_{\x{F}/X/Y}\to \cQuot^{\Phi,\x{L}}_{\x{F}'/X/Y}
$$
Moreover, the quotient  $\x{F}'_S\to \x{G}$ uniquely determines 
a morphism $h\colon S\to Q'$ such that if $c\colon X_S\to X_{Q'}$ is the 
induced morphism, then the quotient $\x{F}'_{Q'}\to {\mathfrak{G}'}$ 
is pulled back to $\x{F}'_S\to \x{G}$ via $c$. The above diagram gives 
the commutative diagram 
$$
\xymatrix{
 &                  &    0\ar[d]   &  &\\
 &        c^*(\x{L}+\x{N})\ar[rd]         &  c^*\x{L}\ar[d] & &\\
 & c^*\x{N}\ar[r] & \x{F}'_{S}\ar[r]^\beta\ar[d]^\alpha\ar[rd] &  \x{F}_{S}\ar[r]\ar[d] & 0\\
 &               &   \x{G} \ar[r]^\gamma\ar[d] & c^*\x{R} \ar[r]\ar[d]\ar[rd] & 0\\ 
   &            &     0     &  0 & 0
}
$$
By construction, the morphism $\alpha$ factors through the morphism $\beta$ 
hence the image of $c^*\x{L}$ contains the image of $c^*\x{N}$ which implies that 
 the image of $c^*\x{L}$ is equal to the image of $c^*(\x{L}+\x{N})$. Thus,  
$\gamma$ is an isomorphism. In particular, this implies that $h\colon S\to Q'$ 
factors through $Q$. Therefore, the functor  $\cQuot^{\Phi,\x{L}}_{\x{F}/X/Y}$ is represented 
by the scheme $Q$ and the quotient $\x{F}_{Q'}\to \x{R}$ 
restricted to $X_Q$.

Finally, we need to prove that $Q$ is a closed subscheme of $Q'$. This can be done 
by proving that $Q$ is proper over $Y$ using the valuative criterion 
for properness similar to Step 7 of the proof of Theorem 
\ref{t-Hilb-quot-functor}.
\end{proof}


\section*{Exercises}
\begin{enumerate}
\item\label{exe-Hilb-polynomial} 
Let $X$ be a projective scheme over a field $k$, $\x{L}$ an 
invertible sheaf on $X$ and $\x{F}$ a coherent sheaf on $X$. Show that the function $\Phi$ on 
$\Z$ defined by $\Phi(l)=\mathcal{X}(X,\x{F}\otimes \x{L}^{\otimes l})$ is a 
 polynomial in $\Q[t]$ which is called the Hilbert polynomial of $\x{F}$ with respect to 
 $\x{L}$.\\

\item\label{exe-Hilb-funtoriality}  
Let $f\colon X\to Y$ be a projective morphism of Noetherian schemes, $\x{L}$ an 
invertible sheaf on $X$, $\x{F}$ a coherent sheaf on $X$ and $\x{F}\to \x{G}$ 
a coherent quotient such that $\x{G}$ is flat over $Y$ with Hilbert polynomial 
$\Phi$  with respect to $\x{L}$. Let $g\colon S\to Y$ be a morphism 
from a Noetherian scheme. Show that $\x{F}_S\to \x{G}_S$ is a 
coherent quotient with $\x{G}_S$ flat over $S$ having Hilbert polynomial $\Phi$.\\
 
\item\label{exe-Hilb-fibre-isom} 
Let $f\colon X\to Y$ be a projective morphism of Noetherian schemes and $\phi\colon \x{F}\to \x{G}$ 
a morphism of coherent sheaves on $X$ such that the induced morphism $\phi_y\colon \x{F}_y\to \x{G}_y$ 
on the fibre $X_y$ is an isomorphism, for every $y\in Y$. Show that $\phi$ is surjective. 
Moreover, if $\x{G}$ is flat over $Y$, show that $\phi$ is an isomorphism (Hint: use 
Exercise \ref{exe-noflat-base-change} of Chapter \ref{ch-flat} and the fact that if $\x{G}$ is flat over $Y$, then $f_*\x{G}(l)$ is locally free 
for every $l\gg 0$).\\

\item\label{exe-Hilb-hypersurface} 
Let $X=\PP^n_k=\Proj k[t_0,\dots,t_n]$ where $k$ is a field and let $\Phi(t)=\binom{n+t}{n}-\binom{n-d+t}{n}$.
Let $Z$ be a closed subscheme of $X$ with Hilbert polynomial $\Phi$ with respect to 
$\x{O}_X(1)$. Show that $Z$ is a hypersurface of degree $d$. 
\end{enumerate}


\appendix\chapter{\tt Castelnuovo-Mumford Regularity}

\begin{defn}
Let $X$ be a projective scheme over a field $k$, and let $\x{F}$ be a coherent 
sheaf on $X$. We say that $\x{F}$ is $m$-regular (with respect to a very ample 
invertible sheaf $\x{O}_X(1)$) if $H^p(X,\x{F}(m-p))=0$ for every 
$p>0$.\\
\end{defn}

Obviously, a fixed $\x{F}$ is $m$-regular for every $m\gg 0$. The following theorems are 
used in the chapter on Hilbert and Quotient schemes. They are proved using easy elementary 
arguments, for example, see Nitsure [\ref{F},\S 5].\\

\begin{thm}\label{t-regularity-1}
Let $X$ be a projective scheme over a field $k$, and let $\x{F}$ be a coherent 
sheaf on $X$ which is $m$-regular. Then, we have the following properties:

$(i)$ $\x{F}$ is $l$-regular for any $l\ge m$,

$(ii)$ the natural map $H^0(X,\x{O}_X(1))\otimes_k H^0(X,\x{F}(l))\to H^0(X,\x{F}(l+1))$ is surjective for any $l\ge m$,

$(iii)$ $\x{F}(l)$ is generated by global sections for any $l\ge m$.\\
\end{thm}

Recall that a polynomial $\Theta(t)\in \Q[t]$ is called numerical if $\Theta(l)\in \Z$ 
for any $l\in \Z$. It is well-known that such a polynomial can be written as 
$\Theta(t)=\sum_{i=0}^n a_i \binom{t}{i}$ for certain integers $a_0,\dots,a_n$ 
(cf. Hartshorne [\ref{Hartshorne}, I, 7.3]).\\

\begin{thm}\label{t-regularity-2}
Let $n,r$ be non-negative integers. Then, there is a polynomial $\Psi_{n,r}\in \Z[t_0,\dots,t_n]$ 
with the following property: if 

$\bullet$ $k$ is a field,

$\bullet$ $\x{F}$ is a coherent sheaf on $X=\PP^n_k$ which is a subsheaf of $\x{O}_X^r$,

$\bullet$ the Hilbert polynomial of $\x{F}$ is $\Phi(t)=\sum_{i=0}^n a_i \binom{t}{i}$,\\\\ 
then $\x{F}$ is $\Psi_{n,r}(a_0,\dots,a_n)$-regular. 
\end{thm}




\end{document}